\documentclass[11pt]{amsart}
\usepackage{amsfonts, amsmath, amssymb, amsthm, color,
  bookmark,graphicx,subfig}

\usepackage{hyperref}
\hypersetup{
  pdfborder={0 0 0},
  colorlinks   = true,
  urlcolor     = blue,
  linkcolor    = blue,
  citecolor   = red
}

\allowdisplaybreaks

\usepackage[centertags]{amsmath}  
\usepackage{graphicx}
\usepackage{amssymb}
\usepackage{indentfirst}
\usepackage{amsmath}
\usepackage{amsthm}  
\usepackage{color}
\usepackage{tikz-cd}
\usepackage{relsize}
\usepackage{tikz, changepage}
\usepackage{mathabx}

            \newtheorem{theorem}{Theorem}[section]                                          
           \theoremstyle{plain}
            \newtheorem{lemma}[theorem]{Lemma}
            \newtheorem{corollary}[theorem]{Corollary}                                          
            \newtheorem{proposition}[theorem]{Proposition}                                      
            \newtheorem{claim}{Claim}[theorem]
            \newtheorem{question}[theorem]{Question}
            \theoremstyle{definition}
            \newtheorem{definition}[theorem]{Definition}
            \newtheorem{example}[theorem]{Example}
            \newtheorem{remark}[theorem]{Remark}
            \newtheorem{notation}[theorem]{Notation}

 \newtheorem{theo}{Theorem}

\renewcommand{\ll }{\langle\hspace{-.7mm}\langle }
\newcommand{\rr }{\rangle\hspace{-.7mm}\rangle }

\newcommand{\m}{\ll\mathcal{N}\rr}
\newcommand{\bg}{\overline{G}}
\newcommand{\bhl}{\overline{H}_{\lambda}}
\newcommand{\pl}{\prod_{\lambda\in\Lambda}}

\newcommand{\ihg}{\mathrm{Ind}^{\overline{G}}_{\overline{H}_{\lambda}}}
\newcommand{\chg}{\mathrm{CoInd}^{\overline{G}}_{\overline{H}_{\lambda}}}
\newcommand{\cm}{A}

\newcommand{\bsl}{\bigoplus_{\lambda\in\Lambda}}

\newcommand{\bh}{\overline{H}}

\newcommand{\R}{\mathbb R}
\newcommand{\Z}{\mathbb Z}
\newcommand{\h}{\mathcal H}
\newcommand{\M}{\overline{M}}
\newcommand{\oh}{\overline{\mathcal H}}

\DeclareMathOperator{\card}{card}
\DeclareMathOperator{\cd}{cd}
\DeclareMathOperator{\diam}{diam}
\DeclareMathOperator{\lab}{\textbf{Lab}}
\renewcommand{\subset}{\subseteq}
\DeclareMathOperator{\tor}{Tor}
\DeclareMathOperator{\ext}{Ext}
\DeclareMathOperator{\hd}{hd}

\begin{document}

\title[Cohomology of group theoretic Dehn fillings II]{Cohomology of group theoretic Dehn fillings II}

\author{Nansen Petrosyan}
\address{School of Mathematics, University of Southampton, Southampton SO17~1BJ, UK}
\email{n.petrosyan@soton.ac.uk}
\thanks{}
\author{Bin Sun}
\address{Department of Mathematics, Michigan State University, East Lansing, MI, 48824, USA}
\email{sunbin1@msu.edu}

\begin{abstract}
We study the cohomology of group theoretic Dehn fillings.  Applying the Cohen-Lyndon property for  sufficiently deep Dehn fillings of hyperbolically embedded subgroups $H\hookrightarrow_h G$, obtained by the second named author in \cite{sun2018cohomologyi}, we derive a spectral sequence that computes the cohomology of the corresponding Dehn filling quotients $\bg$. As an application, we establish an isomorphism between the relative cohomology of the group pair $(G, H)$ and its sufficiently deep Dehn filling quotient pair $(\bg, \bh)$. This allows us to generalise the results of Fujiwara and Manning on simplicial volume of Dehn fillings of hyperbolic manifolds to Dehn fillings of Poincar\'e duality pairs. 

We also strengthen the results of Olshanskii \cite{Olsh}, Dahmani-Guirardel-Osin  \cite{dahmani2017hyperbolically} and Hull \cite{hull2013small} on SQ-universality and common quotients of acylindrically hyperbolic groups by adding cohomological finiteness conditions. We apply these results to obtain hyperbolic and acylindrically hyperbolic quotients with special properties.

\end{abstract}


\maketitle

\section{Introduction}
\subsection{Dehn surgery in $3$-manifolds.} 
In the late 1970's, Thurston dramatically changed the study of $3$-manifolds by introducing his Geometrization Conjecture. As supporting evidence, Thurston proved that many non-Haken $3$-manifolds satisfy the conjecture \cite{thurston1983three}, using the notion of a \textit{Dehn surgery}, which is a two-step procedure of modifying a $3$-manifold by first cutting off a solid torus and then gluing the torus back in a different way. Another motivation of Dehn surgery comes from the Lickorish-Wallace theorem, which states that every closed connected orientable $3$-manifold can be constructed from the $3$-sphere by using finitely many Dehn surgeries.


The second step of the surgery, called \textit{Dehn filling}, starts with a $3$-manifold $M$ with toral boundary and constructs a new manifold by gluing a solid torus to $M$ by identifying their boundaries. Topologically distinct ways of gluing a solid torus are parametrized by free homotopy classes of essential simple closed curves on $\partial M$ (the image of the meridian circle of the solid torus under the identification), called \textit{slopes}. For a slope $s$, the new manifold constructed by the corresponding Dehn filling is denoted by $M_s$. A celebrated result of Thurston asserts that most Dehn fillings preserve hyperbolicity. 
\begin{theorem}[{Thurston \cite{thurston1983three}}]\label{thm. thurson hyperbolic Dehn filling}
Let $M$ be a compact orientable $3$-manifold with boundary a torus, and with interior admitting a complete finite volume hyperbolic structure. Then for all but finitely many slopes $s$ on $\partial M$, $M_s$ admits a hyperbolic structure. 
\end{theorem}

\subsection{Group theoretic Dehn fillings.} There is an analogous construction in group theory, called \textit{(group theoretic) Dehn filling}, which can be formalized as follows. Given a group $G$ with a subgroup $H$ and a normal subgroup $N$ of $H$, the Dehn filling associated with the triple $(G,H,N)$ is the quotient $G/\ll N \rr$, where $\ll N \rr$ is the normal closure of $N$ in $G$.

The relation between these two versions of Dehn fillings can be seen as follows: under the assumptions of Theorem \ref{thm. thurson hyperbolic Dehn filling}, the natural homomorphism $\pi_1(\partial M)\rightarrow \pi_1(M)$ is injective and thus $\pi_1(\partial M)$ can be thought of as a subgroup of $\pi_1(M)$. Let $G=\pi_1(M)$ and $H=\pi_1(\partial M)$. Every slope $s$ on $\partial M$ generates a normal subgroup $N_s\lhd H$. As $s$ is the image of the meridian circle of the solid torus, which bounds a disc, we have $\pi_1(M_s)=G/\ll N_s \rr$.

Dehn filling is a fundamental tool in group theory. It appears, for instance, in the solution of the Virtual Haken Conjecture \cite{agol2013virtual}, the study of the Farrell-Jones Conjecture and the isomorphism problem of relatively hyperbolic groups \cite{antolin2017farrell,dahmani2018recognizing}, and the construction of purely pseudo-Anosov normal subgroups of mapping class groups \cite{dahmani2017hyperbolically}. Other applications of Dehn fillings can be found for example in \cite{agol2016alternate,groves2016boundaries}.

Algebraic analogs of Theorem \ref{thm. thurson hyperbolic Dehn filling} can be proved for groups satisfying certain negative curvature conditions. The first result of this kind was for relatively hyperbolic groups by Osin \cite{osin2007peripheral} and independently, by Groves-Manning \cite{groves2008dehn}. Later, Dahmani-Guirardel-Osin \cite{dahmani2017hyperbolically} introduced a generalization of relative hyperbolicity based on the notion of a hyperbolically embedded subgroup and proved a generalization of the main results of \cite{osin2007peripheral,groves2008dehn}. We postpone the definition and motivation of hyperbolically embedded subgroups until Section \ref{sec. he} and only discuss several examples for the moment. The reader is referred to the survey \cite{osin2017groups} for other examples. We use $H\hookrightarrow_h G$ to indicate that $H$ is a hyperbolically embedded subgroup of $G$.

\begin{example}\label{eg. 1}
If a group $G$ is hyperbolic relative to its subgroup $H$, then $H\hookrightarrow_h G$ \cite[Proposition 2.4]{dahmani2017hyperbolically}. In particular, if $M$ is a compact orientable manifold with one boundary component and $M\smallsetminus\partial M$ admits a complete finite volume hyperbolic structure,
then $\pi_1(\partial M)\hookrightarrow_h \pi_1(M)$ \cite{bowditch2012relatively,farb1998relatively}. 
\end{example}

\begin{example}\label{eg. 2}
Another typical example arises if a group $G$ acts on a Gromov hyperbolic space $S$ acylindrically by isometries and $g\in G$ is a loxodromic element. Then there exists a maximal virtually-cyclic subgroup $E(g)\leqslant G$ containing $g$ such that $E(g)\hookrightarrow_h G$ \cite[Corollary 2.9]{dahmani2017hyperbolically}. In particular, if $G$ is a word-hyperbolic group (resp. mapping class group of a finite type surface \cite[Theorem 2.19]{dahmani2017hyperbolically}, outer automorphism group of a finite rank free group \cite[Theorem 2.20]{dahmani2017hyperbolically}) and $g$ is an infinite order (resp. a pseudo-Anosov, a fully irreducible) element, then $E(g)\hookrightarrow_h G$.
\end{example}

The following is a group theoretic analog of Thurston's Theorem \ref{thm. thurson hyperbolic Dehn filling} due to Dahmani-Guirardel-Osin.

\begin{theorem}[{Dahmani-Guirardel-Osin \cite{dahmani2017hyperbolically}}]\label{thm. simple Dehn filling}
Let $G$ be a group with a subgroup $H\hookrightarrow_h G$. Then there exists a finite set $\mathcal{F}\subset H\smallsetminus\{1\}$ such that if $N\lhd H$ and $N\cap \mathcal{F}=\emptyset$, then the natural homomorphism $H/N\rightarrow G/\ll N \rr $ maps $H/N$ injectively onto a hyperbolically embedded subgroup of $G/\ll N \rr$.
\end{theorem}

In the setting of Theorem \ref{thm. thurson hyperbolic Dehn filling}, we have $\pi_1(\partial M)\hookrightarrow_h M$ by Example \ref{eg. 1}. Theorem \ref{thm. simple Dehn filling} then implies that for all but finitely many slopes $s$ on $\partial M$, we have $\pi_1(\partial M)/N_s\hookrightarrow_h \pi_1(M_s)$. A deeper investigation, using more precise versions of Theorem \ref{thm. simple Dehn filling} (e.g., \cite[Theorem 1.1]{osin2007peripheral}, \cite[Theorem 7.2]{groves2008dehn} and \cite[Theorem 2.27]{dahmani2017hyperbolically}), shows that $\pi_1(M_s)$ is word-hyperboic \cite[Corollary 1.2]{osin2007peripheral} and one-ended \cite[Corollary 1.11]{groves2018dehn}. The Geometrization Conjecture, proved by Perelman, then implies that $M_s$ admits a hyperbolic structure.

\subsection{Motivation:  cohomology of  Dehn fillings.} Theorem \ref{thm. thurson hyperbolic Dehn filling} asserts that $M_s$ is often hyperbolic and thus its universal cover is $\mathbb{H}^3$. It follows that  the cohomology of $\pi_1(M_s)$ can be understood by studying the cohomology of $M_s$  and the  action of $\pi_1(M_s)$ on  $\mathbb{H}^3$. It is therefore natural to investigate the cohomology of $G/\ll N \rr$ in the more general setting of Theorem \ref{thm. simple Dehn filling} where one may look for a similar geometric footing.

\begin{question}\label{question. main question}
For a group $G$ with a subgroup $H\hookrightarrow_h G$ and a normal subgroup $N\lhd H$, what can be said about the cohomology of $G/\ll N \rr$?
\end{question}
The main goal of this series of two papers is to address this question and to illustrate the implications of the results in this direction. 

We should point out that even though  we consider the general case of hyperbolically embedded subgroups $H\hookrightarrow_h G$, all of our results are new in the special case when $G$ is hyperbolic relative to $H$. 


\subsection*{Acknowledgements} Most of this work was done when the second named author was a graduate student at Vanderbilt University. He would like to thank his supervisor, Denis Osin, for the valuable discussions. This paper would not have been written without the help of Osin. The second named author would also like to thank Anna Marie Bohmann for the helpful comments and thank Ian Leary for answering his question and the suggestion of references. The second named author received funding from the European Research Council (ERC) under the European Union’s Horizon 2020 research and innovation programme (Grant agreement No. 850930) and an AMS--Simons Travel Grant (Grant agreement No. IP00672308). The first named 
author would like to thank Clara L\"{o}h and Kevin Li for helpful comments and suggestions. Finally, the authors are thankful to the referees for the careful reading of the paper and many detailed comments which have improved both its  content and exposition.

\section{Statements of main results}

\subsection{Cohomological properties of Dehn fillings.} To simplify the statement, we introduce the following terminology and notation.

\begin{definition}\label{intro: def. sufficiently deep} Let $G$ be a group and $H$ a subgroup of $G$. We say that a property $\mathcal P$ \textit{holds for sufficiently deep normal subgroups} if there is a finite set $\mathcal{F}\subset H\smallsetminus\{1\}$ such that $\mathcal P$ holds whenever $N$ is a normal subgroup of $H$ and $N\cap \mathcal{F}=\emptyset$.

\noindent Given a normal subgroup $N$ of $H$, let $\bg = G/\ll N \rr$ and $\bh = H/N$.
\end{definition}

Theorem \ref{thm. simple Dehn filling} can now be restated as: let $G$ be a group with a subgroup $H\hookrightarrow_h G$. Then for sufficiently deep $N\lhd H$, the natural homomorphism $\bh\rightarrow\bg$ maps $\bh$ injectively onto a hyperbolically embedded subgroup of $\bg$.

The following is a summary of our main results on cohomological properties of Dehn fillings.

\begin{theo}\label{thm. main}
Let $G$ be a group with a subgroup $H\hookrightarrow_h G$. Then the following hold for all sufficiently deep $N\lhd H$ and all $\bg$-modules $A$.
\begin{enumerate}
    \item[(i)] There is a spectral sequence
    \[
    E_2^{p,q}(A)=
    \begin{cases}
     H^p(\bh; H^q(N;A))&\text{for } q>0\\
     H^p(\bg;A)&\text{for } q=0
    \end{cases}
    \Rightarrow  H^{p+q}(G;A),
    \]
    where the action of $G$ on $A$ factors through $\bg$. In particular, the action of $N$ on $A$ fixes $A$ pointwise.\\

    \item[(ii)] (Algebraic Excision) For all $n\geqslant 0$, there is a natural isomorphism induced by the quotient maps $G\rightarrow \bg$ and $H\rightarrow\bh$,
    \[ H^n(\bg,\bh;A)\cong  H^n(G,H;A).\]
    \end{enumerate}
\end{theo}

The spectral sequence together with the algebraic excision have the following application. 
    
\begin{corollary}\label{intro_cor:finiteness}
    Let $G$ be a group with a subgroup $H\hookrightarrow_h G$. Then for all sufficiently deep $N\lhd H$ and all $\bg$-modules $A$, we have
    \begin{enumerate}
   \item[(i)]  For all $n\geqslant \cd(H)+2$,
   \[ H^n(\bg;A)\cong  H^n(G;A)\oplus  H^n(\bh;A).\]
    
 \item[(ii)] $\cd(\bg)\leqslant \max\{\cd(G),\cd(H)+1,\cd(\bh)\}.$\\
    
  \item[(iii)] If $G$ is of type $FP_n$ for some $n\in\mathbb{N}^+\cup\{\infty\}$ (resp. $FP$), then $\bg$ is of type $FP_n$ (resp. $FP$) if and only if $\bh$ is of type $FP_n$ (resp. $FP$).
\end{enumerate}
\end{corollary}

\noindent Here,  $\cd(G)$ stands for the \textit{cohomological dimension} of a group $G$, and $\mathbb{N}^+$ stands for the set of positive integers. This notion, \textit{property} $FP_n$, \textit{property} $FP$, and \textit{relative cohomology of group pairs}  are reviewed in Section \ref{sec. relative group cohomology}.

Analogous homological statements to Theorem \ref{thm. main} and Corollary \ref{intro_cor:finiteness} also hold. The spectral sequence in Theorem \ref{thm. main} (i) is a refinement of the classical Lyndon-Hochschild-Serre spectral sequence \cite{hochschild1953cohomology,lyndon1948cohomology} in the setting of Dehn fillings.

Let $M$ and $M_s$ be as in Theorem \ref{thm. thurson hyperbolic Dehn filling} and let $G=\pi_1(M),H=\pi_1(\partial M), N=\langle s \rangle$. Then Theorem \ref{thm. main} (ii) is an immediate consequence of excision. Therefore, Theorem \ref{thm. main} (ii) can be thought of as an algebraic analog of excision and computes $H^n(\bg,\bh;A)$ from $H^n(G,H;A)$. In the special case where $A=\Z \bg$, a spectral sequence to compute $H^\ast(\bg,\bh;\Z \bg)$ with $H^\ast(G,H;\Z G)$ has been developed by \cite{wang2018spectral}. We remark that the spectral sequence in \cite{wang2018spectral} is different from ours.

Instead of proving Theorem \ref{thm. main}, we will prove more general results (see Section \ref{sec. proof of main results}), which cover the case of \textit{a hyperbolically embedded family of subgroups} and will be useful in the proof of Theorems \ref{intro:thm:pd}, \ref{intro:thm:norm}, and \ref{thm. common quotient of acylindrically hyperbolic groups} below, and also cover the case of \textit{weakly hyperbolically embedded subgroups} and can be applied to graph of groups (see Example \ref{eg. product hyperbolically embedded} (e)).

Corollary \ref{intro_cor:finiteness} was recently used by Arenas \cite{Arenas}, who gave a variation of the Rips Construction that produces cubulated hyperbolic groups  of cohomological dimension bounded above by the cohomological dimension of an associated compact special cube complex.




\subsection{Poincar\'{e} duality and simplicial volume} Simplicial volume is a homotopy invariant of oriented closed connected manifolds  \cite{frigerio2017,  loh2011}. It was introduced by Gromov in his seminal paper \cite{gromov1982}. 

In \cite{fujmann2010}, Fujiwara and Manning generalize Gromov-Thurston's  $2\pi$-theorem  \cite{Bleiler_Hodgson} on Dehn fillings of $3$-manifolds  to higher dimensional finite volume hyperbolic manifolds $M^n$ with toral cusps. The resulting $2\pi$-fillings are  pseudomanifolds and are manifolds if and only if all the filling cores have dimension exactly $n-2$. They prove that every $2\pi$-filling admits a complete locally CAT(0) metric. In \cite{fujmann2011},  they show that the simplicial volume of every $2\pi$-filling is positive and bounded above by the relative simplicial volume $||\M, \partial \M||$ of $M$. 

In certain cases, simplicial volume can also be defined in more abstract setting of groups and topological spaces. The following  two theorems can be seen as natural generalizations and group theoretic analogs of the results of Fujiwara and Manning. 

\begin{theo}\label{intro:thm:pd} Let $G$ be a group and $\h=\{H_i\}_{i=1}^m$ a collection of subgroups such that $\h \hookrightarrow_h G$. Suppose, for some integer $2\leqslant n$,  $(G, \h )$ is a PD$(n)$-pair and that there are sufficiently deep  $\{\Z\cong N_i\lhd H_i\}_{i=1}^m$,  such that every $\bh_i$ is a PD$(n-2)$-group. Then $\bg$ is a PD$(n)$-group.
\end{theo}

The hypothesis on subgroups in Theorem \ref{intro:thm:pd} are for example  satisfied when each $H_i$ is a torsion-free nilpotent group with center of rank at least $2$ (see e.g.~ Lemma \ref{lem:nil}).

\begin{theo}\label{intro:thm:norm}
Let $G$ be a group and $\h=\{H_i\}_{i=1}^m$ a collection of subgroups such that $\h \hookrightarrow_h G$. Suppose, for some integer $n\geq 2$,  $(G, \h )$ is a PD$(n)$-pair and for a sufficiently deep  $\{N_i\lhd H_i\}_{i=1}^m$,  $\mathrm{cd}(\bh_i)\leq n-2$ for each $1\leq i\leq m$.  Then, $\cd(\bg)=n$, $H_n(\bg; \Z)=\Z$. In addition,
\begin{itemize}
\item[(i)]  if the group $\bh_i$ is amenable  for each $1\leq i\leq m$, then $||\bg ||\leq ||G, \h ||$, where $||\cdot||$ denotes the simplicial volume (see Definition \ref{def:sv});
\item[(ii)] if $G$ is hyperbolic relative to $\h$, then $||\bg|| >0$.
\end{itemize}
\end{theo}

We should point out that conjecturally by Kropholler, amenable groups of finite cohomological dimension such as $\bh_i$ are virtually solvable \cite[Question, p.~2]{deg2016}. Also, by Poincar\'{e} duality group or pair we will always mean orientable ones.

Theorem \ref{intro:thm:norm} can for example be applied when $G$ is the fundamental group of a Riemannian manifold with a complete pinched negative sectional curvature and finite volume. We give one such application which generalises Theorem 1.5 of \cite{fujmann2011}. 

\begin{corollary}[Corollary \ref{cor:pinched}]\label{intro:cor:pinched}
Let $\M$ be a compact oriented $n$-manifold with nilmanifold boundary components  such that the centre of the fundamental group of each boundary component is of rank at least $2$. Suppose the interior of $\M$ admits a Riemannian metric with a complete pinched negative sectional curvature and finite volume.   If $M_T$ is a sufficiently deep Dehn filling manifold of $\M$, then $M_T$ is a closed oriented aspherical $n$-manifold with
$$0<||M_T||\leq ||\M, \partial \M||.$$ 
\end{corollary}

For the definition of $M_T$ we refer to Definition \ref{sufficient_deep_mfld}. It is worth mentioning that the asphericity of $M_T$ above is not immediate and follows from the Cohen-Lyndon property. 



\subsection{Quotients of acylindrically hyperbolic groups.} The notion of an acylindrically hyperbolic group was introduced by Osin \cite{osin2016acylindrically} as a generalization of non-elementary hyperbolic and non-elementary relatively hyperbolic groups. Examples of acylindrically hyperbolic groups can be found in many classes of groups that attracted group theorists for years, e.g., mapping class groups of surfaces \cite{masur1999geometry,bowditch2008tight}, outer automorphism groups of free groups \cite{bestvina2009hyperbolic}, small cancellation groups \cite{Gruber2018infinitely}, convergence groups \cite{sun2017dynamical}, the Cremona group (see \cite{dahmani2017hyperbolically} and references therein; the main contribution towards showing the acylindrical hyperbolicity of the Cremona group is due to \cite{Cantat2013normal}), and tame automorphism groups of $3$-dimensional affine spaces \cite{lamy2019acylindrical}. We refer to \cite{osin2017groups} for details and other examples.

Every acylindrically hyperbolic group $G$ contains hyperbolically embedded subgroups \cite[Theorem 6.14]{dahmani2017hyperbolically} and Dehn fillings can often be applied to construct useful quotients of $G$. We use Theorem \ref{thm. main} to study homological properties of those quotients.


Recall that every acylindrically hyperbolic group $G$ has a unique maximal finite normal subgroup denoted by $K(G)$ \cite[Theorem 6.14]{dahmani2017hyperbolically}.

\begin{theo}\label{cd SQ-universality}
Let $G$ be an acylindrically hyperbolic group, and let $C$ be any countable group. Then $C$ embeds into a quotient $\bg$ of $G/K(G)$ (in particular, $\bg$ is a quotient of $G$) such that 
\begin{enumerate}
\item[(i)] $\bg$ is acylindrically hyperbolic;\\
\item[(ii)] if $C$ is finitely generated, then $C \hookrightarrow_h \bg$;\\
\item[(iii)] if $G$ and $C$ are torsion-free, then so is $\bg$;\\
\item[(iv)] for all $n\geqslant 3$ and every $\bg$-module $A$, we have
\[ H^n(\bg;A)\cong  H^n(G/K(G);A)\oplus  H^n(C;A),\]
where the action of $G/K(G)$ (resp.~$C$) on $A$ is induced by the quotient map $G/K(G)\rightarrow\bg$ (resp.~the embedding $C\hookrightarrow \bg$);\\
\item[(v)] $\cd(\bg)\leqslant \max\{\cd(G),\cd(C)\}$;\\
\item[(vi)] if $C$ is finitely generated and $G$ is of type $FP_n$ for some $n\in\mathbb{N}^+\cup\{\infty\}$ (resp.~$FP$), then $\bg$ is of type $FP_n$ (resp.~$FP$) if and only if $C$ is of type $FP_n$ (resp.~$FP$).
\end{enumerate}
\end{theo}

As an application, we strengthen SQ-universality of hyperbolic groups given by Olshanskii \cite{Olsh} and independently by Delzant \cite{Del} by adding cohomological conditions. 

\begin{corollary}[Corollary \ref{cor. hyperbolic}]\label{intro:cor. hyperbolic}
Let $G$ be a non-elementary hyperbolic group and $C$ any hyperbolic group. Then there is a hyperbolic quotient $\bg$ of $G/K(G)$ (in particular, $\bg$ is a quotient of $G$) such that $C$ embeds into $\bg$ and the following hold.
\begin{enumerate}
\item[(i)] For all $n\geqslant 3$ and every $\bg$-module $A$, we have
\[ H^n(\bg;A)\cong  H^n(G/K(G);A)\oplus  H^n(C;A),\]
where the action of $G/K(G)$ (resp. $C$) on $A$ is induced by the quotient map $G/K(G)\rightarrow\bg$ (resp. the embedding $C\hookrightarrow \bg$).\\
\item[(ii)] $\cd(\bg)\leqslant \max\{\cd(G),\cd(C)\}$.\\
\end{enumerate}
\end{corollary}

\begin{theo}\label{thm. common quotient of acylindrically hyperbolic groups}
Let $G_1$ and $G_2$ be finitely generated acylindrically hyperbolic groups. Then there exists a common quotient $G$ of $G_1/K(G_1)$ and $G_2/K(G_2)$ (in particular, $G$ is a common quotient of $G_1$ and $G_2$) such that
\begin{enumerate}
    \item[(i)] $G$ is acylindrically hyperbolic;\\
    \item[(ii)] for all $n\geqslant 3$ and every $G$-module $A$, we have
    \[ H^n(G;A)\cong  H^n(G_1/K(G_1);A)\oplus  H^n(G_2/K(G_2);A),\]
    where the actions of $G_1/K(G_1)$ and $G_2/K(G_2)$ on $A$ factor through $G$;\\
    \item[(iii)] $\cd(G)\leqslant \max\{\cd(G_1),\cd(G_2)\};$\\
    \item[(iv)] if $G_1$ and $G_2$ are of type $FP_n$ for some $n\in\{2,3,...,\infty\}$ (resp. $FP$), then so is $G$.
\end{enumerate}
\end{theo}

 Homological analogs of Theorem \ref{cd SQ-universality} (iv), (v) and Theorem \ref{thm. common quotient of acylindrically hyperbolic groups} (ii), (iii) also hold (see Remarks \ref{rm. homological analog of D (v)} and \ref{rm. hd of common quotient}).
  Theorem \ref{cd SQ-universality} (i), (ii) are proved in \cite[Theorem 2.33]{dahmani2017hyperbolically} and Theorem \ref{thm. common quotient of acylindrically hyperbolic groups} (i) is proved in \cite[Corollary 7.4]{hull2013small}. The benefit of Theorems \ref{cd SQ-universality} and \ref{thm. common quotient of acylindrically hyperbolic groups} is that they allow one to control the cohomology of the resulting  acylindrically hyperbolic quotients. As we illustrate in Section \ref{sec. applications}, this facilitates the constructions of various acylindrically hyperbolic groups satisfying certain cohomological properties. We list some of them below.
 
  \begin{corollary}[Corollary \ref{cor. true F}]\label{intro:F}
   Every torsion-free acylindrically hyperbolic group $G$ of type $FP_{\infty}$ has a torsion-free acylindrically hyperbolic quotient $\bg$ of type $FP_{\infty}$ which contains the Thompson group $F$. 
\end{corollary}

The existence of a pair $H<G$, where $G$ is hyperbolic and $H$ is of type $F$ but not hyperbolic, was a well-known open problem, raised in particular by Bestvina \cite[Question 2.1]{bestvina?questions}, Brady \cite[Question 7.2]{brady1999branched}, Bridson \cite[Question 4.1]{bridson?problems}, and Jankiewicz-Norin-Wise \cite[Section 7]{jankiewicz2021virtually}. The first example of such a pair was constructed recently by Italiano-Martelli-Migliorini \cite[Corollary 2]{italiano2021hyperbolic}. By Corollary \ref{intro:cor. hyperbolic}, we obtain:

\begin{corollary}[Corollary \ref{cor. non-hyperbolic}]
Let $n\geq 5$ be an integer.
Every non-elementary hyperbolic group $G$ with $\cd(G)\leq n$ has a hyperbolic  quotient $\bg$ with $\cd(\bg)=n$ such that $\bg$ contains the Italiano-Martelli-Migliorini group. In particular, there is a type $F$ non-hyperbolic subgroup $H<\bg$.
\end{corollary}

The next two corollaries strengthen a result of \cite[Corllary 1.7]{hull2013small} stating that every acylindrically hyperbolic group has an acylindrically hyperbolic quotient with Kazhdan's Property (T).

\begin{corollary}[Corollary \ref{cor:T}]\label{intro:cor:T}
Every acylindrically hyperbolic group $G$ of type $FP_n$ for some $n\in\{2,3,...,\infty\}$ (resp.~$FP$) has an acylindrically hyperbolic quotient $\bg$  of type $FP_n$ (resp.~$FP$) with Kazhdan's Property (T) such that $\cd(\bg)\leq\max\{\cd(G), 2\}$.
\end{corollary}

\begin{corollary}[Corollary \ref{cor. fp_n but not fp_n+1}]\label{intro:cor: fp_n but not fp_n+1}
Let $G$ be any acylindrically hyperbolic group of type $FP_{\infty}$. Then $G$ has a family of acylindrically hyperbolic quotients $\{G_k\}^{\infty}_{k=2}$ such that for each $k$, $G_k$ has Kazhdan's Property (T), is of type $FP_{k-1}$ but not of type $FP_{k}$.
\end{corollary}

In particular, since mapping class groups of surfaces of finite type, outer automorphism groups of free groups of finite rank and most $3$-manifold groups are acylindrically hyperbolic and of type $FP$, they all exhibit such quotients.

\subsection{A few words on the proofs of the main results.} The first step of the proof of Theorem \ref{thm. main} is to establish, under the assumptions of the theorem, the  isomorphism
\begin{equation}\label{eq. induced module}
     H^n(\ll N \rr;A)\cong \mathrm{CoInd}^{\bg}_{\bh} H^n(N;A)
\end{equation}
for all $n>0$. Here,  $\mathrm{CoInd}^{\bg}_{\bh}$ stands for the  co-induction from $\mathbb{Z}\bh$ to $\mathbb{Z}\bg$. The Lyndon-Hochschild-Serre spectral sequence associated to the triple $(G,\ll N \rr,A)$ takes the form
\begin{equation*}\label{eq. LHS}
    E^{p,q}_2(A)= H^p(\bg; H^q(\ll N \rr;A))\Rightarrow  H^{p+q}(G;A).
\end{equation*}
Shapiro's lemma together with \eqref{eq. induced module}  yields Theorem \ref{thm. main} (i). In fact, we will establish a more precise result which involves a morphism between the Lyndon-Hochschild-Serre spectral sequences associated to $(G,\ll N \rr,A)$ and $(H,N,A)$. Part (ii) of Theorem \ref{thm. main} will be proved by an inspection of this morphism.

To  prove Theorem \ref{intro:thm:pd}, we analyse the spectral sequence of Theorem \ref{thm. main} (i) and apply part (ii) and Corollary \ref{intro_cor:finiteness} (iii). To show that $\bg$ is  a Poincar\'{e} duality group, we use Johnson-Wall characterisation \cite{johnwall72}; namely, a group $\Gamma$ is a Poincar\'{e} duality group of dimension $n$ if and only if $\Gamma$ is of type $FP$, $H^i(\Gamma, \Z\Gamma)=0$ for $i\ne n$ and $H^n(\Gamma, \Z\Gamma)=\Z$. 

The algebraic excision Theorem \ref{thm. main} (ii) is again key in proving Theorem \ref{intro:thm:norm}. Since all $\bh_i$ are amenable, the natural map in bounded cohomology $H^n_b(\bg, \overline{\h} ; \R)\to H^n_b(\bg ; \R)$ is an isometric  isomorphism. The duality pairing between bounded cohomology and ordinary homology leads to the inequality $||\bg ||\leq ||G, \h ||$. When $G$ is hyperbolic relative to $\h$, then a sufficiently deep Dehn filling quotient $\bg$ is hyperbolic relative to $\oh$ \cite[Theorem 1.1]{osin2007peripheral}. This allows us to show that $||\bg ||>0$.

The proof of Theorem \ref{cd SQ-universality} is a modification of the proof of \cite[Theorem 2.33]{dahmani2017hyperbolically}. Given any acylindrically hyperbolic group $G$, one can find a non-cyclic free group $F\hookrightarrow_h G_0=G/K(G)$. For any countable group $C$ and any finite subset $\mathcal{F}\subset F\smallsetminus\{1\}$, we will use small cancellation theory to construct a normal subgroup $N\lhd F$ such that $N\cap\mathcal{F}=\emptyset$, $C$ embeds into $F/N$, and $F/N$ has the desired cohomological properties. Theorem \ref{thm. simple Dehn filling} then implies that $C$ embeds into $G_0/\ll N \rr$ and Theorem \ref{thm. main} and Corollary \ref{intro_cor:finiteness} applied to $N\lhd F\hookrightarrow_h G_0$ yields the desired cohomological results.

The proof of \cite[Corollary 7.4]{hull2013small} uses small cancellation theory instead of Dehn filling. In order to apply our main result, we carry out an alternative approach. Given finitely generated acylindrically hyperbolic groups $G_1$ and $G_2$, we construct subgroups $H_1,H_2<\widetilde{G}=G'_1\ast G'_2$, where $G'_1=G_1/K(G_1)$ and $G'_2=G_2/K(G_2)$, such that the family $\{H_1,H_2\}$ \textit{hyperbolically embeds} into $\widetilde{G}$. For any finite sets $\{\mathcal{F}_i\subset H_i\smallsetminus\{1\}\}_{i=1,2}$, we will use small cancellation theory to construct normal subgroups $\{N_i\lhd H_i\}_{i=1,2}$ such that $N_i\cap\mathcal{F}_i=\emptyset$ and $N_1$ (resp. $N_2$) identifies a finite set of generators of $G'_2$ (resp. $G'_1$) with certain elements of $G'_1$ (resp. $G'_2$). The quotient $G=\widetilde{G}/\ll N_1\cup N_2 \rr$ is thus a common quotient of $G'_1$ and $G'_2$. Theorem \ref{thm. common quotient of acylindrically hyperbolic groups} is then proved by applying general versions of Theorems \ref{thm. simple Dehn filling} and \ref{thm. main}. The main difficulty of this argument is the construction of $H_1$ and $H_2$, which is presented in Section \ref{sec. construct h.e. subgroups}, using a technical tool provided by \cite{dahmani2017hyperbolically} (see also \cite{osin2007peripheral}) called \textit{isolated components}.

\subsection{Organization of the paper.} We will start with preliminaries in Section \ref{sec.convention}, recalling basic definitions of group cohomology, the notion of (weakly) hyperbolically embedded subgroups, isolated components, acylindrically hyperbolic groups, and the structural result of \cite{sun2018cohomologyi} called the \textit{Cohen-Lyndon property}. The proof of (the general version of) Theorem \ref{thm. main} and Corollary \ref{intro_cor:finiteness} is given in Section \ref{sec. proof of main results}. Theorems  \ref{intro:thm:pd} and \ref{intro:thm:norm} are proved in Sections \ref{sec:duality} and \ref{sec:sv}, respectively. Theorems \ref{cd SQ-universality} and \ref{thm. common quotient of acylindrically hyperbolic groups} are proved in Section \ref{quotients} with applications given in Section \ref{sec. applications}.

\section{Preliminaries}\label{sec.convention}

\subsection{Direct sum and product of spectral sequences}

We briefly recall a property concerning convergence of the direct sum and direct product of a (possibly infinite) family of spectral sequences of ring modules. For detailed introduction to spectral sequences we refer to \cite{Weibel}, \cite{brown1982cohomology} and \cite{Cartan_Eilenberg}.

Let $R$ be a unital ring, and let $E^{2,\lambda}_{p,q}\Rightarrow H^{\lambda}_{p+q},\lambda\in\Lambda$ be a family of convergent homological spectral sequences of $R$-modules such that $E^{2,\lambda}_{p,q}=0$ if either $p$ or $q$ is less than $0$.

\begin{lemma}\label{lem. sum and product}
    We have
    \[\bigoplus_{\lambda\in\Lambda}E^{2,\lambda}_{p,q}\Rightarrow \bigoplus_{\lambda\in\Lambda}H^{\lambda}_{p+q}, \;\;\; \prod_{\lambda\in\Lambda}E^{2,\lambda}_{p,q}\Rightarrow \prod_{\lambda\in\Lambda}H^{\lambda}_{p+q}.\]
    Moreover, the same statement holds for cohomological spectral sequences as well.
\end{lemma}

\begin{proof}
    For simplicity, denote $\bigoplus_{\lambda\in\Lambda}E^{2,\lambda}_{p,q}$ as $E^2_{p,q}$. First note the the $E^{\infty}$-terms satsify
    \[E^{\infty}_{p,q}=\bigoplus_{\lambda\in\Lambda}E^{\infty,\lambda}_{p,q}.\]
    Indeed, fix $p,q$ for the moment. We have
    \[E^{\infty}_{p,q}=E^{p+q+1}_{p,q}=\bigoplus_{\lambda\in\Lambda}E^{p+q+1,\lambda}_{p,q}=\bigoplus_{\lambda\in\Lambda}E^{\infty,\lambda}_{p,q}.\]

    Next, let us recall that the convergence $E^{2,\lambda}_{p,q}\Rightarrow H^{\lambda}_{p+q}$ means that for each $\lambda$ and each $H^{\lambda}_{p+q}$, there is a filtration
    \[0=F_{-1}H^{\lambda}_{p+q}<F_0H^{\lambda}_{p+q}<...<F_{p+q}H^{\lambda}_{p+q}=H^{\lambda}_{p+q}\]
    such that for all $p,q$ we have a short exact sequence
    \[0\rightarrow F_{p-1}H^{\lambda}_{p+q}\rightarrow F_pH^{\lambda}_{p+q}\rightarrow E^{\infty,\lambda}_{p,q}\rightarrow 0.\]

    Consider the following filtration of $\bigoplus_{\lambda\in\Lambda}H^{\lambda}_{p+q}$:
    \[0=\bigoplus_{\lambda\in\Lambda}F_{-1}H^{\lambda}_{p+q}<\bigoplus_{\lambda\in\Lambda}F_0H^{\lambda}_{p+q}<...<\bigoplus_{\lambda\in\Lambda}F_{p+q}H^{\lambda}_{p+q}=\bigoplus_{\lambda\in\Lambda}H^{\lambda}_{p+q}.\]
    Then we have short exact sequences
    \[0\rightarrow \bigoplus_{\lambda\in\Lambda}F_{p-1}H^{\lambda}_{p+q}\rightarrow \bigoplus_{\lambda\in\Lambda}F_pH^{\lambda}_{p+q}\rightarrow \bigoplus_{\lambda\in\Lambda}E^{\infty,\lambda}_{p,q}\rightarrow 0\]
    by taking the direct sum of the above short exact sequences, as the direct sum functor is exact \cite[Theorem 2.6.15]{Weibel}.

    The proof for direct product is similar. We only point out that the direct product functor is exact for $R$-modules, as the category of $R$-modules satisfies Axiom (AB4$^\ast$) of \cite[Section 1.5]{grothendieck1957sur}. Also, the proof for cohomological spectral sequences is almost identical.
\end{proof}

\subsection{Cohomology of groups}\label{sec. relative group cohomology}
Let $G$ be a group. Recall that the \textit{homological} and \textit{cohomological dimension} of $G$ can be defined by
\begin{align*}
    &\hd(G)=\sup\{n\in\mathbb{N}\mid  H_{n}(G,A)\neq 0\text{ for some }\mathbb Z G\text{-module }A\},\\
    &\cd(G)=\sup\{n\in\mathbb{N}\mid  H^{n}(G,A)\neq 0\text{ for some }\mathbb Z G\text{-module }A\},
\end{align*}
respectively.

$G$ is of \textit{type} $FP_n$ for some $n\in\mathbb{N}^+\cup\{\infty\}$ if there is a projective resolution \[\cdot\cdot\cdot\rightarrow P_2\rightarrow P_1\rightarrow P_0\rightarrow \mathbb{Z}\]
over $\mathbb{Z}G$ such that $P_k$ are finitely generated $G$-modules for all $k\leqslant n$. $G$ is of type $FP$ if $\cd(G)<\infty$ and $G$ is of type $FP_{\infty}$.

Property $FP_n$ can be characterized by the cohomology functor. The following will be useful in the proof of Theorem \ref{property FP}.

\begin{theorem}[{Bieri \cite[Theorem 1.3]{bieri1981homological}, Brown \cite[Theorem 2]{brown1975homological}}]\label{homological criterion}
For a group $G$, the following are equivalent.
\begin{enumerate}
    \item[(a)] $G$ is of type $FP_n$ for some $n\in\mathbb{N}^+\cup\{\infty\}$.
    \item[(b)] For every $k\leqslant n$ and every direct system $\{A_i\}_{i\in I}$ of $G$-modules such that $\varinjlim A_i=0$, we have $\varinjlim  H^k(G;A_i)=0$.
\end{enumerate}
\end{theorem}

Given a family $\{H_{\lambda}\}_{\lambda\in\Lambda}$ of subgroups of $G$, \cite{bieri1978relative} defined the relative (co)homology of the group pair $(G,\{H_{\lambda}\}_{\lambda\in\Lambda})$. We briefly recall the definition. Let $\Delta$ be the kernel of the augmentation $\bsl\mathbb{Z}[G/H_{\lambda}]\twoheadrightarrow\mathbb{Z}$ which sends every left $H_{\lambda}$-coset to $1$. By definition,
\begin{equation}\label{eq. relative group cohomology}
     H_n(G,\{H_{\lambda}\}_{\lambda\in\Lambda};A)=\tor^G_{n-1}(\Delta,A),~~ H^n(G,\{H_{\lambda}\}_{\lambda\in\Lambda};A)=\ext^{n-1}_G(\Delta,A)
\end{equation}
for any $G$-module $A$. The dimension shift in the above definition ensures a long exact sequence between the absolute and relative (co)homology (see \cite[Proposition 1.1]{bieri1978relative}).


\subsection{(Weakly) hyperbolically embedded subgroups}\label{sec. he}
The notion of (weakly) hyperbolically embedded subgroups was introduced by \cite{dahmani2017hyperbolically}, which is our main reference for Sections \ref{sec. he}, \ref{sec.ic}, and \ref{sec.ah}. We first recall the definition and present some examples. The motivation will be discussed afterwards.

Let $G$ be a group, $\{H_\lambda\}_{\lambda\in\Lambda}$ a family of subgroups of $G$, $X$ a subset of $G$ such that $G$ is generated by $X$ together with the union of all $H_{\lambda}$ (in which case $X$ is called a \textit{relative generating set} of $G$ with respect to $\{H_{\lambda}\}_{\lambda\in\Lambda}$), and $\mathcal{H}=\bigsqcup_{\lambda\in\Lambda}H_{\lambda}$. Consider the Cayley graph $\Gamma(G,X\sqcup \mathcal{H})$. Note that, for $\lambda\in\Lambda$ there is a natural embedding $\Gamma(H_{\lambda},H_{\lambda})\hookrightarrow \Gamma(G,X\sqcup\mathcal{H})$ whose image is the subgraph of $\Gamma(G,X\sqcup\mathcal{H})$ with vertices and edges labeled by elements of $H_{\lambda}$.

\begin{remark}
We do allow $X\cap H_{\lambda}\neq \emptyset$ and $H_{\lambda}\cap H_{\mu}\neq\{1\}$ for distinct $\lambda,\mu\in\Lambda$, in which case there will be multiple edges between some pairs of vertices of $\Gamma(G,X\sqcup \mathcal{H})$.
\end{remark}

For $\lambda\in\Lambda$, an edge path in $\Gamma(G,X\sqcup \mathcal{H})$ between vertices of $\Gamma(H_{\lambda},H_{\lambda})$ is called $H_{\lambda}$\textit{-admissible} if it does not contain any edge of $\Gamma(H_{\lambda},H_{\lambda})$. Note that an $H_{\lambda}$-admissible path is allowed to pass through vertices of $\Gamma(H_{\lambda},H_{\lambda})$. 

For example, consider the simple case where $\{H_{\lambda}\}_{\lambda\in\Lambda}=\{H\}$ consists of only one subgroup $H\leqslant G$. The Cayley graph $\Gamma(G,X\sqcup H)$ is displayed in Figure \ref{fig. hyperbolically embedded subgroup}. The blue path is admissible. The red path is an edge from $1$ to $h$ labeled by $h\in H$, and thus is inadmissible. If $h$ happens to be an element of $X$, i.e., there exists $x\in X$ with $x=h$, and the red path were labeled by $x$ instead of $h$, then the red path would be admissible.

\begin{definition}\label{def. relative metric}
For every pair of elements $h,k\in H_{\lambda}$, let $\widehat{d}_{\lambda}(h,k)\in[0,\infty]$ be the length of a shortest $H_{\lambda}$-admissible path connecting the vertices labeled by $h$ and $k$. If no such path exists, set $\widehat{d}_{\lambda}(h,k)=\infty$. The laws of summation on $[0,\infty)$ extend naturally to $[0,\infty]$ and it is easy to verify that $\widehat{d}_{\lambda}:H_{\lambda}\times H_{\lambda}\rightarrow [0,+\infty]$ defines a metric on $H_{\lambda}$, which is called the \textit{relative metric on} $H_{\lambda}$ \textit{with respect to} $X$.
\end{definition}

\begin{definition}\label{hyperbolically embedded}
We say that the family $\{H_\lambda\}_{\lambda\in\Lambda}$ \textit{weakly hyperbolically embeds into} $(G,X)$ (denoted by $\{H_\lambda\}_{\lambda\in\Lambda}\hookrightarrow_{wh}(G,X)$) if $G$ is generated by the set $X$ together with union of all $H_{\lambda},\lambda\in\Lambda$, and the Cayley graph $\Gamma(G,X\sqcup \mathcal{H})$ is a Gromov hyperbolic space.

If $\{H_\lambda\}_{\lambda\in\Lambda}\hookrightarrow_{wh}(G,X)$ and for each $\lambda\in\Lambda$, the metric space $(H_{\lambda},\widehat{d}_{\lambda})$ is proper, i.e., every ball of finite radius contains only finitely many elements, then $\{H_\lambda\}_{\lambda\in\Lambda}$ \textit{hyperbolically embeds into} $(G,X)$ (denoted by $\{H_\lambda\}_{\lambda\in\Lambda}\hookrightarrow_h (G,X)$). If in addition, $X$ and $\Lambda$ are finite, then we say that $G$ is \textit{hyperbolic relative} to $\{H_\lambda\}_{\lambda\in\Lambda}$. 

Further, we say that the family $\{H_\lambda\}_{\lambda\in\Lambda}$ \textit{hyperbolically embeds into} $G$, denoted by $\{H_\lambda\}_{\lambda\in\Lambda}\hookrightarrow_h G$, if there exists some subset $X\subset G$ such that $\{H_\lambda\}_{\lambda\in\Lambda}\hookrightarrow_h(G,X)$.
\end{definition}

\begin{notation}
In case $\{H_\lambda\}_{\lambda\in\Lambda}=\{H\}$ is a singleton, we will drop the braces and write $H\hookrightarrow_{wh} (G,X)$ and $H\hookrightarrow_h G$.
\end{notation}

We refer to Figure \ref{fig. hyperbolically embedded subgroup} for an illustration of the situation $H\hookrightarrow_h G$. The grey discs represent cosets of $H$ in $G$. The black edges are labeled by elements of $X$. The edges and discs appear in a tree-like pattern as $\Gamma(G,X\sqcup H)$ is Gromov hyperbolic.


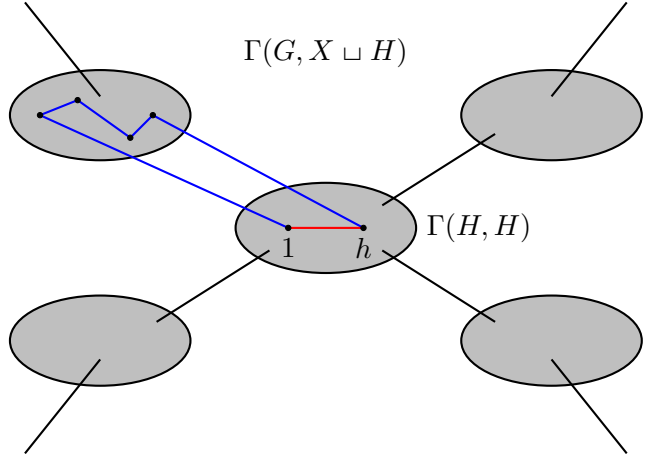
\begin{figure}[ht!]
    \centering
    \begin{tikzpicture}
        \filldraw[color=black!100, fill=gray!50, thick] (-3,1.5) ellipse (1.2 and 0.6);
        \filldraw[color=black!100, fill=gray!50, thick] (3,1.5) ellipse (1.2 and 0.6);
        \filldraw[color=black!100, fill=gray!50, thick] (-3,-1.5) ellipse (1.2 and 0.6);
        \filldraw[color=black!100, fill=gray!50, thick] (3,-1.5) ellipse (1.2 and 0.6);
        \filldraw[color=black!100, fill=gray!50, thick] (0,0) ellipse (1.2 and 0.6);
        \draw[thick] (1.2,0) circle (0pt) node[anchor=west] {$\Gamma(H,H)$};
        \draw[blue, thick] (-0.5,0.0) -- (-3.8,1.5);
        \draw[blue, thick] (0.5,0.0) -- (-2.3,1.5);
        \draw[blue, thick] (-3.8,1.5) -- (-3.3,1.7);
         \draw[blue, thick] (-3.3,1.7) -- (-2.6,1.2);
          \draw[blue, thick] (-2.6,1.2) -- (-2.3,1.5);
         \draw[red, thick] (-0.5,0.0) -- (0.5,0.0);
        \draw[thick] (-3,1.75) -- (-4, 3);
        \draw[thick] (3,1.75) -- (4, 3);
        \draw[thick] (-3,-1.75) -- (-4, -3);
        \draw[thick] (3,-1.75) -- (4, -3);
        \draw[thick] (-0.75,-0.3) -- (-2.25,-1.25);
        \draw[thick] (0.75,0.3) -- (2.25,1.25);
        \draw[thick] (0.75,-0.3) -- (2.25,-1.25);
        \filldraw[black] (-3.8,1.5) circle (1pt);
        \filldraw[black] (-3.3,1.7) circle (1pt);
        \filldraw[black] (-2.6,1.2) circle (1pt);
        \filldraw[black] (-2.3,1.5) circle (1pt);
        \filldraw[black] (-0.5,0.0) circle (1pt) node[anchor=north] {$1$};
        \filldraw[black] (0.5,0.0) circle (1pt) node[anchor=north] {$h$};
        \draw[thick] (0,2) circle (0pt) node[anchor=south] {$\Gamma(G,X\sqcup H)$};
    \end{tikzpicture}
    \caption{Illustration of $H\hookrightarrow_h G$}
    \label{fig. hyperbolically embedded subgroup}
\end{figure}

\begin{remark}\label{rem. relative hyperbolic}
Notice that the above definition of relative hyperbolicity is not commonly used in literature. One of the most commonly used definitions for relative hyperbolicity which we  will use later is the following: a group $G$ is \textit{hyperbolic relative} to a family $\{H_{\lambda}\}_{\lambda\in\Lambda}$ of its subgroups if $G$ has a finite relative presentation with respect to $\{H_{\lambda}\}_{\lambda\in\Lambda}$ with a linear relative isoperimetric function (see \cite[Definition 3.6]{dahmani2017hyperbolically}). The equivalence of these two definitions is proved in \cite[Remark 4.41 and Theorem 4.42]{dahmani2017hyperbolically}.
\end{remark}

\begin{example}\label{eg. product hyperbolically embedded}
~
\begin{enumerate}
    \item[(a)] $H\hookrightarrow_{wh}(G,G)$ for every subgroup $H\leqslant G$.
    \item[(b)] $H\hookrightarrow_h(G,G)$ for every finite subgroup $H\leqslant G$.
    \item[(c)] $G\hookrightarrow_h(G,\emptyset)$.
    \item[(d)] If $G$ can be decomposed as a free product of its subgroups $\{G_{\lambda}\}_{\lambda\in\Lambda}$ (denoted by $G=\Asterisk_{\lambda\in\Lambda}G_{\lambda}$), then $\{G_{\lambda}\}_{\lambda\in\Lambda}\hookrightarrow_h (G,\emptyset)$ \cite[Example 4.12]{dahmani2017hyperbolically}.
    \item[(e)] More generally, suppose that $G=\pi_1(\mathcal{G})$, where $\mathcal{G}$ is a graph of groups. Let $\{G_v\}_{v\in V\mathcal{G}}$ be the collection of vertex subgroups and $\{G_e\}_{e\in E\mathcal{G}}$ the collection of edge subgroups. By \cite[Example 4.12]{dahmani2017hyperbolically}, $\{G_v\}_{v\in V\mathcal{G}}\hookrightarrow_{wh}(\pi_1(\mathcal{G}),X)$, where the subset $X\subset G$ consists of stable letters (i.e., there is a spanning tree $T\mathcal{G}$ of $\mathcal{G}$ such that $X$ consists of generators corresponding to the edges of $\mathcal{G}\smallsetminus T\mathcal{G}$).
\end{enumerate}
\end{example}

Recall that a group is \textit{word-hyperbolic} if it is finitely generated and for some (equivalently, any) finite generating set the corresponding Cayley graph is Gromov hyperbolic. The notion of weak hyperbolic embedding is thus an attempt to study possibly non-word-hyperbolic groups via Gromov hyperbolic spaces. Example \ref{eg. product hyperbolically embedded} (a) illustrates a triviality of this notion, in the sense that the weak hyperbolic embedding does not provide any information about the group $G$. Notice that in that example, the corresponding relative metric is bounded. One might therefore refine the notion by imposing additional conditions on the relative metric, for example, requiring local finiteness of the relative metric and obtaining the notion of hyperbolic embedding. We note that Examples \ref{eg. product hyperbolically embedded} (b) and (c) exhibit two kinds of trivialities of hyperbolic embedding, and a further refinement is given by the notion of an \textit{acylindrically hyperbolic group} (see Section \ref{sec.ah}).

The next lemma tells us how to find hyperbolically embedded subgroups.

\begin{lemma}[{Dahmani-Guirardel-Osin \cite[Lemma 4.21]{dahmani2017hyperbolically}}]\label{lem. criterion for he}
Suppose that $\card(\Lambda)<\infty$, $G$ acts on a Gromov hyperbolic space $(S,d)$ by isometries, and the following three conditions are satisfied, then $\{H_{\lambda}\}_{\lambda\in\Lambda}\hookrightarrow_h G$.
\begin{enumerate}
    \item[(C$_1$)] Every $H_{\lambda}$ acts on $S$ properly.
    \item[(C$_2$)] There exists $s\in S$ such that for every $\lambda\in\Lambda$, the $H_{\lambda}$-orbit $H_{\lambda}(s)$ of $s$ is quasi-convex in $S$.
    \item[(C$_3$)] For every $\epsilon>0$ and some $s\in S$, there exists $R>0$ such that the following holds. Suppose that for some $g\in G$ and $\lambda,\mu\in\Lambda$, we have
    $$\diam\left(H_{\mu}(s)\cap \left(gH_{\lambda}(s)\right)^{+\epsilon}\right)\geqslant R,$$
    then $\lambda=\mu$ and $g\in H_{\lambda}$, where $\left(g H_{\lambda}(s)\right)^{+\epsilon}$ denotes the $\epsilon$-neighborhood of $gH_{\lambda}(s)$ in $S$.
\end{enumerate}
\end{lemma}

The following proposition, which roughly says that being a hyperbolically embedded subgroup is a transitive property, will be used later.



\begin{proposition}[{Dahmani-Guirardel-Osin \cite[Proposition 4.35]{dahmani2017hyperbolically}}]\label{prop. iterative hyperbolically embedded}
Let $G$ be a group, let $\{H_{\lambda}\}_{\lambda\in\Lambda}$ be a finite family of subgroups of $G$, let $X\subset G$, and let $Y_{\lambda}\subset H_{\lambda}$ for every $\lambda\in\Lambda$. Suppose that $\{H_{\lambda}\}_{\lambda\in\Lambda}\hookrightarrow_h (G,X)$ and for every $\lambda\in\Lambda$, there is a family of subgroups $\{K_{\lambda,\mu}\}_{\mu\in M_{\lambda}}\hookrightarrow_h (H_{\lambda},Y_{\lambda})$. Then
$$\bigcup_{\lambda\in\Lambda}\{K_{\lambda,\mu}\}_{\mu\in M_{\lambda}}\hookrightarrow_h \left(G,X\cup\left(\bigcup_{\lambda\in\Lambda}Y_{\lambda}\right)\right).$$
\end{proposition}



The cohomological finiteness properties of a group $G$ are inherited by its hyperbolically embedded subgroups:

\begin{lemma}[{Dahmani-Guirardel-Osin \cite[Remark 4.26 and Corollary 4.32]{dahmani2017hyperbolically}}]\label{lem. he finiteness}
    Let $G$ be a group with a hyperbolically embedded family of subgroups $\{H_{\lambda}\}_{\lambda\in\Lambda}$. Suppose that $G$ is of type $FP_n$ for some $n\in \{2,3,...,\infty\}$. Then for every $\lambda$, the group $H_{\lambda}$ is of type $FP_n$.
\end{lemma}

\subsection{Isolated components}\label{sec.ic}
In the proof of Theorem \ref{thm. common quotient of acylindrically hyperbolic groups}, we need to construct specific hyperbolically embedded subgroups. A tool to do this is the notion of an isolated component, which was introduced by \cite{osin2007peripheral} for relatively hyperbolic groups and generalized to hyperbolically embedded subgroups by \cite{dahmani2017hyperbolically}. In this section, we recall the definition and collect several results.
 
We start with conventions. Let $G$ be a group and $X$ a generating set of $G$. Consider the Cayley graph $\Gamma(G,X)$. Let $p$ be a path in $\Gamma(G,X)$. The \textit{label} of $p$, denoted $\lab(p)$, is obtained by concatenating all labels of the edges in $p$ and is a word over $X$. The length $p$ is denoted by $\ell_X(p)$, and the initial (resp. terminal) vertex of $p$ is denoted by $p^-$ (resp. $p^+$).

Now suppose that $\{H_\lambda\}_{\lambda\in\Lambda}$ is a family of subgroups of $G$. Let $\mathcal{H}=\bigsqcup_{\lambda\in\Lambda}H_{\lambda}$, and let $X$ be a relative generating set of $G$ with respect to $\{H_\lambda\}_{\lambda\in\Lambda}$. For $\lambda\in\Lambda$, let $\widehat{d}_{\lambda}$ be the relative metric on $H_{\lambda}$ with respect to $X$. The following terminology goes back to \cite{osin2006relatively}.

\begin{definition}\label{connect}
Let $p$ be a path in $\Gamma(G,X\sqcup \mathcal{H})$. For every $\lambda\in\Lambda$, an $H_{\lambda}$\textit{-subpath} $q$ of $p$ is a nontrivial subpath of $p$ such that $\lab(q)$ is a word over the alphabet $H_{\lambda}$ (if $p$ is a cycle, we allow $q$ to be a subpath of some cyclic shift of $p$). An $H_{\lambda}$-subpath $q$ of $p$ is an $H_{\lambda}$-\textit{component} if $q$ is not properly contained in any other $H_{\lambda}$-subpath. Two $H_{\lambda}$-components $q_1$ and $q_2$ of $p$ are \textit{connected} if there exists a path $t$ in $\Gamma(G,X\sqcup \mathcal{H})$ such that $t$ connects a vertex of $q_1$ to a vertex of $q_2$, and that $\lab(t)$ is a letter from $H_{\lambda}$. An $H_{\lambda}$-component $q$ of $p$ is \textit{isolated} if it is not connected to any other $H_{\lambda}$-component of $p$.
\end{definition}

\begin{remark}
    The definition of connectedness in \cite{dahmani2017hyperbolically} is seemingly different from the version above: instead of requiring $\lab(t)$ to be a letter from $H_{\lambda}$, \cite{dahmani2017hyperbolically} only requires $\lab(t)$ to be a word over $H_{\lambda}$ for some $\lambda\in\Lambda$. However, as every element of $H_{\lambda}$ belongs to the generating set $X\sqcup \mathcal{H}$, if there is a path $t$ connecting a vertex of $q_1$ to a vertex of $q_2$ with $\lab(t)$ being a word over $H_{\lambda}$, then there is another path $t'$ connecting a vertex of $q_1$ to a vertex of $q_2$ with $\lab(t')$ being a single letter of $H_{\lambda}$ ($\lab(t')$ is the element of $H_{\lambda}$ represented by $\lab(t)$).
\end{remark}

Suppose that $q$ is an $H_{\lambda}$-component of a path $p\subset \Gamma(G,X\sqcup\mathcal{H})$. Then $q^-$ (resp. $q^+$) is labeled by an element $g\in G$ (resp $h\in G$) and we have $g^{-1}h\in H_{\lambda}$. In this case, let
$$\widehat{\ell}_{\lambda}(q)=\widehat{d}_{\lambda}(1,g^{-1}h).$$

A nice property of isolated components is that in a geodesic polygon $p$, the total $\widehat{\ell}$-length of isolated components of $p$ is bounded linearly by the number of sides of $p$. More precisely:

\begin{proposition}[{Dahmani-Guirardel-Osin \cite[Proposition 4.14]{dahmani2017hyperbolically} (see also \cite[Proposition 3.2]{osin2007peripheral})}]\label{lem. total length of i.c. in g.p.}
If $\{H_\lambda\}_{\lambda\in\Lambda}\hookrightarrow_{wh}(G,X)$, then there exists a number $D>0$ satisfying the following property: Let $p$ be an $n$-gon in $\Gamma(G,X\sqcup \mathcal{H})$ with geodesic sides $p_1,...,p_n$ and let $I$ be a subset of the set of sides of $p$ such that every side $p_i\in I$ is an isolated $H_{\lambda_i}$-component of $p$ for some $\lambda_i\in\Lambda$. Then
$$\sum_{p_i\in I}\widehat{\ell}_{\lambda_i}(p_i)\leqslant Dn.$$
\end{proposition}

In fact, the above proposition is the reason why certain properties (for example, Theorems \ref{thm. group theoretic Dehn filling} and \ref{thm. CL-property} below) hold for \textit{sufficiently deep} Dehn fillings (see Definition \ref{def. sufficiently deep for wh}).

The technical lemma below will be used in Section \ref{sec. construct h.e. subgroups} along with Lemma \ref{lem. criterion for he} to construct hyperbolically embedded subgroups.

\begin{lemma}[{Dahmani-Guirardel-Osin \cite[Lemma 4.21]{dahmani2017hyperbolically}}]\label{lem. consecutive components}
Suppose $\{H_{\lambda}\}_{\lambda\in\Lambda}\hookrightarrow_{wh}(G,X)$. Let $W$ be the set consisting of all words $w$ over $X\sqcup\mathcal{H}$ such that
\begin{enumerate}
    \item[(W1)] $w$ contains no subwords of type $xy$, where $x,y\in X$;
    \item[(W2)] if $w$ contains a letter $h\in H_{\lambda}$ for some $\lambda\in \Lambda$, then $\widehat{d}_{\lambda}(1,h)>50D$, where $D$ is given by Proposition \ref{lem. total length of i.c. in g.p.};
    \item[(W3)] if $h_1xh_2$ (resp. $h_1h_2$) is a subword of $w$, where $x\in X,h_1\in H_{\lambda},h_2\in H_{\mu}$, then either $\lambda\neq \mu$ or the element represented by $x$ in $G$ does not belong to $H_{\lambda}$ (resp. $\lambda\neq \mu$).
\end{enumerate}
Then the following hold.
\begin{enumerate}
    \item[(a)] Every path in the Cayley graph $\Gamma(G,X\sqcup \mathcal{H})$ labeled by a word from $W$ is a $(4,1)$-quasi-geodesic.
    \item[(b)] If $p$ is a path in $\Gamma(G,X\sqcup \mathcal{H})$ labeled by a word from $W$, then for every $\lambda\in\Lambda$, every $H_{\lambda}$-component of $p$ is isolated.
    \item[(c)] For every $\epsilon>0$, there exists $R>0$ satisfying the following condition. Let $p,q$ be two paths in $\Gamma(G,X\sqcup \mathcal{H})$ such that 
    $$\ell_{X\sqcup \mathcal{H}}(p)\geqslant R,~~~~\lab(p),\lab(q)\in W,$$
    and $p,q$ are oriented $\epsilon$-close, i.e.,
    $$\max\{d_{X\sqcup \mathcal{H}}(p^-,q^-),d_{X\sqcup \mathcal{H}}(p^+,q^+)\}\leqslant \epsilon,$$
    where $d_{X\sqcup \mathcal{H}}$ is the combinatorial metric of $\Gamma(G,X\sqcup \mathcal{H})$.
    Then there exist five consecutive components of $p$ which are respectively connected to five consecutive components of $q$. In other words,
    $$p=ra_1x_1a_2x_2a_3x_3a_4x_4a_5s,~~~~q=tb_1y_1b_2y_2b_3y_3b_4y_4b_5u,$$
    such that the following hold.
    \begin{enumerate}
        \item[(i)] $r$ (resp. $t$) is a subpath of $p$ (resp. $q$) whose label does not end with a letter from $\mathcal{H}$.
        \item[(ii)] $s$ (resp. $u$) is a subpaht of $p$ (resp. $q$) whose label does not start with a letter from $\mathcal{H}$.
        \item[(iii)] For $i=1,...,4$, $x_i$ and $y_i$ are either trivial subpaths or subpaths labeled by a letter over $X$; 
        \item[(iv)] For $i=1,...,5$, $a_i$ and $b_i$ are connected $H_{\lambda_i}$-components.
    \end{enumerate}
\end{enumerate}
\end{lemma}

\begin{remark}
Conclusion (b) of Lemma \ref{lem. consecutive components} is not stated in \cite[Lemma 4.21]{dahmani2017hyperbolically}, but it is proved in the second paragraph of the proof of \cite[Lemma 4.21]{dahmani2017hyperbolically}.
\end{remark}

\subsection{Acylindrical hyperbolicity}\label{sec.ah}

We notice that Examples \ref{eg. product hyperbolically embedded} (b) and (c) are two occasions where having a hyperbolically embedded subgroup does not provide any information about the group $G$. We also notice that in those two cases, the hyperbolically embedded subgroup is either finite or improper. It is therefore natural to look at the groups $G$ with a proper infinite hyperbolically embedded subgroup. By \cite[Theorem 7.19]{dahmani2017hyperbolically} and \cite[Theorem 1.2]{osin2016acylindrically}, this is equivalent to saying that $G$ is acylindrically hyperbolic.

\begin{definition}\label{def. ah}
A group $G$ is \textit{acylindrically hyperbolic} if $G$ admits a non-elementary acylindrical action on some Gromov hyperbolic space by isometries.
\end{definition}

For the definition an acylindrical action, the reader is referred to \cite{osin2017groups}. Intuitively, one can think of acylindricity as an analog of properness. An acylindrical action of a group $G$ is \textit{non-elementary} if its orbits are unbounded and $G$ is not virtually-cyclic \cite[Theorem 1.1]{osin2016acylindrically}.

Techniques of hyperbolically embedded subgroups are mainly applied to acylindrically hyperbolic groups, because every acylindrically hyperbolic group contains infinitely many hyperbolically embedded virtually free subgroups.

\begin{theorem}[{Dahmani-Guirardel-Osin \cite[Theorem 6.14]{dahmani2017hyperbolically}}]\label{thm. hyperbolically embedded virtually free subgroup}
Let $G$ be an acylindrically hyperbolic group. Then $G$ has a unique maximal finite normal subgroup, denoted by $K(G)$. Moreover, for every $n\in\mathbb{N}^+$, there exists a free group $F$ of rank $n$ such that $F\times K(G)\hookrightarrow_h G$.
\end{theorem}

\subsection{Sufficient deepness and Cohen-Lyndon triples}

One consequence of Proposition \ref{lem. total length of i.c. in g.p.} is that acylindrical hyperbolicity is preserved by Dehn fillings, provided that the Dehn fillings are \textit{sufficiently deep} and done on hyperbolically embedded subgroups (see Theorem \ref{thm. group theoretic Dehn filling}).

\begin{definition}\label{def. sufficiently deep for wh}
Let $G$ be a group with a family of subgroups $\{H_{\lambda}\}_{\lambda\in\Lambda}\hookrightarrow_{wh}(G,X)$ for some subset $X\subset G$. For every $\lambda\in\Lambda$, let $\widehat{d}_{\lambda}$ be the relative metric on $H_{\lambda}$ with respect to $X$. A property $\mathcal{P}$ \textit{holds for all sufficiently deep normal subgroups} if there exists a constant $C>0$ such that $\mathcal{P}$ holds for every family of normal subgroups $\{N_{\lambda}\lhd H_{\lambda}\}_{\lambda\in\Lambda}$ with $\widehat{d}_{\lambda}(1,n)>C$ for all $n\in N_{\lambda}\smallsetminus\{1\}$.
\end{definition}

\begin{example}
If $G=H_1\ast_A H_2$ is a free product with amalgamation, then $\{H_1,H_2\}\hookrightarrow_{wh}(G,\emptyset)$ by Example \ref{eg. product hyperbolically embedded} (e). For normal subgroups $\{N_i\lhd H_i\}_{i=1,2}$, consider the property

\begin{enumerate}
    \item[$\mathcal{P}$:] The quotient $G/\ll N_1\cup N_2\rr$ splits as an amalgamated free product, where $\ll N_1\cup N_2\rr$ denotes the normal closure of $N_1\cup N_2$ in $G$.
\end{enumerate}

Then $\mathcal{P}$ holds for all sufficiently deep normal subgroup $\{N_i\lhd H_i\}_{i=1,2}$, because $\mathcal{P}$ holds whenever $N_i\cap A=\{1\}$, which amounts to saying that $\widehat{d}_i(n)>1$ for all $n\in N_i\smallsetminus\{1\}$, where $\widehat{d}_i$ is the relative metric on $H_i$ corresponding to the weak hyperbolic embedding $\{H_1,H_2\}\hookrightarrow_{wh}(G,\emptyset)$.
\end{example}

\begin{remark} \label{rem. sufficiently deep for he}
In Definition \ref{def. sufficiently deep for wh}, if $\{H_{\lambda}\}_{\lambda\in\Lambda}\hookrightarrow_h(G,X)$, then the relative metrics $\widehat{d}_{\lambda}$ are locally finite. Thus,
$$\card\left(\{h\in H_{\lambda}\mid \widehat{d}_{\lambda}(1,h)\leqslant C\}\right)<\infty$$
for all $C>0$. Therefore:

Let $G$ be a group with a family of subgroups $\{H_{\lambda}\}_{\lambda\in\Lambda}\hookrightarrow_h G$. Suppose that a property $\mathcal{P}$ holds for all sufficiently deep normal subgroups $\{N_{\lambda}\lhd H_{\lambda}\}_{\lambda\in\Lambda}$. Then there exist finite sets $\{\mathcal{F}_{\lambda}\subset H_{\lambda}\smallsetminus\{1\}\}_{\lambda\in\Lambda}$ such that $\mathcal{P}$ holds whenever $N_{\lambda}\cap \mathcal{F}_{\lambda}=\emptyset$ for all $\lambda\in\Lambda$.

If in addition, $|\Lambda|<\infty$, then a property $\mathcal{P}$ holds for all sufficiently deep normal subgroups $\{N_{\lambda}\lhd H_{\lambda}\}_{\lambda\in\Lambda}$ if and only if there exists a finite set $\mathcal{F}\subset \left(\bigcup_{\lambda\in \Lambda}H_{\lambda}\right)\smallsetminus\{1\}$ such that $\mathcal{P}$ holds whenever $N_{\lambda}\cap \mathcal{F}=\emptyset$. In particular, if $G$ is a group with a hyperbolically embedded subgroup $H\hookrightarrow_h G$, then Definition \ref{intro: def. sufficiently deep} is a special case of Definition \ref{def. sufficiently deep for wh}.
\end{remark}

To simplify statements, we use the following notation.

\begin{notation}\label{definition of natural maps}
Let $G$ be a group and $S$ a subset of $G$. Then $\ll S \rr$ denotes the normal closure of $S$ in $G$. Suppose that $\{H_{\lambda}\}_{\lambda\in\Lambda}$ is a family of subgroups of $G$ and $\{N_{\lambda}\lhd H_{\lambda}\}_{\lambda\in\Lambda}$ is a family of normal subgroups. We call $(G,\{H_{\lambda}\}_{\lambda\in\Lambda},\{N_{\lambda}\}_{\lambda\in\Lambda})$ a \textit{group triple}. We also let
\[\mathcal{N}=\bigcup_{\lambda\in\Lambda}N_{\lambda},~~\bg=G/\m,~~\bhl=H_{\lambda}/N_{\lambda}.\]
\end{notation}

\begin{theorem}[{Dahmani-Guirardel-Osin \cite[Theorem 7.19]{dahmani2017hyperbolically}, Osin \cite[Theorem 1.2]{osin2016acylindrically}}]\label{thm. group theoretic Dehn filling}
Let $G$ be a group with a family of subgroups $\{H_{\lambda}\}_{\lambda\in\Lambda}\hookrightarrow_h G$. Then for all sufficiently deep normal subgroups $\{N_{\lambda}\lhd H_{\lambda}\}_{\lambda\in\Lambda}$, the natural homomorphism $\bhl\rightarrow\bg$ is injective for $\lambda\in\Lambda$ and we have $\{\bhl\}_{\lambda\in\Lambda}\hookrightarrow_h \bg$. Moreover, if for some $\lambda\in\Lambda$, $\card(\bhl)=\infty$ and $\bhl$ is a proper subgroup of $\bg$, then $\bg$ is acylindrically hyperbolic.
\end{theorem}

The normal subgroup $\m$ in the above theorem can be described more precisely: it has a particular free product structure.

\begin{definition}\label{CLp}
A group triple $(G,\{H_{\lambda}\}_{\lambda\in\Lambda},\{N_{\lambda}\}_{\lambda\in\Lambda})$ is called a \textit{Cohen-Lyndon triple} if there exist left transversals $T_{\lambda}$ of $H_{\lambda}\m$ in $G$ such that
$$\m =\Asterisk_{\lambda\in\Lambda,t\in T_{\lambda}} tN_{\lambda}t^{-1}.$$
\end{definition}

The above free product structure was first proved by Cohen-Lyndon \cite{cohen1963free} for free groups, hence the name. The following theorem was partially proved by \cite[Theorem 7.19]{dahmani2017hyperbolically}, which was later improved by \cite{sun2018cohomologyi}.

\begin{theorem}[{Sun \cite[Theorem 5.1]{sun2018cohomologyi}}]\label{thm. CL-property}
Let $G$ be a group with a family of subgroups $\{H_{\lambda}\}_{\lambda\in\Lambda}\hookrightarrow_{wh}(G,X)$ for some $X\subset G$. Then for all sufficiently deep normal subgroups $\{N_{\lambda}\lhd H_{\lambda}\}_{\lambda\in\Lambda}$, $(G,\{H_{\lambda}\}_{\lambda\in\Lambda},\{N_{\lambda}\}_{\lambda\in\Lambda})$ is a Cohen-Lyndon triple.
\end{theorem}


\section{Cohomology of Dehn fillings}\label{sec. proof of main results}

In this section, we prove Theorem \ref{thm. main}. To simplify the notation, we use \ref{definition of natural maps}. 
Recall, given a group $G$, a subgroup $H\leq G$ and a $\Z H$-module $M$, the {\it induced module} is 
$$\mathrm{Ind}_H^G M:= \Z G\otimes_{\Z H} M$$
and the {\it co-induced module} is 
$$\mathrm{CoInd}_H^G M:= \mathrm{Hom}_{\Z H}(\Z G, M).$$
The left action of $G$ on $\Z G$ induces a $\Z G$-module structure on both the induced and the co-induced modules \cite[\S III.5]{brown1982cohomology}. 

Throughout, we adopt the standard convention that given $\Z G$-modules $M$ and $N$, then $M\otimes N := M\otimes_{\Z} N$ and $M\otimes_{G} N := M\otimes_{\Z G} N$. We repeatedly make use of the following tenor identities which are a consequence of the associativity of the tensor products \cite[\S III.3 and III.5]{brown1982cohomology}.
\begin{lemma}[Tensor product identities]\label{tensor_identities} Let $G$ be a group, let $H$ be a subgroup of $G$, and let $K$ be a normal subgroup of $G$. For any $G$-modules $M$, $N$ and $H$-module $L$, we have
\begin{enumerate}
\item[(i)] $M\otimes \Z[G/H] \cong \mathrm{Ind}_H^G \mathrm{Res}_H^G M$ as $G$-modules,\\
\item[(ii)] $M\otimes_G N \cong (M\otimes_K N)\otimes_{G/K} \Z$,\\
\item[(iii)] $(L\otimes_H \Z G)\otimes_G N\cong L\otimes_H N$.
\end{enumerate}
\end{lemma}

\begin{proposition}\label{prop. induce module}
Let $(G,\{H_{\lambda}\}_{\lambda\in\Lambda},\{N_{\lambda}\}_{\lambda\in\Lambda})$ be a Cohen-Lyndon triple. Then for any $G$-module $A$ and $q>1$ (also for any $\overline{G}$-module $A$ and $q>0$) there are isomorphisms
\begin{align*}
    & H_q(\m;A)\cong \bsl \ihg  H_q(N_{\lambda};A),\\
    & H^q(\m;A)\cong \pl\chg H^q(N_{\lambda};A)
\end{align*}
induced by the inclusions $N_{\lambda}\hookrightarrow \m$.
\end{proposition}

We remark that an easy consequence of $(G,\{H_{\lambda}\}_{\lambda\in\Lambda},\{N_{\lambda}\}_{\lambda\in\Lambda})$ being a Cohen-Lyndon triple is that the natural maps $\bhl\rightarrow \bg$ are injective \cite[Lemma 6.4]{sun2018cohomologyi}, and thus it makes sense to consider the (co)inductions $\ihg$ and $\chg$.

\begin{proof}
We will prove the homological version. The proof of the cohomological version will be analogous.

Let $\{T_{\lambda}\}_{\lambda\in\Lambda}$ be the family of left transversals associated to the Cohen-Lyndon triple $(G,\{H_{\lambda}\}_{\lambda\in\Lambda},\{N_{\lambda}\}_{\lambda\in\Lambda})$, given by Definition \ref{CLp}. Then $\m=\Asterisk_{\lambda\in\Lambda,t\in T_{\lambda}}tN_{\lambda}t^{-1}$. Consider the Bass-Serre tree $X$ associated to this free product decomposition with vertex set $V$ and edge set $E$. Viewing $X$ as a $1$-dimensional CW-complex, we get a short exact sequence of $\m$-modules
\[0\rightarrow\mathbb{Z}[E]\xrightarrow{\partial_1}\mathbb{Z}[V]\xrightarrow{\partial_0}\mathbb{Z}\rightarrow 0,\]
where $C_1(X)= {\Z}[E]$,  $C_0(X)= {\Z}[V]$, and $\partial_*$ is the usual boundary map. 

Since the terms in the above exact sequence are $\Z$-free modules, tensoring it with $A$ preserves exactness \cite[\S V.6, p.~278-279]{Bredon} and we obtain the short exact sequence
\begin{equation}\label{eq. short exact inter}
    0\rightarrow  \Z[E]\otimes A\rightarrow \Z[V]\otimes A\rightarrow A \rightarrow 0.
\end{equation}
\noindent Let $F_i\twoheadrightarrow\mathbb{Z}$ be a free $\mathbb{Z}G$-resolution of $\mathbb{Z}$. Tensoring the above sequence by $F_i\otimes_{\m}-$, we obtain a short exact sequence
\begin{equation}\label{eq. short exact (2)}
    0\rightarrow F_i\otimes_{\m}(\mathbb{Z}[E]\otimes A)\rightarrow F_i\otimes_{\m}(\mathbb{Z}[V]\otimes A)\rightarrow F_i\otimes_{\m}A \rightarrow 0.
\end{equation}

Note that there is an isomorphism of $\m$-modules
\[\mathbb{Z}[V]\cong \bigoplus_{\lambda\in\Lambda,t\in T_{\lambda}} \mathbb{Z}[\m/tN_{\lambda}t^{-1}].\]

The short exact sequence \eqref{eq. short exact (2)} then transforms to a short exact sequence of chain complexes
\begin{equation}\label{eq. short exact (4)}
    0\rightarrow F_i\otimes_{\m}(\mathbb{Z}[E] \otimes A)\rightarrow \bigoplus_{\lambda\in\Lambda,t\in T_{\lambda}}F_i\otimes_{tN_{\lambda}t^{-1}}A \xrightarrow[]{\theta} F_i\otimes_{\m}A\rightarrow 0.
\end{equation}

The long exact sequence corresponding to \eqref{eq. short exact (4)} yields
\begin{align*}
    &\cdots \to  H_q({\m}; \mathbb{Z}[E]\otimes A)\to  \bigoplus_{\lambda\in\Lambda, t\in T_{\lambda}}  H_q({tN_{\lambda}t^{-1}} ; A)\xrightarrow{\theta_{\ast}} H_q(\m; A)\to \\
    &\cdots \to \mathbb{Z}\otimes_{\m} (\mathbb{Z}[E]\otimes A)\xrightarrow[]{f} \mathbb{Z}\otimes_{\m} (\mathbb{Z}[V]\otimes A)\to \mathbb{Z}\otimes_{\m} A\to 0.
\end{align*}

Since $\mathbb{Z}[E]$ is a free $\m$-module, $ H_q(\m;\mathbb{Z}[E]\otimes A)=0$ for $q>0$ and $f$ is injective if $\m$ acts trivially on $A$. This yields an isomorphism
\begin{equation}\label{eq. short exact (5)}
    \theta_{\ast}:\bigoplus_{\lambda\in\Lambda,t\in T_{\lambda}} H_q(tN_{\lambda}t^{-1};A)\rightarrow H_q(\m;A)
\end{equation}
for $q>1$ and also for $q=1$ if $\m$ acts trivially on $A$. The action of $\bg$ on $ H_q(\m;A)$ and the isomorphism $\theta_{\ast}$ endow $\bigoplus_{\lambda\in\Lambda,t\in T_{\lambda}} H_q(tN_{\lambda}t^{-1};A)$ with a $\bg$-action. Next, we will show that this action is a direct sum of permutation actions.

The $\bg$-action on $ H_q(\m;A)$ comes from the $G$-action on $A$ and conjugation of $G$ on $\m$ which is induced by the diagonal $G$-action on $F_{\ast}\otimes_{\m}A$. More explicitly, each $g\in G$ acts as
\[\tau_g:F_i\otimes_{\m} A\rightarrow F_i\otimes_{\m}A,~ [x,a]\mapsto [gx,ga].\]

For each $\lambda\in \Lambda$ and $t\in T_{\lambda}$, the group $tH_{\lambda}t^{-1}$ acts on $F_i\otimes_{tN_{\lambda}t^{-1}}A$ by
\[[x,a]\mapsto[gx,ga],~\forall g\in tH_{\lambda}t^{-1}.\]
This gives us a $tH_{\lambda}t^{-1}$-equivariant restriction of $\theta$
\[\theta:F_i\otimes_{tN_{\lambda}t^{-1}}A\rightarrow F_i\otimes_{\m}A,\]
which in turn induces the $tH_{\lambda}t^{-1}$-equivariant inclusion
\[\theta_{\ast}:  H_q(tN_{\lambda}t^{-1};A)\hookrightarrow  H_q(\m;A).\]

Now, for each $\lambda\in\Lambda$ and $t\in T_{\lambda}$, the action of $t$ on $ H_q(\m;A)$ induces a map between two summands on the left-hand side of the isomorphism \eqref{eq. short exact (5)}. To see this, note that the map defined by
\[\sigma_t:F_i\otimes_{N_{\lambda}}A\rightarrow F_i\otimes_{tN_{\lambda}t^{-1}}A,\]
\[[x,a]\mapsto [tx,ta],\]
satisfies $\theta\circ\sigma_t=\tau_t\circ \theta$. In homology, we then have $(\sigma_t)_{\ast}=\theta^{-1}_{\ast}\circ (\tau_t)_{\ast}\circ \theta_{\ast}$ as claimed.

We have shown that for each $\lambda\in\Lambda$, the $G$-action on $\bigoplus_{t\in T_{\lambda}} H_q(tN_{\lambda}t^{-1};A)$ restricts to the $tH_{\lambda}t^{-1}$-action on the summand $ H_q(tN_{\lambda}t^{-1};A)$ and that it permutes $ H_q(N_{\lambda};A)$ and $ H_q(tN_{\lambda}t^{-1};A)$. It is not difficult to show now, see for example \cite[Proposition III.5.3]{brown1982cohomology}, that
\[\bigoplus_{t\in T_{\lambda}} H_q(tN_{\lambda}t^{-1};A)\cong \ihg  H_q(N_{\lambda};A).\]
This finishes the prove of the homology isomorphism of the lemma. 

Next, we outline the proof of the cohomology isomorphism which is similar.

First, the analog of (\ref{eq. short exact inter}) is the exact sequence
\begin{equation*}
    0\rightarrow  A\rightarrow \mathrm{Hom}(\Z[V], A)\rightarrow \mathrm{Hom}(\Z[E], A) \rightarrow 0.
\end{equation*}
which can be identified with the cochain complex $C^*(X, A)$ of the tree $X$ with the coefficients in $A$.  Applying  $\mathrm{Hom}_{\m}(F_i,-)$ to the above sequence, we obtain a short exact sequence
\begin{equation*}
    0\rightarrow \mathrm{Hom}_{\m}(F_i, A)\rightarrow \mathrm{Hom}_{\m}(F_i, \mathrm{Hom}(\Z[V], A))\rightarrow \mathrm{Hom}_{\m}(F_i, \mathrm{Hom}(\Z[E], A))  \rightarrow 0.
\end{equation*}
Using the tensor-hom adjunction, we obtain
\begin{align*}
\mathrm{Hom}_{\m}(F_i, \mathrm{Hom}(\Z[V], A)) &\cong \mathrm{Hom}_{\m}(F_i\otimes \Z[V], A)\\
& \cong  \prod_{\lambda\in\Lambda} \mathrm{Hom}_{tN_{\lambda}t^{-1}}(F_i, A).
\end{align*}
where second isomorphism follows from Lemma \ref{tensor_identities} and the adjointness of restriction and extension of scalars functors \cite[\S III.3, eq.~(3.3)]{brown1982cohomology}. Substituting in the long exact sequence above and taking the cohomology, gives the long exact sequence 
\begin{align*}
 0 &\rightarrow \mathrm{Hom}_{\m}(\Z, A)\rightarrow \mathrm{Hom}_{\m}(\Z[V], A)\xrightarrow{f} \mathrm{Hom}_{\m}(\Z[E], A)  \rightarrow\\
    \cdots &\to  H^q({\m};  A)\xrightarrow{\theta^{\ast}}   \prod_{\lambda\in\Lambda, t\in T_{\lambda}}  H^q({tN_{\lambda}t^{-1}} ; A)\to H^q(\m; \mathrm{Hom}(\Z[E], A))\to \cdots
\end{align*}
If $\m$ acts trivially on $A$, then the first three terms in the above sequence compute the cohomology of the quotient graph $X/{\m}$ which is contractible. Therefore $f$ is surjective in this case. It follows that $\theta^*$ is an isomorphism for $q>1$ and also for $q=1$ if $\m$ acts trivially on $A$. By an analogous argument to the homological case but now using   \cite[Proposition III.5.8]{brown1982cohomology}, it follows that 
\[\prod_{t\in T_{\lambda}} H^q(tN_{\lambda}t^{-1};A)\cong \chg  H^q(N_{\lambda};A).\]
 \qedhere
\end{proof}

Let $(G,\{H_{\lambda}\}_{\lambda\in\Lambda},\{N_{\lambda}\}_{\lambda\in\Lambda})$ be a Cohen-Lyndon triple and $A$ any $G$-module. There are Lyndon-Hoschild-Serre spectral sequences
\[E^{\lambda,2}_{p,q}= H_p(\bhl; H_q(N_{\lambda};A))\Rightarrow  H_{p+q}(H_{\lambda};A),\]
\[F^2_{p,q}= H_p(\bg; H_q(\m;A))\Rightarrow  H_{p+q}(G;A)\]
associated with the triples $(H_{\lambda},N_{\lambda},A)$ and $(G,\m,A)$, respectively (see for example \cite[Chapter VII]{brown1982cohomology}). Let $E_{p,q}=\bsl E^{\lambda}_{p,q}$. Then
\[E^2_{p,q}=\bsl  H_p(\bhl; H_q(N_{\lambda};A))\Rightarrow\bsl  H_{p+q}(H_{\lambda};A)\]
by Lemma \ref{lem. sum and product}.

The inclusions $\bhl\hookrightarrow\bg$ and $N_{\lambda}\hookrightarrow\m$ induce a morphism $\phi: E_{p,q}\rightarrow F_{p,q}$. By Proposition \ref{prop. induce module} and Shapiro's lemma, the restriction $\phi:E^2_{p,q}\rightarrow F^2_{p,q}$ is an isomorphism for all $q>1$ and also for $q=1$ if $\m$ acts trivially on $A$. In short:

\begin{theorem}\label{thm. CL imply spectral sequence}
Let $(G,\{H_{\lambda}\}_{\lambda\in\Lambda},\{N_{\lambda}\}_{\lambda\in\Lambda})$ be a Cohen-Lyndon triple. Then for every $G$-module $A$, there is a natural morphism of homological spectral sequences $\phi: E_{p,q}\rightarrow F_{p,q}$ such that
    \begin{align*}
        &E^2_{p,q}=\bsl  H_p(\bhl; H_q(N_{\lambda};A))\Rightarrow\bsl  H_{p+q}(H_{\lambda};A),\\
        &F^2_{p,q}= H_p(\bg; H_q(\m;A))\Rightarrow  H_{p+q}(G;A),
    \end{align*}
$\phi$ is induced by the inclusions $\bhl\hookrightarrow\bg,N_{\lambda}\hookrightarrow\m$, and $\phi$ restricts to an isomorphism $\phi:E^2_{p,q}\xrightarrow[]{\cong} F^2_{p,q}$ for all $q>1$ and also for $q=1$ if $\m$ acts trivially on $A$.
    
Moreover, the analogous statement holds for cohomology as well.
\end{theorem}

We further investigate the morphism $\phi$ of the above theorem. Let $P_i\twoheadrightarrow\mathbb{Z}$ be a free resolution of $\mathbb{Z}$ over $\mathbb{Z}\bg$ and $S_i\twoheadrightarrow\mathbb{Z}$ a free resolution of $\mathbb{Z}$ over $\mathbb{Z}G$. The spectral sequence $E_{p,q}$ is induced by the double complex $C_{p,q}=\bsl(P_p\otimes S_q)\otimes_{H_{\lambda}}A$ and $F_{p,q}$ is induced by $D_{p,q}=(P_p\otimes S_q)\otimes_G A$. The surjections 
$$(P_p\otimes S_q)\otimes_{H_{\lambda}}A\twoheadrightarrow (P_p\otimes S_q)\otimes_G A$$
 induce a surjection $C_{p,q}\twoheadrightarrow D_{p,q}$, which in turn induces the morphism $\phi$. Let $R_{p,q}$ be the kernel of the surjection $C_{p,q}\twoheadrightarrow D_{p,q}$ and $E_{p,q}(R)$ the spectral sequence associated with $R_{p,q}$. It turns out that $E_{p,q}(R)\Rightarrow H_{p+q+1}(G,\{H_{\lambda}\}_{\lambda\in\Lambda};A)$ and we have the following commutative diagram of long exact sequences.


\begin{proposition}\label{prop. long exact relative}
Let $(G,\{H_{\lambda}\}_{\lambda\in\Lambda},\{N_{\lambda}\}_{\lambda\in\Lambda})$ be a Cohen-Lyndon triple and $A$ a $\bg$-module. Then there is a commutative diagram of exact sequences
\begin{adjustwidth}{-20pt}{-100pt}
\begin{small}
\begin{tikzcd}
    \cdots\arrow[r] &\bsl H_i(H_{\lambda};A)\arrow[r,"\iota_{\ast}"] \arrow[d] & H_i(G;A)\arrow[r] \arrow[d] & H_i(G,\{H_{\lambda}\}_{\lambda\in\Lambda};A)\arrow[r] \arrow[d,"id"] &\bsl H_{i-1}(H_{\lambda};A)\arrow[r] \arrow[d] &\cdots\\
    \cdots\arrow[r] &\bsl H_i(\bhl;A)\arrow[r,"\overline{\iota}_{\ast}"] & H_i(\bg;A)\arrow[r] & H_i(G,\{H_{\lambda}\}_{\lambda\in\Lambda};A)\arrow[r] &\bsl H_{i-1}(\bhl;A)\arrow[r] &\cdots
\end{tikzcd}
\end{small}
\end{adjustwidth}
where $\iota_{\ast}$ and $\overline{\iota}_{\ast}$ are induced by the inclusions $H_{\lambda}\hookrightarrow G$ and $\bhl\hookrightarrow\bg$, respectively. Moreover, the cohomological analog of the above statement holds as well.
\end{proposition}


\begin{proof}
We will only prove the homological version since the cohomological version is similar. First, we compute the limit of the spectral sequence $E_{p,q}(R)$.
\begin{claim}
$E_{p,q}(R)\Rightarrow  H_{p+q+1}(G,\{H_{\lambda}\}_{\lambda\in\Lambda};A)$.
\end{claim}
\begin{proof}[Proof of the claim]
Let $\Delta$ be the kernel of the augmentation $\bsl\mathbb{Z}[G/H_{\lambda}]\rightarrow\mathbb{Z}$. 
Then we have a short exact sequence
\begin{equation}\label{eq. short exact star}
    0\rightarrow P_p\otimes S_q\otimes \Delta\rightarrow P_p\otimes S_q\otimes\left(\bsl\mathbb{Z}[G/H_{\lambda}]\right)\rightarrow P_p\otimes S_q\rightarrow 0.
\end{equation}

\noindent Note that
\[\left(P_p\otimes S_q\otimes\mathbb{Z}[G/H_{\lambda}]\right)\otimes_G A\cong(P_p\otimes S_q)\otimes_{H_{\lambda}}A.\]

\noindent  Thus, tensoring \eqref{eq. short exact star} by $-\otimes_G A$ yields
\begin{equation}\label{eq. short exact kernel}
    0\rightarrow (P_p\otimes S_q\otimes \Delta)\otimes_G A\rightarrow C_{p,q}\rightarrow D_{p,q}\rightarrow 0.
\end{equation}

\noindent  Let $T_n=\bigoplus_{p+q=n}P_p\otimes S_q$. Notice that $T_n\otimes \Delta\twoheadrightarrow \Delta$ is a free resolution of $\Delta$ over $\mathbb{Z}G$. Therefore, it can be used to compute the relative homology of $G$, which by definition given in (\ref {eq. relative group cohomology}), is
$$H_{n}(G,\{H_{\lambda}\}_{\lambda\in\Lambda};A)=\tor^G_{n-1}(\Delta,A)=H_{n-1}((T_*\otimes \Delta) \otimes_G A).$$
Hence, to prove the claim, it suffices to show that 
\[(P_p\otimes S_q\otimes \Delta)\otimes_G A\cong R_{p,q},\]
which will be established once we show that \eqref{eq. short exact kernel} is exact. It is easy to check that
\[0\rightarrow A\otimes\Delta\rightarrow A\otimes\left(\bsl\mathbb{Z}[G/H_{\lambda}]\right)\rightarrow A\rightarrow 0\]
is exact (by e.g.~\cite[Corollary 3.1.5]{Weibel}). Since $P_p$ is free abelian and $S_q$ is a free $G$-module, we have an exact sequence
\[0\rightarrow (P_p\otimes\Delta\otimes A)\otimes_G S_q\rightarrow \left[P_p\otimes A\otimes \left(\bsl\mathbb{Z}[G/H_{\lambda}]\right)\right]\otimes_G S_q\rightarrow (P_p\otimes A)\otimes_G S_q\rightarrow 0.\]

\noindent  which, by basic tensor identities, transforms to the desired one.
\end{proof}

Let us return to the proof of Proposition \ref{prop. long exact relative}. We have a short exact sequence 
$$0\rightarrow R_{p,q}\rightarrow C_{p,q}\rightarrow D_{p,q}\rightarrow 0,$$
 whose vertical homology gives the following long exact sequence of the $E^1$-terms of the associated spectral sequences
\[\cdots\rightarrow H_q(R_{p,\ast})\rightarrow  H_q(C_{p,\ast})\rightarrow  H_q(D_{p,\ast})\rightarrow  H_{q-1}(R_{p,\ast})\rightarrow\cdots\]

\noindent We claim that $H_q(D_{p,\ast})=H_q(C_{p,\ast})$ for all $q>0$. To see this, first note that for all $p$ and all $q>0$, applying Lemma \ref{tensor_identities}, we have
\begin{align*}
    D_{p,q}&=(P_p\otimes S_q)\otimes_G A\\
    &=(S_q\otimes P_p\otimes A)\otimes_G \Z &\text{by Lemma \ref{tensor_identities} (ii)}\\
    &=S_q\otimes_G (P_p\otimes A) &\text{by Lemma \ref{tensor_identities} (ii)}\\
    &=(S_q\otimes_{\m}(P_p\otimes A))\otimes_{\bg} \Z &\text{by Lemma \ref{tensor_identities} (ii)}\\
    &=((S_q\otimes_{\m}\Z)\otimes P_p\otimes A)\otimes_{\bg} \Z &\text{as }\m \text{ acts trivially on }P_p \text{ and } A\\
    &=P_p\otimes_{\bg}((S_q\otimes_{\m}\Z)\otimes A) &\text{by Lemma \ref{tensor_identities} (ii)}\\
    &=P_p\otimes_{\bg}(S_q\otimes_{\m}A) &\text{as }\m \text{ acts trivially on }A.
\end{align*}

\noindent As $P_p$ is a free $\Z \bg$-module, we have
\begin{align*}
H_q(D_{p,\ast})&=P_p\otimes_{\bg}H_q(S_{\ast}\otimes_{\m}A)\\
&=P_p\otimes_{\bg}H_q(\m;A)\\
&=P_p\otimes_{\bg}\left(\bsl \ihg  H_q(N_{\lambda};A)\right)\\
&=\bsl\left( P_p\otimes_{\bg}\ihg  H_q(N_{\lambda};A)\right)\\
&=\bsl\left( P_p\otimes_{\bhl}  H_q(N_{\lambda};A)\right)
\end{align*}
by Proposition \ref{prop. induce module} and Shapiro's lemma.

\noindent Similarly, we have
\[H_q(C_{p,\ast})=\bsl \left(P_p\otimes_{\bhl}H_q(N_{\lambda};A)\right),\]
as desired.

\noindent  The equality $H_q(D_{p,\ast})=H_q(C_{p,\ast})$ for $q>0$ shows that $E^1_{p,q}(R)= H_q(R_{p,\ast})=0$ for all $q>0$ and hence $E^2_{p,0}(R)= H_{p+1}(G,\{H_{\lambda}\}_{\lambda\in\Lambda};A)$ for all $p$.
Also, since $\mathbb{Z}\otimes_{\m}A=A$, we have
\begin{align*}
    &E^1_{p,0}(C)= H_0(C_{p,\ast})=\bsl P_p\otimes_{\bhl}A,\\
    &E^1_{p,0}(D)= H_0(D_{p,\ast})=P_p\otimes_{\bg}A.
\end{align*}

\noindent So, we get a short exact sequence
\begin{equation}\label{eq. short exact lower row}
    0\rightarrow E^1_{p,0}(R)\rightarrow \bsl P_p\otimes_{\bhl} A\rightarrow P_p\otimes_{\bg}A\rightarrow 0,
\end{equation}
whose long exact sequence is the bottom row of the desired diagram.

Similar to the proof of the claim, one can show that there is a short exact sequence
\begin{equation}\label{eq. short exact upper row}
  0\rightarrow (S_q\otimes \Delta) \otimes_G A \rightarrow \bsl S_q\otimes_{H_{\lambda}} A \rightarrow S_q\otimes_G A \rightarrow 0. 
\end{equation}

\noindent  As $S_q$ is a free $\Z G$-module and we can view $P_p$ as a $\Z G$-module (through the $\bg$-action), there is a $\Z G$-module homomorphism from the resolution $S_q\twoheadrightarrow\mathbb{Z}$ to $P_p\twoheadrightarrow\mathbb{Z}$. This induces a map $\bsl S_q\otimes_{H_{\lambda}} A\rightarrow \bsl P_p\otimes_{\bhl} A$ and a map $S_q\otimes_G A\rightarrow P_p\otimes_{\bg}A$, yielding a commutative diagram

\begin{equation}\label{cd}
\begin{tikzcd}
    \bsl P_p\otimes_{\bhl} A \arrow[r] & P_p\otimes_{\bg}A\\
    \bsl S_q\otimes_{H_{\lambda}} A\arrow[u] \arrow[r] & S_q\otimes_G A\arrow[u]
\end{tikzcd}
\end{equation}

\noindent By composing $\bsl S_q\otimes_{H_{\lambda}} A\rightarrow \bsl P_p\otimes_{\bhl} A$ with $(S_q\otimes \Delta) \otimes_G A \rightarrow \bsl S_q\otimes_{H_{\lambda}} A$, we get a map $(S_q\otimes \Delta) \otimes_G A \rightarrow \bsl P_p\otimes_{\bhl} A$, whose image lies in 
$$\ker(\bsl P_p\otimes_{\bhl} A\rightarrow P_p\otimes_{\bg}A)=E^1_{p,0}(R)$$
 as the diagram \eqref{cd} commutes. We therefore have a map from \eqref{eq. short exact upper row} to \eqref{eq. short exact lower row}, whose associated long exact sequences is the commutative diagram in the statement.
\end{proof}

\begin{remark}\label{rm. iso of 2.2 (ii)}
Similar to the proof of Proposition \ref{prop. long exact relative}, one can show that there is a short exact sequence
\begin{equation}\label{eq. short exact extra row}
    0\rightarrow (P_p\otimes\overline{\Delta})\otimes_{\bg}A\rightarrow\bsl P_p\otimes_{\bhl}A\rightarrow P_p\otimes_{\bg}A\rightarrow 0,
\end{equation}
where $P_p$ and $A$ are as in the proof of Proposition \ref{prop. long exact relative}, and $\overline{\Delta}$ is the kernel of the augmentation $\bsl\mathbb{Z}[\bg/\bhl]\twoheadrightarrow\mathbb{Z}$. The natural map from \eqref{eq. short exact lower row} to \eqref{eq. short exact extra row} gives rise to a commutative diagram of the corresponding long exact sequences, which together with the five lemma yields an isomorphism 
\begin{equation}\label{excision_iso}
H_{\ast}(\bg,\{\bhl\}_{\lambda\in\Lambda};A)\cong  H_{\ast}(G,\{H_{\lambda}\}_{\lambda\in\Lambda};A)
\end{equation}
 (of course, under the assumption that $(G,\{H_{\lambda}\}_{\lambda\in\Lambda},\{N_{\lambda}\}_{\lambda\in\Lambda})$ is a Cohen-Lyndon triple and $A$ is a $\bg$-module). Moreover, The analogous isomorphism for cohomology holds as well.
\end{remark}

An easy consequence of Proposition \ref{prop. long exact relative} is a direct sum decomposition of (co)homology.

\begin{corollary}\label{direct sum of cohomology}
Let $(G,\{H_{\lambda}\}_{\lambda\in\Lambda},\{N_{\lambda}\}_{\lambda\in\Lambda})$ be a Cohen-Lyndon triple and $A$ a $\bg$-module. 
\begin{enumerate}
    \item[(a)] If for some $p\in\mathbb{N}$, $\bsl H_p(H_{\lambda};A)=0$ and the natural map $\bsl H_{p-1}(H_{\lambda};A)\rightarrow H_{p-1}(G;A)$ is injective, then
    \begin{equation*}
         H_p(\bg;A)\cong  H_p(G;A)\oplus\left(\bsl H_p(\bhl;A)\right).
    \end{equation*}
    \item[(b)] If for some $p\in\mathbb{N}$, $\pl H^p(H_{\lambda};A)=0$ and the natural map $ H^{p-1}(G;A)\rightarrow\pl H^{p-1}(H_{\lambda};A)$ is surjective, then
    \begin{equation*}
     H^p(\bg;\cm)\cong  H^p(G;\cm)\oplus \left(\pl  H^p(\bhl;\cm)\right).   
\end{equation*}
\end{enumerate}
\end{corollary}

\begin{proof}
We only prove the homological version (a) and point out that the proof of (b) is analogous. To shorten the notation, we denote $ H_{\ast}(-;A)$ by $ H_{\ast}(-)$, $\bsl H_{\ast}(H_{\lambda};A)$ by $ H_{\ast}(\mathcal{H})$, $\bsl H_{\ast}(\bhl;A)$ by $ H_{\ast}(\overline{\mathcal{H}})$, and $ H_{\ast}(G,\{H_{\lambda}\}_{\lambda\in\Lambda};A)$ by $ H_{\ast}(G,\mathcal{H})$. Use Proposition \ref{prop. long exact relative} and consider the commutative diagrams of exact sequences

\begin{center}

\begin{tikzcd}
 H_{p+1}(\overline{\mathcal{H}}) \arrow[r, "\phi_{p+1}"]  &  H_{p+1}(\overline{G})  \arrow[r, "\lambda_{p+1}"] &  H_{p+1}({G}, \mathcal{H})   \arrow[r, "\xi_{p+1}"]&  H_{p}(\overline{\mathcal{H}}) \\
{ H}_{p+1}(\mathcal{H}) \arrow[r, "r_{p+1}"] \arrow[u, "\psi_{p+1}"] & { H}_{p+1}(G) \arrow[r,"\rho_{p+1}"] \arrow[u, "\theta_{p+1}"]  &{ H}_{p+1}({G, \mathcal{H}}) \arrow[r,"\eta_{p+1}"] \arrow[u, "id_{p+1}"] & { H}_{p}({\mathcal{H}}) \arrow[u, "\psi_p"]
\end{tikzcd}

\begin{tikzcd}
 H_{p}(\overline{\mathcal{H}}) \arrow[r, "\phi_{p}"]  &  H_{p}(\overline{G})  \arrow[r, "\lambda_{p}"] &  H_{p}({G}, \mathcal{H})   \arrow[r, "\xi_{p}"]&  H_{p-1}(\overline{\mathcal{H}}) \\
{ H}_{p}(\mathcal{H}) \arrow[r, "r_{p}"] \arrow[u, "\psi_{p}"] & { H}_{p}(G) \arrow[r,"\rho_{p}"] \arrow[u, "\theta_{p}"]  &{ H}_{p}({G, \mathcal{H}}) \arrow[r,"\eta_{p}"] \arrow[u, "id_{p}"] & { H}_{p-1}({\mathcal{H}}) \arrow[u, "\psi_{p-1}"]
\end{tikzcd}

\end{center}

\noindent
where the horizontal maps come from the exact sequences for the pairs $(G,\mathcal{H})$, $(\bg,\overline{\mathcal{H}})$ and $\psi_{\ast},\theta_{\ast}$ are natural maps induced by the surjections $G\twoheadrightarrow \bg,H_{\lambda}\twoheadrightarrow\bhl$, respectively.

The hypothesis imply that $\rho_p:H_p(G) \rightarrow  H_p(G,\mathcal{H})$ is an isomorphism which in turn shows that the map $\lambda_{p}$ is surjective. Since ${ H}_{p}(\mathcal{H})=0$, $\rho_{p+1}$ is surjective. So $\lambda_{p+1}$ is also surjective. This shows $\xi_{p+1}=0$ and hence $\phi_p$ is injective. Since $\theta_p\circ ({\rho_p})^{-1}$ is a section for $\lambda_p$, the result follows.
\end{proof}

An immediate corollary to the above is an estimate of the (co)homological dimension. To shorten the notation, if $(G,\{H_{\lambda}\}_{\lambda\in\Lambda},\{N_{\lambda}\}_{\lambda\in\Lambda})$ is a Cohen-Lyndon triple, let
\begin{align*}
    &\hd(\mathcal{H})=\sup_{\lambda\in\Lambda}\{\hd(H_{\lambda})\},~\hd(\overline{\mathcal{H}})=\sup_{\lambda\in\Lambda}\{\hd(\bhl)\},\\
    &\cd(\mathcal{H})=\sup_{\lambda\in\Lambda}\{\cd(H_{\lambda})\},~~\cd(\overline{\mathcal{H}})=\sup_{\lambda\in\Lambda}\{\cd(\bhl)\}.
\end{align*}

\begin{corollary}\label{cd of Dehn fillings}
Let $(G,\{H_{\lambda}\}_{\lambda\in\Lambda},\{N_{\lambda}\}_{\lambda\in\Lambda})$ be a Cohen-Lyndon triple. Then
$$\hd(\bg)\leqslant \max\{\hd(G),\hd(\mathcal{H})+1,\hd(\overline{\mathcal{H}})\},~~\cd(\bg)\leqslant \max\{\cd(G),\cd(\mathcal{H})+1,\cd(\overline{\mathcal{H}})\}.$$
\end{corollary}

By Theorem \ref{homological criterion},  the commutativity of colimits of coefficients with the cohomology functor can be used to characterize property $FP_n$, which is our next goal.

\begin{theorem}\label{property FP}
Let $G$ be a group with a finite family of subgroups $\{H_i\}^m_{i=1}\hookrightarrow_h G$. Suppose that $G$ is of type $FP_n$ for some $n\in\mathbb{N}^+\cup\{\infty\}$. Then for sufficiently deep $\{N_i\lhd H_i\}^m_{i=1}$, $\bg$ is of type $FP_n$ if and only if all $\bh_i$ are of type $FP_n$.
\end{theorem}

\begin{proof}
Suppose that $\bg$ is of type $FP_n$. By Theorem \ref{thm. group theoretic Dehn filling}, we may assume that $\{\bh_i\}^m_{i=1}\hookrightarrow_h \bg$. Then by Lemma \ref{lem. he finiteness}, $\bh_i$ are of type $FP_n$ as $\bg$ is.

Conversely, suppose that all $\bh_i$ are of type $FP_n$. To shorten notations, we write $ H^{\ast}(G,\mathcal{H};-)$ for $ H^{\ast}(G,\{H_i\}^m_{i=1};-)$, $ H^{\ast}(\mathcal{H};-)$ for $\prod^m_{i=1} H^{\ast}(H_i;-)$, and $ H^{\ast}(\overline{\mathcal{H}};-)$ for $\prod^m_{i=1} H^{\ast}(\bh_i;-)$. Let $\{A_j\}_{j\in J}$ be a direct system of $\bg$-modules such that $\varinjlim A_j=0$. By Proposition \ref{prop. long exact relative}, we have a commutative diagram of exact sequences for each $j\in J$:
\begin{center}
\begin{tikzcd}
\cdots\arrow[r] &  H^k(G,\mathcal{H};A_j) \arrow[r] \arrow[d,"id"] &  H^k(\bg;A_j) \arrow[r] \arrow[d] &  H^k(\overline{\mathcal{H}};A_i) \arrow[r] \arrow[d] &  H^{k+1}(G,\mathcal{H};A_j) \arrow[r] \arrow[d,"id"] & \cdots \\
\cdots\arrow[r] &  H^k(G,\mathcal{H};A_j) \arrow[r] &  H^k(G;A_j) \arrow[r] &  H^k(\mathcal{H};A_j) \arrow[r] &  H^{k+1}(G,\mathcal{H};A_j) \arrow[r] & \cdots
\end{tikzcd}
\end{center}
The above remains a commutative diagram of exact sequences after taking direct limit, by \cite[Theorem 2.6.15]{Weibel}. For $k\leqslant n$, we have $\varinjlim H^k(G;A_j)=0$ by Theorem \ref{homological criterion}. By Lemma \ref{lem. he finiteness}, we also have $\varinjlim  H^k(\mathcal{H};A_j)=0$, and thus $\varinjlim  H^k(G,\mathcal{H};A_j)=0$, which implies that $\varinjlim H^k(\bg;A_j)=\varinjlim H^k(\overline{\mathcal{H}};A_j)=0$ except for $k=n$. Since $\varinjlim  H^n(\overline{\mathcal{H}};A_j)=0$, we have $\varinjlim H^n(\bg;A_j)=0$ as well, and thus $\bg$ is of type $FP_n$, again by Theorem \ref{homological criterion}.
\end{proof}

We emphasize that the finiteness of $\Lambda$ is needed in the above proof to guarantee that $\varinjlim  H^k(\mathcal{H};A_j)=\varinjlim H^k(\overline{\mathcal{H}};A_j)=0$ for $k\leqslant n$.

\begin{corollary}\label{cor. finite FP property}
Let $G$ be a group with a finite family of subgroups $\{H_i\}^m_{i=1}\hookrightarrow_h G$. Suppose that $G$ is of type $FP$. Then for sufficiently deep $\{N_i\lhd H_i\}^m_{i=1}$, $\bg$ is of type $FP$ if and only if all $\bh_i$ are of type $FP$.
\end{corollary}

\begin{proof}
    Theorem \ref{property FP} implies that $\bg$ is of type $FP_{\infty}$ if and only if all $\bh_i$ are of type $FP_{\infty}$. So we only need to prove that $\cd(\bg)<\infty$ if and only if $\cd(\bh_i)<\infty$ for all $i$. Suppose that $\cd(\bg)<\infty$. Then since $\bh_i$ are subgroups of $\bg$, we have $\cd(\bh_i)\leqslant \cd(\bg)<\infty$. Conversely, suppose that $\cd(\bh_i)<\infty$ for all $i$. Since $G$ is of type $FP$, we have $\cd(G)<\infty$. As $H_i$ are subgroups of $G$, we also have $\cd(H_i)<\infty$ for all $i$. Then Corollary \ref{cd of Dehn fillings} gives
    \[\cd(\bg)\leqslant \max\{\cd(G),\max_{1\leqslant i\leqslant m}\{  \cd(H_i)+1\},\max_{1\leqslant i\leqslant m}\{\cd(\bh_i)\}\}<\infty.\]
\end{proof}

\begin{proof}[Proof of Theorem \ref{thm. main}]
By Theorem \ref{thm. CL-property}, $(G,H,N)$ is a Cohen-Lyndon triple for sufficiently deep $N\lhd H$. Items (ii)  follows directly from Propositions \ref{prop. long exact relative} and Remark \ref{rm. iso of 2.2 (ii)}.

Theorem \ref{thm. CL imply spectral sequence} provides us a spectral sequences 
\[E^{p,q}_2\Rightarrow  H^{p+q}(G;A)\]
and isomorphisms
\[E^{p,q}_2\cong  H^p(\bh; H^q(N;A))\]
for $q>0$. Theorem \ref{thm. main} (i) then follows by observing that $ H^p(\bg; H^0(\ll N \rr;A))\cong  H^p(\bg;A)$. Similarly, one can prove the homological version.
\end{proof}

Corollary \ref{intro_cor:finiteness} follows directly from  Proposition \ref{direct sum of cohomology}, Corollary \ref{cd of Dehn fillings}, Theorem \ref{property FP}, and Corollary \ref{cor. finite FP property}.

We collect the results of this section and state the full version of Theorem \ref{thm. main}.

\begin{theorem}\label{thm. real main}
    Let $G$ be a group with a hyperbolically embedded family of subgroups $\{H_{\lambda}\}_{\lambda\in\Lambda}\hookrightarrow_h G$. Then the following hold for all sufficiently deep $N_{\lambda}\lhd H_{\lambda},\lambda\in\Lambda$ and all $\bg$-modules $A$, where $\bg=G/\ll\cup_{\lambda\in\Lambda} N_{\lambda}\rr$.
\begin{enumerate}
    \item[(i)] There is a spectral sequence
    \[
    E_2^{p,q}(A)=
    \begin{cases}
     \prod_{\lambda\in \Lambda}H^p(\bhl; H^q(N_{\lambda};A))&\text{for } q>0\\
     H^p(\bg;A)&\text{for } q=0
    \end{cases}
    \Rightarrow  H^{p+q}(G;A),
    \]
    where $\bhl=H_{\lambda}/N_{\lambda}$ for all $\lambda$ and the action of $G$ on $A$ factors through $\bg$. In particular, the action of $\ll\cup_{\lambda\in\Lambda} N_{\lambda}\rr$ on $A$ fixes $A$ pointwise.\\

    \item[(ii)] (Algebraic Excision) For all $n\geqslant 0$ and $\lambda\in\Lambda$, there is a natural isomorphism induced by the quotient maps $G\rightarrow \bg$ and $H_{\lambda}\rightarrow\bhl$,
    \[ H^n(\bg,\bhl;A)\cong  H^n(G,H_{\lambda};A).\]
    
    \item[(iii)]  For all $n\geqslant \sup_{\lambda\in\Lambda}\cd(H_{\lambda})+2$,
    \[ H^n(\bg;A)\cong  H^n(G;A)\oplus \left(\prod_{\lambda\in\Lambda} H^n(\bhl;A)\right).\]
    
    \item[(iv)] $\cd(\bg)\leqslant \max\{\cd(G),\sup_{\lambda\in\Lambda}\cd(H_{\lambda})+1,\sup_{\lambda\in\Lambda}\cd(\bhl)\}.$\\
    
    \item[(v)] If $|\Lambda|<\infty$ and $G$ is of type $FP_n$ for some $n\in\mathbb{N}^+\cup\{\infty\}$, then $\bg$ is of type $FP_n$ if and only if every $\bhl$ is of type $FP_n$.

    \item[(vi)] If $|\Lambda|<\infty$ and $G$ is of type $FP$, then $\bg$ is of type $FP$ if and only if every $\bhl$ is of type $FP$.
\end{enumerate}
\end{theorem}

We end this section by showing that the assumption $n\geqslant \cd(H)+2$ in Corollary \ref{intro_cor:finiteness} (i) cannot be dropped.

\begin{example}\label{eg. counter in +1}
Let $G$ be a free group with basis $\{x,y\}$ and let $H=\langle h\rangle \leqslant G$ where $h=xyx^{-1}y^{-1}$. Then $H\hookrightarrow_h G$ by Example \ref{eg. 2} and $\cd(H)+1=2$. Let $N=\langle h^k \rangle \lhd H$. Note that we can pick $k$ large enough so that $N$ avoids any given finite subset of $H\smallsetminus\{1\}$. By \cite[Theorem 11.1]{lyndon1950cohomology}, $ H^2(\bg;\mathbb{Z})\cong\mathbb{Z}$, and it is well-known that $ H^2(G;\mathbb{Z})=0$ and $ H^2(\bh;\mathbb{Z})\cong \mathbb{Z}/k\mathbb{Z}$. Thus, $ H^2(\bg;\mathbb{Z})\not\cong  H^2(G;\mathbb{Z})\oplus  H^2(\bh;\mathbb{Z})$. Similarly, $\hd(H)+1=2$ and one can show that $ H_2(\bg;\mathbb{Z})\not\cong  H_2(G;\mathbb{Z})\oplus  H_2(\bh;\mathbb{Z})$.
\end{example}

\section{Dehn fillings and duality}\label{sec:duality}

 Let $G$ a group and $\h=\{H_{i}\}_{i=1}^m$ a finite collection of subgroups. Following Bieri-Eckmann \cite{bieri1978relative}, we say that $(G, \h)$ is a {\it duality pair} of dimension $n$, with dualizing module $C$, if for all $k\in \Z$ and all $G$-modules $A$, one has
$$H^k(G; A)\cong H_{n-k}(G, \h ; C\otimes A)$$
$$H^k(G, \h; A)\cong H_{n-k}(G ; C\otimes A)$$
given by the cap product by the fundamental class $e\in H_n(G, \h ; C)$. Here, we adapt the convention that $H^k(G;A)=H^k(G,\h;A)=H_k(G, \h ; C\otimes A)=H_k(G ; C\otimes A)=0$ for all $k<0$. In which case, it follows that $C\cong H^n(G, \h; \Z G)$. If $C=\Z$ with trivial $G$-action, the pair is called an (orientable) {\it Poincar\'{e} duality} pair, in short, a PD$(n)$-pair.

Let $(G, \h)$ be a duality pair of dimension $n$ with dualizing module $C$ and let $\Delta$ be the kernel of the augmentation $\bigoplus_{i=1}^m \mathbb{Z}[G/H_{i}]\rightarrow \mathbb{Z}$. By \cite[Theorem 4.2]{bieri1978relative}, $G$ is a duality group of dimension $n-1$ with dualizing module $\Delta \otimes C$ and each $H_i$ is a  duality group of dimension $n-1$ with dualizing module $C$ (thinking as an $H_i$-module).

The following result generalises \cite[Corollary 1.5]{wang2018spectral} which deals with the case where $(G, \h)$ is a type $F_{\infty}$ relatively hyperbolic group pair.

\begin{corollary}\label{cor:dual} Let $G$ be a group and $\h=\{H_{i}\}_{i=1}^m$ a collection of subgroups such that $\h\hookrightarrow_h G$. Suppose $(G, \h)$ is a {\it duality pair} of dimension $n$, with dualizing module $C$. Then for all sufficiently deep $\{N_i\lhd H_i\}_{i=1}^m$ and $k\in \Z$
$$H^k(\bg, \overline{\h} ; \Z\bg)\cong 
\begin{cases}
   \bigoplus_{i=1}^m \mathrm{Ind}_{\bh_i}^{\bg}H_{n-k}(N_i; C)&\text{for } k\ne n,\\
     C\otimes_{\m} \Z &\text{for } k=n.
    \end{cases}$$
\end{corollary}
\begin{proof} By Theorem \ref{thm. real main} (ii) (see also  Remark \ref{rm. iso of 2.2 (ii)}) and duality, for $k<n$ one has
$$H^k(\bg, \overline{\h} ; \Z\bg)\cong H_{n-k}(G; C\otimes \Z \bg)\cong H_{n-k}(G; \mathrm{Ind}_{\m}^{G}\mathrm{Res}_{\m}^{G} C)\cong H_{n-k}(\m; C),$$
where the second isomorphism follows from Lemma \ref{tensor_identities}. The claim now follows from Theorem \ref{thm. CL-property} and Proposition \ref{prop. induce module}.
Also, 
$$H^n(\bg, \overline{\h} ; \Z\bg)\cong H_{0}(G; C\otimes \Z \bg)\cong (C\otimes \Z\bg)\otimes_{\Z G} \Z \cong (C\otimes_{\m} \Z G)\otimes_{\Z G} \Z \cong C\otimes_{\m} \Z.$$
\end{proof}

It is worth noting that the assumption that the collection $\h$ is finite in Corollary \ref{cor:dual} cannot be dropped, since if  $(G, \h)$ is a duality pair, then $\h$ must be finite  \cite[Theorem 4.2]{bieri1978relative}. 

\begin{theorem}\label{cor:pd} Let $G$ be a group and $\h=\{H_i\}_{i=1}^m$ a collection of subgroups such that $\h \hookrightarrow_h G$. Suppose, for some integer $0<n$,  $(G, \h )$ is a PD$(n)$-pair and that there are sufficiently deep  $\{\Z\cong N_i\lhd H_i\}_{i=1}^m$  such that each member of $\overline{\mathcal{H}}=\{\bh_i\}_{i=1}^m$ is a PD$(n-2)$-group. Then $\bg$ is a PD$(n)$-group.
\end{theorem}
\begin{proof} For each $1\leq i\leq m$,  $H^k(\bh_i ; \Z\bg)=0$ if $k\ne n-2$ and by Corollary \ref{cor:dual}, $H^k(\bg, \overline{\h} ; \Z\bg)=0$ if $k\ne n-1, n$. The long exact sequence in cohomology for the pair $(\bg, \oh)$ shows that $H^k(\bg ; \Z\bg)=0$ if $k\ne n-2, n-1, n$ and for $k=n$ it gives  $H^{n}(\bg ; \Z\bg)\cong H^{n}(\bg, \oh ; \Z\bg)\cong \Z$.

Next, we consider the spectral sequence of Theorem \ref{thm. real main}

 \[
    E_2^{p,q}=
    \begin{cases}
     \prod_{i=1}^m H^p(\bh_i; H^q(N_i; \Z\bg))&\text{for } q>0\\
     H^p(\bg; \Z\bg)&\text{for } q=0
    \end{cases}
    \Rightarrow  H^{p+q}(G; \Z\bg).
    \]

\noindent Since each $N_i\cong \Z$, $E_2^{p,q}=0$ for $q>1$ and $E_2^{p,1}\cong \prod_{i=1}^m H^p(\bh_i; \Z\bg)$. Since there are only two nontrivial rows on the $E_2^{p,q}$-page of the spectral sequence, by  the cohomological analog of \cite[Ex.~5.2.2]{Weibel}, see also \cite[Theorem XV.5.11]{Cartan_Eilenberg}, we obtain a long exact sequence 

\begin{equation}\label{gysin}
\small{ \dots \to \prod_{i=1}^mH^{k-2}(\bh_i ; H^1(N_i;\Z\bg))\xrightarrow{d_2} H^{k}(\bg ; \Z\bg)\to 
  H^{k}(G ; \Z\bg)\to\prod_{i=1}^mH^{k-1}(\bh_i ; H^1(N_i;\Z\bg))\xrightarrow{d_2} \dots}
  \end{equation}

\noindent  For $k=n-2$,  (\ref{gysin})  yields 
\begin{align*}
H^{n-2}(\bg ; \Z\bg) & \cong  H^{n-2}(G ; \Z\bg)\\
& \cong H_2(G, \h;  \Z\bg)\\
& \cong H_2(\bg, \oh;  \Z\bg)=0.
\end{align*}
\noindent  For $k=n-1$,  (\ref{gysin}) and  the morphism between the Lyndon-Hochschild-Serre spectral sequences associated to the inclusions $(H_{\lambda},N_{\lambda})\hookrightarrow (G,\ll N \rr)$  give us

\begin{center}
\begin{tikzcd}
 0\arrow[r] & H^{n-1}(\bg ; \Z\bg) \arrow[r] &  H^{n-1}(G ; \Z\bg) \arrow[r] \arrow[d, "r^{n-1}"] &  \prod_{i=1}^mH^{n-2}(\bh_i ; H^1(N_i;\Z\bg))\arrow[d, "="]   \\
  & &   \prod_{i=1}^m H^{n-1}(H_i ; \Z\bg) \arrow[r, "\cong"] &  \prod_{i=1}^mH^{n-2}(\bh_i ; H^1(N_i;\Z\bg))
\end{tikzcd}
\end{center}
where the bottom isomorphism follows from the Lyndon-Hochschild-Serre spectral sequence applied to the extensions $1\rightarrow N_i\rightarrow H_i\rightarrow \overline{H}_i\rightarrow 1$. To show that $H^{n-1}(\bg ; \Z\bg)=0$ amounts to showing that $r^{n-1}$ is injective. By \cite[Theorem 2.1]{bieri1978relative}, there is a commutative diagram

\begin{center}
\begin{tikzcd}
 H^{n-1}(G ; \Z\bg) \arrow[r, "r^{n-1}"] \arrow[d, "\cong "] &  \prod_{i=1}^m H^{n-1}(H_i ; \Z\bg)\arrow[d, "\cong"]   \\
     H_{1}(G, \h ; \Z\bg) \arrow[r, "\partial"] & \bigoplus_{i=1}^m H_{0}(H_i ; \Z\bg)
\end{tikzcd}
\end{center}
where by duality the vertical maps are isomorphisms. Thus, it suffices to show that the connecting homomorphism $\partial$  is injective. This follows from the commutativity of the diagram
\begin{center}
\begin{tikzcd}
 H_{1}(G, \h ; \Z\bg) \arrow[r, "\partial"] \arrow[d, "\cong "] & \bigoplus_{i=1}^m H_{0}(H_i ; \Z\bg)\arrow[d, "\cong"]   \\
    H_{1}(\bg, \oh ; \Z\bg) \arrow[r, "\overline\partial"] &  \bigoplus_{i=1}^m  H_{0}(\bh_i ; \Z\bg)
\end{tikzcd}
\end{center}
and the fact the kernel of $\overline\partial$ is $H_{1}(\bg, \Z\bg)=0$. This finishes the proof of $H^{n-1}(\bg ; \Z\bg)=0$.

By \cite[Theorem 6.2]{bieri1978relative}, $G$ and each $\bh_i$ are of type $FP$. So, by Corollary \ref{cor. finite FP property}, $\bg$ is of type $FP$. We have also established that $H^k(\bg ; \Z\bg)=0$ if $k\ne n$ and  $H^{n}(\bg ; \Z\bg)\cong \Z$. By \cite[Theorem 6.2(i)]{bieri1978relative}, $\bg$ is a PD$(n)$-group.
\end{proof}

\section{Simplicial volume of Dehn fillings}\label{sec:sv}
 For detailed background on bounded cohomology and simplicial volume, we refer to  \cite{gromov1982}, \cite{ivanov85} and \cite{frigerio2017}. 


Let $G$ be a group. Consider the singular chain complex $C_*(BG; \R)$ endowed with the $\ell^1$-norm 
$$|c |_1=\sum_{i=1}^k |a_i|, \;\;\; \forall c=\sum_{i=1}^k a_i \sigma_i \in C_*(BG; \R).$$
The cochain complex of bounded cochains $C^*_b(BG; \R)$ coincides  with the normed dual complex of $C_*(BG; \R)$.  The norm on chains induces a $\ell^1$-semi-norm $||\cdot||_1$ on $H_*(G ; \R)= H_*(C_*(BG;\R))$ and $\ell^{\infty}$-semi-norm $||\cdot||_{\infty}$ on $H^*_b(G ; \R)= H^*(C^*_b(BG;\R))$. When $\h=\{H_{i}\}_{i=1}^m$ is a collection of subgroups of $G$,  $H_*(G, \h ; \R)$ and $H^*_b(G, \h ; \R)$ and their  semi-norms are defined analogously \cite[\S 9.2]{minyam2007}.



As before, when $N_i\lhd  H_i$, we let $\mathcal{N}=\bigcup^m_{i=1}N_i$, \; $\overline{\h}=\{\bh_{i}\}_{i=1}^m$ and $\bg=G/{\ll \mathcal N\rr}$.

\begin{lemma}\label{lem:fund} Let $G$ be a group and $\h=\{H_i\}_{i=1}^m$ a collection of subgroups such that $\h \hookrightarrow_h G$. Suppose, for some integer $n\geqslant 2$,  $(G, \h )$ is a PD$(n)$-pair and for a sufficiently deep  $\{N_i\lhd H_i\}_{i=1}^m$,  $\mathrm{cd}(\bh_i)\leq n-2$ for each $1\leq i\leq m$. Then  $H_n(\bg; \Z)=\Z$.
\end{lemma}

\begin{proof} The result follows from the long exact sequence in homology of the pair $(\bg, \overline{\h})$ and, by Theorem \ref{thm. real main}, the isomorphism $H_n(\bg, \oh; \Z)\cong H_n(G, \h; \Z)=\Z$.
\end{proof}

\begin{definition}\label{def:sv} Under the hypothesis of Lemma \ref{lem:fund}, we define the {\it simplicial volume} of $\bg$ by
$$||\bg || = \inf \{ |c|_1 \; | \; c\in C_n(B\bg ; \mathbb R), \; [c]=[\bg]\in H_n(\bg; \mathbb R)\},$$ 
where $[G]$ is the image  of the fundamental class under the change of coefficients map $H_n(\bg; \mathbb Z)\to H_n(\bg; \mathbb R)$. Similarly, we define 
$$||G, \h || = \inf \{ |f|_1 \; | \; f\in C_n(BG,  \bigsqcup_{i=1}^m BH_{i}) ; \mathbb R), \; [f]=[G, \h]\in H_n(G, \h; \mathbb R)\}.$$ 
The simplicial volume of $(\bg, \oh)$ is defined analogously. 
\end{definition}




The following result is a generalisation of \cite[Theorem 1.5]{fujmann2011}.

\begin{theorem}\label{thm:norm}
Let $G$ be a group and $\h=\{H_i\}_{i=1}^m$ a collection of subgroups such that $\h \hookrightarrow_h G$. Suppose, for some integer $n\geq 2$,  $(G, \h )$ is a PD$(n)$-pair and for a sufficiently deep  $\{N_i\lhd H_i\}_{i=1}^m$,  $\mathrm{cd}(\bh_i)\leq n-2$ for each $1\leq i\leq m$.  Then, $\cd(\bg)=n$, $H_n(\bg; \Z)=\Z$. In addition,
\begin{itemize}
\item[(i)]  if the group $\bh_i$ is amenable for each $1\leq i\leq m$, then $||\bg ||\leq ||G, \h ||;$
\item[(ii)] if $G$ is hyperbolic relative to $\h$, then $||\bg|| >0$.
\end{itemize}
\end{theorem}
\begin{proof}  Recall that since  $(G, \h )$ is a PD$(n)$-pair, each $H_i$ is a PD$(n-1)$-group  \cite[Theorem 4.2]{bieri1978relative} and hence $\cd(H_i)=n-1$. The first two claims now follow from Theorem \ref{thm. real main} (iv) and Lemma \ref{lem:fund}. 

For part (i), by Theorem \ref{thm. real main} (ii), ${p_n} :H_n(G, \h ; \Z)\to H_n(\bg, \overline{\h} ; \Z)$  induced by $(G, \h)\twoheadrightarrow (\bg, \overline{\h})$ is an isomorphism. Since each $\bh_i$ is amenable,  $j^n: H^n_b(\bg, \overline{\h} ; \R)\to H^n_b(\bg ; \R)$ is an isometric  isomorphism \cite[Theorem 5.14]{frigerio2017}. By Duality Principle \cite[Lemma 6.1]{frigerio2017}, we obtain
\begin{align*}
||\bg|| 
&=\sup \{ \langle j^n(\psi), [\bg] \rangle \; | \; \psi\in H^n_b(\bg, \overline{\h} ; \R), \; ||\psi ||_{\infty}\leq 1\},\\
&=\sup\{ \langle \psi, j_n([\bg]) \rangle \; | \; \psi\in H^n_b(\bg, \overline{\h} ; \R), \; ||\psi ||_{\infty}\leq 1\},\\
&= ||j_n([\bg]) ||_1,\\
&= ||\bg, \overline{\h} ||,\\
&\leq ||G, \h ||,
\end{align*}
where the last inequality follows from the functoriality of the $\ell^1$-semi-norm \cite[Proposition 2.1]{loh2011}.

For part (ii), since $G$ is hyperbolic relative to $\h$, then $\bg$ is hyperbolic relative $\oh$ \cite[Theorem 1.1]{osin2007peripheral}. Consider the commutative diagram of long exact sequences 
\begin{center}
\begin{tikzcd}
  \prod_{i=1}^m H^{n-1}_b(\bh_{i} ; \R) \arrow[r] \arrow[d] &   H^n_b(\bg, \overline{\h}; \R) \arrow[r] \arrow[d] &   H^n_b(\bg ; \R) \arrow[r] \arrow[d, "c"] &  \prod_{i=1}^m H^n_b(\bh_{i} ; \R) \arrow[d]  \\
  \prod_{i=1}^m  H^{n-1}(\bh_{i}; \R) \arrow[r] & H^n(\bg, \overline{\h}; \R) \arrow[r] &   H^n(\bg ; \R)   \arrow[r] \ & \prod_{i=1}^m  H^n(\bh_{i}; \R)
\end{tikzcd}
\end{center}
 Since  $\mbox{cd}(\bh_i)\leq n-2$, the map  $H^n(\bg, \overline{\h}; \R) \to   H^n(\bg ; \R)$ is onto. Since  $\bg$ is hyperbolic relative to $\overline{\h}$, the comparison map  $H^n_b(\bg, \overline{\h}; \R)\to H^n(\bg, \overline{\h}; \R)$ is onto \cite[Theorem 1.1]{franc2018}, see also \cite{minyam2007}. It follows that the comparison map $c: H^n_b(\bg ; \R)\to H^n(\bg ; \R)$ is also onto which shows that $||\overline G||>0$  \cite[Lemma 6.1]{frigerio2017}, see also \cite[Proposition 7.10]{frigerio2017}.
\end{proof}

\subsection{Dehn fillings of pinched negatively curved manifolds} In this section, we illustrate how the results of the previous sections apply in a geometric setting. First, we need a lemma.

\begin{lemma} \label{lem:nil} Let $G$ be a group and $\h=\{H_{i}\}_{i=1}^m$ a collection of  finitely generated torsion-free nilpotent subgroups with $rk(Z(H_i))\geq 2$, for all $i$. Then there are infinitely many collections  of subgroups $\{N_{i}\}_{i=1}^m$ such that $\Z\cong N_i\lhd Z(H_i)$, $\bh_i$ is torsion-free and $j_k: H_k(\bg; \Z)\to H_k(\bg, \overline\h ; \Z)$ is an isomorphism for all $k> \max\{ \hd (H_i)\;|\; 1\leq i\leq m\}$.
\end{lemma}
\begin{proof}  For each $i$, the center $Z(H_i)$ contains a factor $\langle x_i, y_i\rangle\cong \Z^2$.  Let $N^{s,t}_i=\langle x^s_iy^t_i\rangle\lhd H_i$ where $s, t\in \Z$ are co-prime. Since $s, t$ are co-prime, $Z(H_i)/N^{s,t}_i$ is torsion-free. Hence the quotient group $\bh_i= H_i/N^{s,t}_i$ is torsion-free, since it has a composition series with torsion-free factor groups. So, $\bh_i$ is a torsion-free nilpotent group. The long exact sequence in homology for the pair $(\bg, \overline{\h})$ establishes the isomorphism $j_k: H_k(\bg ; \Z)\to H_k(\bg, \overline{\h} ; \Z)$ for all 
\begin{align*}
k&> \max\{\mbox{hd}(\bh_i)\;|\; 1\leq i\leq m\}+1\\
&=\max\{\hd(H_i)\;|\; 1\leq i\leq m\}.
\end{align*}
The above equality holds since homological dimension is additive for group extensions of finitely generated torsion-free nilpotent groups  \cite[Theorem 5.5]{bieri1981homological}, see also \cite[Theorem 7.10(a)]{bieri1981homological}. The collections $\{N^{s,t}_{i}\}_{i=1}^m$ enumerated by the co-prime pairs $(s,t)$ have the required properties.
\end{proof}

\begin{definition}\label{sufficient_deep_mfld} Let $M$ be a Riemannian $n$-manifold with a complete pinched negative sectional curvature  and  finite volume. All the cuspidal ends of $M$ are almost-flat manifolds and hence by the results of Gromov and Ruh are infra-nil \cite{gromov1978, ruh1982}. In particular, they are finitely covered by nilmanifolds. Let us assume here that all the cuspidal ends $\{L_i\}_{i=1}^m$ of $M$ are nilmanifolds such that  $rk(Z(\pi_1(L_i))\geq 2$. Let $\overline{M}$ be the natural compactification of $M$ with boundary components $\{L_i\}_{i=1}^m$. Let $G=\pi_1(\M)$ and $H_i= \pi_1(L_i)$. Then $G$ is hyperbolic relative to the fundamental groups of the cuspidal ends and in particular $\{H_{i}\}_{i=1}^m \hookrightarrow_h G$ \cite{bowditch2012relatively,farb1998relatively}. By Theorem \ref{thm. real main} (ii) and Lemma \ref{lem:nil}, there are infinitely many collections  of sufficiently deep normal subgroups $\{N_{i}\}_{i=1}^m$ satisfying the conclusions of both such that $j_n: H_n(\bg ; \Z)\xrightarrow{\cong}H_n(\bg, \overline{\h} ; \Z)$.  Let $\{N_{i}\}_{i=1}^m$ be one such collection. By \cite[Theorem 1.1]{osin2007peripheral}, we can also assume that $\{N_{i}\}_{i=1}^m$ are sufficiently deep so that $\bg$ is hyperbolic relative to $\overline{\h}$. Let 
$T_i\hookrightarrow L_i\xrightarrow{\pi_i} B_i$
 be the circle bundle with a nilmanifold base which arises as a quotient of the Malcev completion (see e.g.~\cite[\S 1.2, p.~9]{Dekimpe})  $\R\hookrightarrow \mathbf L_i \twoheadrightarrow \mathbf B_i$ of the short exact sequence 
$N_i\hookrightarrow H_i\twoheadrightarrow \bh_i$. Denote by $C(L_i, T_i)$ the mapping cylinder of $\pi_i$ for each $i$ (see Figure \ref{fig. c(l,t)} for an illustration). Note that $C(L_i, T_i)$  is manifold with boundary $L_i$. We define 
$$M_T=\overline{M} \bigcup_{\phi_1 \sqcup \dots \sqcup \phi_m} (C(L_1, T_1)\sqcup \dots \sqcup C(L_m, T_m)),$$
where each $\phi_i$ is the canonical identification of $\partial C(L_i, T_i)$ with $L_i$ and call it a {\it sufficiently deep Dehn filling} of $\M$. 
\end{definition}

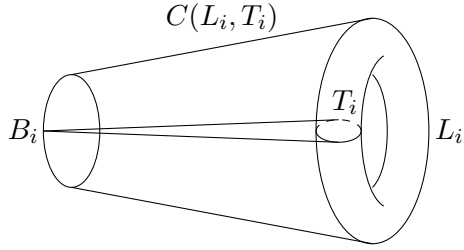
\begin{figure}[ht!]
    \centering
    \begin{tikzpicture}
    \draw (2,1.125) circle (0pt) node[anchor=south, minimum size=0.3in] {$C(L_i,T_i)$};
    \draw (0,0.75) -- (4,1.5);
    \draw (0,-0.75) -- (4,-1.5);
    \draw  (4.15,1) .. controls (3.75,0.750) and (3.75,-0.75) .. (4.15,-1);
    \draw (4,0.75) .. controls (4.2625,0.5) and (4.2625,-0.5) .. (4,-0.75) ;
    \draw (0,0) ellipse (0.375 and 0.75) node[anchor=east, minimum size=0.5in] {$B_i$};
    \draw (4,0) ellipse (0.75 and 1.5) node[anchor=west, minimum size=0.8in] {$L_i$};
    \draw (3.85,0) arc[x radius=0.3, y radius= 0.15, start angle=0, end angle=-180];
    \draw[dashed] (3.85,0) arc[x radius=0.3, y radius= 0.15, start angle=0, end angle=180] node[anchor=south west, minimum size=0.3in] {$T_i$};
    \draw (3.55,0.15) -- (-0.375,0);
    \draw (3.55,-0.15) -- (-0.375,0);
    \end{tikzpicture}
    \caption{Mapping cylinder $C(L_i,T_i)$}
    \label{fig. c(l,t)}
\end{figure}

Given a closed oriented $n$-manifold $M$ possibly with boundary $\partial M$, the {\it simplicial volume of $(M, \partial M)$} is defined as
$$||M, \partial M ||=||[M, \partial M]||_1$$ 
where $[M, \partial M]\in H_n(M, \partial M;\R)$ is the image of the fundamental class under the change of coefficients map $H_n(M, \partial M;\Z)\to H_n(M, \partial M;\R)$. The following result is again a generalisation of \cite[Theorem 1.5]{fujmann2011}.

\begin{corollary}\label{cor:pinched}
Let $\M$ be a compact oriented $n$-manifold with nilmanifold boundary components $\{L_i\}_{i=1}^m$ as above. Suppose $M_T$ is a sufficiently deep Dehn filling of $\M$. Then, $M_T$ is a closed oriented aspherical $n$-manifold with fundamental group $\bg$ and
$$0<||M_T||\leq ||\M, \partial \M||.$$ 
\end{corollary}
\begin{proof} We first show that $M_T$ is aspherical with fundamental group $\bg$.  

Let $\pi:\widetilde{M}\to M$ be the universal covering map and for each $1\leq i\leq m$, denote by $\mathbf L_i$ the universal cover  of $L_i$ which is  the Malcev completion \cite[\S 1.2]{Dekimpe}  of $H_i$. It is a simply connected nilpotent Lie group containing $H_i$ as a uniform lattice. Observe that $\M$ is homeomorphic to a submanifold $R\subset M$ where $R$ is obtained by removing the interiors of the cusps of $M$  \cite{farb1998relatively}.
 The universal cover $Y$ of $R$ is a subspace of $\widetilde{M}$ with boundary a $G$-orbit of almost-flat totally geodesic submanifolds homeomorphic to $\mathbf L_i$  for each $1\leq i \leq m$.  Let $K=Y/{\ll \mathcal N\rr}$, which is a cover of $R$ with fundamental group  ${\ll \mathcal N\rr}$.  The boundary of $K$ is a disjoint union of  $\bg$-orbits of  submanifolds homeomorphic to $L'_i=\mathbf L_i/N_i$   for each $1\leq i\leq m$. The circle bundle $T_i\hookrightarrow L_i\xrightarrow{\pi_i} B_i$ is the quotient of the short exact sequence $\R\hookrightarrow \mathbf L_i \twoheadrightarrow \mathbf B_i$ of  simply connected nilpotent Lie groups by the action of $H_i$. Hence, it lifts to the circle bundle $T_i\hookrightarrow L'_i\xrightarrow{{\pi}'_i} \mathbf B_i$. Denote by $C(L'_i, T_i)$ the mapping cylinder of ${\pi}'_i$. Define 
$$M'=K \bigcup_{\psi_1 \sqcup \dots \sqcup \psi_m} (\bg\times _{\bh_1}C(L'_1, T_1)\sqcup \dots \sqcup  \bg\times _{\bh_m}C(L'_m, T_m)),$$
where each $\psi_i$ is the canonical identification of $\bg\times _{\bh_i}\partial C(L'_i, T_i)$  with $\bg L'_i$. The manifold $M'$ is simply connected by constuction. By Mayer-Vietoris homology sequence and the Cohen-Lyndon property of Theorem \ref{thm. CL-property}, it follows $M'$ has also trivial homology and is therefore contractible. The group $\bg$ acts freely on $M'$ and the quotient  $M'/\bg$ is  homeomorphic to $M_T$.

Since $(\M, \partial \M)\simeq (BG,  \bigsqcup_{i=1}^m BH_{i})$, we have $[\M, \partial \M] \in H_n(G, \h; \R)$. Similarly, since $M_T\simeq B\bg$, we have $[M_T]\in H_n(\bg ; \R)$. Applying Theorem \ref{thm:norm}, we obtain 
$$0<||M_T||\leq ||\M, \partial \M||.\qedhere$$
\end{proof}

\begin{remark}\label{rem:simplicial} In the geometric setting of Corollary \ref{cor:pinched}, the isomorphism  
$${p_n} :H_n(G, \h ; \Z)\xrightarrow{\cong} H_n(\bg, \overline{\h} ; \Z)$$
 used in the proof of Theorem \ref{thm:norm} also follows from excising the interiors of the attached submanifolds $\{C(L_{i}, T_i)\}_{i=1}^m$ from $M_T$ which gives an isomorphism between the (co)homologies of   $(\M, \partial M)$ and $(M_T,  \{C(L_{i}, T_i)\}_{i=1}^m)$. So,  one can think of the isomorphism in Theorem \ref{thm. real main} (ii) as a group theoretic analog of topological excision.  

We should also remark that the right-hand-side inequality of Corollary \ref{cor:pinched} can be deduced from Gromov's Additivity Theorem \cite[Theorem 7.6]{frigerio2017}. \end{remark}

\section{Quotients of acylindrically hyperbolic groups}\label{quotients}

\subsection{Cohomology and embedding theorems}\label{sec. Cohomology and embedding theorems}
We prove Theorem \ref{cd SQ-universality} in this subsection. The reader is referred to Section 2.4 for an outline of the proof. In the sequel, we employ the convention that if $X$ is a set of alphabets and $w$ is a word over $X$, then $\|w\|$ denotes the length of $w$. In certain cases, it might be possible to view $w$ as a word over another alphabet $Y$. In such a case, we will use $\|w\|_X$ (resp. $\|w\|_Y$) to denote the number of letters of $X$ (resp. $Y$) in $w$. For two words $w$ and $v$, we write $w\equiv v$ to indicate that there is a letter-by-letter equality between $w$ and $v$.

\begin{lemma}\label{free group situation}
Let $F_4$ be a free group of rank $4$, $\mathcal{F}\subset F_4$ a finite set, and $C$ a countable group with $\cd(C)\geqslant 2$. Then there exists a quotient $R$ of $F_4$ such that the following hold.
\begin{enumerate}
\item[(1)] $R$ can be decomposed as a free product $R=\mathbb{Z}\ast R_0$ with $\card(R_0)=\infty$. Moreover, $C$ embeds into $R_0$.
\item[(2)] The quotient map $F_4\rightarrow R$ is injective on $\mathcal{F}$.
\item[(3)] If $C$ is torsion-free, then so is $R$.
\item[(4)] $\hd(R)\leqslant \max\{\hd(C),2\},~\cd(R)\leqslant \cd(C)$.
\item[(5)] For every $n\geqslant 3$ and every $R$-module $A$, we have 
\[ H_n(R;A)\cong H_n(C;A),~ H^n(R;A)\cong  H^n(C;A),\]
where the action of $C$ on $A$ is induced by the embedding $C\hookrightarrow R$.
\item[(6)] If $C$ is finitely generated, then $R_0$ is hyperbolic relative to $C$.
\item[(7)] If $C$ is of type $FP_n$ for some $n\in\mathbb{N}^+\cup\{\infty\}$, then so is $R$.
\item[(8)] If $C$ is of type $FP$, then so is $R$.
\end{enumerate}
\end{lemma}

\begin{remark}
Lemma \ref{free group situation} (1) (2) (6) are proved in \cite[Lemma 8.4]{dahmani2017hyperbolically}. We refine the method therein so as to impose (co)homological conditions.

The proof of Lemma \ref{free group situation} relies on small cancellation theory, the reader is referred to \cite[Chapter V]{lyndon2015combinatorial} for a treatment.
\end{remark}

\begin{proof}
Let $\{x,y,z,t\}$ be a free basis of $F_4$, let $F_3<F_4$ be the subgroup generated by $x,y,t$, and let $\{c_i\}_{i\in I}$ be a generating set of $C$. There exist freely reduced words $\{w_i\}_{i\in I},\{v_i\}_{i\in I}$ over the alphabet $\{x,y\}$ such that 
\begin{enumerate}
\item[(a)] the words $\{c_iw_i\}_{i\in I}$ satisfy the $C'(1/2)$ small cancellation condition over the free product $\langle x \rangle \ast \langle y \rangle \ast C$;
\item[(b)] the words $\{v_i\}_{i\in I}$ satisfy the $C'(1/2)$ small cancellation condition over the alphabet $\{x,y\}$;
\item[(c)] the words $\{tc_iw_it^{-1}v_i\}_{i\in I}$ satisfy the $C'(1/6)$ small cancellation condition over the free product $\langle x \rangle \ast \langle y \rangle \ast \langle t \rangle \ast C$.
\end{enumerate}
Indeed, we can first construct words $w_i$ satisfying condition (a), and then pick sufficiently long words $v_i$ to ensure conditions (b) and (c).

Let $N$ (resp. $N_0$) be the normal subgroup of $F_4\ast C$ (resp. $F_3\ast C$) generated by $\{tc_iw_it^{-1}v_i\}_{i\in I}$, and let 
$$R_0=(F_3\ast C)/N_0,~~~~R=(F_4\ast C)/N.$$
For $i\in I$, let $\overline{t}$ (resp. $\overline{c}_i,\overline{w}_i,\overline{v}_i,\overline{z}$) be the image of $t$ (resp. $c_i,w_i,v_i,z$) under the quotient map $F_4\ast C\rightarrow R$. Then we have 
$$R=\langle \overline{z} \rangle\ast R_0=\mathbb{Z}\ast R_0.$$

Note that $\overline{t}\overline{c}_i\overline{w}_i\overline{t}^{-1}\overline{v}_i=1$ and we can rewrite this equation as $\overline{c}_i=\overline{t}^{-1}\overline{v}^{-1}_i\overline{t}\overline{w}^{-1}_i$. Thus, $R$ is generated by $\overline{t},\overline{z},\overline{w}_i,\overline{v}_i,i\in I,$ and hence is a quotient of $F_4$.

Let 
$$\alpha:F_4\rightarrow R$$ 
be the corresponding quotient map, which is the restriction of the quotient map $F_4\ast C\rightarrow R$ to $F_4$. It follows from the Greendlinger's lemma for free products \cite[Chapter V Theorem 9.3]{lyndon2015combinatorial} that if $\|w_i\|,\|v_i\|,i\in I,$ are sufficiently large, then $\alpha$ is injective on $\mathcal{F}$ and thus statement (2) is guaranteed.

Let $L=\langle x \rangle \ast \langle y \rangle \ast C$, let $U\leqslant L$ be the subgroup generated by $\{c_iw_i\}_{i\in I}$, and let $V\leqslant L$ be the subgroup generated by $\{v_i\}_{i\in I}$. 

\begin{claim}\label{claim. 3}
$U$ (resp. $V$) is a free group with basis $\{c_iw_i\}_{i\in I},$ (resp. $\{v_i\}_{i\in I}$). In particular, $U$ and $V$ are both of rank $\card(I)$.
\end{claim}

\begin{proof}[Proof of Claim \ref{claim. 3}]
We prove the claim for $U$. The proof for $V$ is similar. Let 
$$u\equiv\prod^{\ell}_{k=1}(c_{i_k}w_{i_k})^{\epsilon_k}$$ 
be a nonempty freely reduced word over the alphabet $\{c_iw_i\}_{i\in I}$, where $i_k\in I$ and $\epsilon_k=\pm 1$ for $k=1,...,\ell$. Think of $u$ as a word over the alphabet $\langle x\rangle \cup \langle y \rangle \cup C$ and then reduce $u$ to its normal form $\overline{u}$ corresponding to the free product $\langle x\rangle \ast \langle y \rangle \ast C$ (see \cite[Chapter IV]{lyndon2015combinatorial} for the definition of normal forms). By condition (a) and that the words $w_i$ do not involve inverses of the generators $x,y$, for each factor $(c_{i_k}w_{i_k})^{\epsilon_k}$ of $u$, a non-empty subword of $(c_{i_k}w_{i_k})^{\epsilon_k}$ survives in $\overline{u}$. In particular, $\overline{u}$ is nonempty and thus $u$ does not represent $1$ in $L$.
\end{proof}

Note that the relations $\overline{t}\overline{c}_i\overline{w}_i\overline{t}^{-1}\overline{v}_i=1$ can be rewritten as $\overline{t}\overline{c}_i\overline{w}_i\overline{t}^{-1}=\overline{v}^{-1}_i$. Thus, $R_0$ is the HNN-extension of $L$ with associated subgroups $U$ and $V$. In particular, $L$ embeds into $R_0$. As $\card(L)=\infty$, we have $\card(R_0)=\infty$. Since $C$ embeds into $L$, $C$ embeds into $R_0$. Thus, statement (1) holds. 

If $C$ is torsion-free, then so is $L$. Being an HNN-extension of $L$, $R_0$ is also torsion-free, and thus so is $R=\mathbb{Z}\ast R_0$, which is statement (3).

By \cite[Theorem 3.1]{bieri1975mayer}, there is a long exact sequence for any $R_0$-module $A$,
\begin{equation}\label{long exact sequence for HNN}
\cdots \rightarrow  H^{n-1}(U;A)\rightarrow  H^n(R_0;A)\rightarrow  H^n(L;A)\rightarrow  H^n(U;A)\rightarrow \cdots
\end{equation}

As $U$ is free, exact sequence \eqref{long exact sequence for HNN} implies for $n\geqslant 3$,
$$ H^n(R_0;A)\cong  H^n(L;A)\cong  H^n(C;A).$$
As $R=\mathbb{Z}\ast R_0$, the cohomology part of statement (5) holds. Similarly, one can prove the homology part of statement (5). Statement (4) follows from statement (5) and $\cd(C)\geqslant 2$.

If $C$ is finitely generated, then we can construct $R$ using a finite generating set of $C$. Then $R_0$ is the quotient of $F_3\ast C$ by adding finitely many relations $tc_iw_it^{-1}v_i,$ and thus has a finite relative presentation over $C$. The Greendlinger's lemma for free products implies that the relative isoperimetric function of $R_0$ with respect to $C$ is linear. Thus, $R_0$ is hyperbolic relative to $C$ (see Remark \ref{rem. relative hyperbolic}), which is statement (6).

If $C$ is of type $FP_n$ for some $n\in\mathbb{N}^+\cup\{\infty\}$, then since $W$ is of type $FP$, we have $R_0$ is of type $FP_n$ by \cite[Proposition I.2.13(b)]{bieri1981homological}. As $R=\mathbb{Z}\ast R_0$, $R$ is of type $FP_n$, which is statement (7). Finally, statement (8) follows from (4) and (7).
\end{proof}

\begin{proof}[Proof of Theorem \ref{cd SQ-universality}]
By Theorem \ref{thm. hyperbolically embedded virtually free subgroup}, $G$ has a unique maximal finite normal subgroup $K(G)$. By \cite[Lemma 5.10]{hull2013small}, $G_0=G/K(G)$ is acylindrically hyperbolic.

If $\cd(C)=0$, then $C=\{1\}$. Let $\bg=G_0$. By Theorem \ref{thm. hyperbolically embedded virtually free subgroup}, $C\hookrightarrow_h \bg$. First consider statement (vi). As $\bg$ and $G$ are quasi-isometric (note that the assumption of (vi) implies that $G$ and $\bg$ are finitely generated), \cite[Corollary 9]{alonso1994finiteness} implies (vi). Other conclusions of Theorem \ref{cd SQ-universality} hold trivially.

If $\cd(C)=1$, then by the Stallings-Swan theorem \cite[corollary to Theorem 1]{swan1969groups}, $C$ is free. By Theorem \ref{thm. hyperbolically embedded virtually free subgroup}, there exists a finitely generated non-cyclic free group $F\hookrightarrow_h G_0$. Let $\bg=G_0$. It is well-known that the free group $C$ embeds into $F$. Thus, $C$ also embeds into $\bg$. All conclusions except for (ii) hold trivially. If in addition, $C$ is finitely generated, then $C$ is a finite rank free group and we can let $F=C$. Thus, (ii) also holds.

Let us assume $\cd(C)\geqslant 2$. By Theorem \ref{thm. hyperbolically embedded virtually free subgroup}, there exists $X\subset G_0$ and a free subgroup $F_4\leqslant G_0$ of rank $4$ such that $F_4\hookrightarrow_h (G_0,X)$. There exists a finite set $\mathcal{F}\subset F_4\smallsetminus\{1\}$ such that if $N\lhd F_4$ satisfies $N\cap \mathcal{F}=\emptyset$, then the conclusions of Theorems \ref{thm. main} and \ref{thm. group theoretic Dehn filling} and \cite[Theorem 7.15]{dahmani2017hyperbolically} hold.

By Lemma \ref{free group situation}, $C$ embeds into an infinite quotient $R=\mathbb{Z}\ast R_0$ of $F_4$ such that the conclusions of Lemma \ref{free group situation} hold and the quotient map $F_4\rightarrow R$ is injective on $\mathcal{F}$. Let $N$ be the kernel of $F_4\rightarrow R$. Then $N\cap\mathcal{F}=\emptyset$. Let $\bg=G/\ll N \rr$.

As $R=\mathbb{Z}\ast R_0$, $R_0$ is a proper subgroup of $R$ and in particular, $R_0$ is a proper subgroup of $G$. By Example \ref{eg. product hyperbolically embedded} (d), $R_0\hookrightarrow_h R$. Proposition \ref{prop. iterative hyperbolically embedded} and $R\hookrightarrow_h G$ then imply that $R_0\hookrightarrow_h G$. As $\card(R_0)=\infty$, Theorem \ref{thm. group theoretic Dehn filling} implies that $\bg$ is acylindrically hyperbolic, that is, statement (i) holds. As $C$ embeds into $R_0$, $C$ also embeds into $\bg$.

If $C$ is finitely generated, then Lemma \ref{free group situation} implies that $R_0$ is hyperbolic relative to $C$, in particular, $C\hookrightarrow_h R_0$. As $R_0\hookrightarrow_h \bg$, we have $C\hookrightarrow_h \bg$ by Proposition \ref{prop. iterative hyperbolically embedded}. Thus, statement (ii) holds. 

Suppose that $G$ and $C$ are torsion-free and there is a non-trivial finite-order element $\overline{g}\in\bg$. Then $G_0=G$. Denote the image of $X$ under the quotient map $G\twoheadrightarrow\bg$ by $\overline{X}$. Then $R\hookrightarrow_h(\bg,\overline{X})$ \cite[Theorem 7.15 (b)]{dahmani2017hyperbolically}. As $\overline{g}$ has finite order, it acts elliptically on the Cayley graph $\Gamma(\bg,\overline{X}\sqcup R)$. By \cite[Theorem 7.15 (f)]{dahmani2017hyperbolically}, there is an element $g\in G$ such that $g$ is mapped to $\overline{g}$ under the quotient map $G\twoheadrightarrow\bg$ and $g$ acts elliptically on $\Gamma(G,X\sqcup F_4)$. As $\overline{g}$ has finite order, $g^n\in\ll N \rr$ for some $n>0$. Since $g^n$ is elliptic as $g$ is, there is some $h\in G$ such that $hg^nh^{-1}\in N\leqslant F_4$ \cite[Theorem 7.15 (d)]{dahmani2017hyperbolically}. By Lemma \ref{free group situation} (3), $R$ is torsion free and thus $hgh^{-1}\not\in F_4$. But
\[|(hgh^{-1})F_4(hgh^{-1})^{-1}\cap F_4|\geqslant |\langle hg^nh^{-1} \rangle|=\infty,\]
which is in contradiction with the almost malnormality of $F_4$ in $G$ \cite[Proposition 2.10]{dahmani2017hyperbolically}. We have proved statement (iii).

Statement (iv) follows from Corollary \ref{intro_cor:finiteness} (i).

Consider statement (v). Corollary \ref{intro_cor:finiteness} (ii) implies that
\[\cd(\bg)\leqslant \max\{\cd(G_0),\cd(F_4)+1,\cd(R)\}.\]
If $K(G)\neq \{1\}$, then $G$ has torsion and thus $\cd(G)=\infty$ \cite[Chapter VIII Corollary 2.5]{brown1982cohomology}, in which case (v) is a void statement. Thus, let us assume $K(G)=\{1\}$ and thus $G_0=G$. As $\cd(R)\leqslant \cd(C)$ and $\cd(C)\geqslant 2$, we have
$$\cd(\bg)\leqslant \max\{\cd(G),\cd(F_4)+1,\cd(R)\}\leqslant \max\{\cd(G),2,\cd(C)\}=\max\{\cd(G),\cd(C)\}.$$

If $G$ and $C$ are of type $FP_n$ for some $n\in\mathbb{N}^+\cup\{\infty\}$, then Lemma \ref{free group situation} implies that so is $R$. As $G$ and $G_0$ are quasi-isometric, $G_0$ is of type $FP_n$ \cite[Corollary 9]{alonso1994finiteness}. As $F_4$ is of type $FP_{\infty}$, Corollary \ref{intro_cor:finiteness} (iii) implies that $\bg$ is of type $FP_n$. Conversely, if $C$ is finitely generated and $G$ and $\bg$ are of type $FP_n$ for some $n\in\mathbb{N}^+\cup\{\infty\}$, then $C\hookrightarrow_h \bg$ by the previously proved statement (ii). Thus, $C$ is of type $FP_n$ \cite[Theorem 2.11]{dahmani2017hyperbolically}. The previously proved statement (v) implies that $\cd(\bg)<\infty$ if and only if $\max\{\cd(G),\cd(C)\}<\infty$, which further implies the statement about type $FP$. This finishes the proof of statement (vi).
\end{proof}

\begin{remark}\label{rm. homological analog of D (v)}
Analogous to the above proof, in the setting of Theorem \ref{cd SQ-universality}, we have $\hd(\bg)\leqslant \max\{\hd(G),\hd(C),2\}$.
\end{remark}

\subsection{Constructing hyperbolically embedded subgroups}\label{sec. construct h.e. subgroups}
We will prove Theorem \ref{thm. common quotient of acylindrically hyperbolic groups} in the next subsection, where the following technical result will be needed. The reader is referred to Section 2.4 for an outline of the proof.

\begin{proposition}\label{prop. showing h.e. general}
Suppose that $G$ is a group, $\Lambda$ is a finite set and $\{k_{\lambda}\}_{\lambda\in \Lambda}$ is a set of positive integers. Let $a_{\lambda,i}\in G$ for $\lambda\in \Lambda$ and $i=1,2,\cdots,k_{\lambda}$. Also suppose that there is a family of subgroups $\{F_{\lambda}\}_{\lambda\in\Lambda}\hookrightarrow_h G$ such that
\begin{enumerate}
    \item[(a)] each $F_{\lambda}$ is free of rank $2k_{\lambda}$ with basis $\{f_{\lambda,i},g_{\lambda,i}\}^{k_{\lambda}}_{i=1}$; and
    \item[(b)] $a_{\lambda,i}\not\in F_{\lambda}\smallsetminus\{1\}$ for all $\lambda\in \Lambda$ and $i=1,2,\cdots,k_{\lambda}$.
\end{enumerate}
Then for sufficiently large $n\in\mathbb{N}^+$, the set
$$\{f^n_{\lambda,i}a_{\lambda,i}g^n_{\lambda,i}\}^{k_{\lambda}}_{i=1}\subset G$$
freely generates a subgroup $H_{\lambda}\leqslant G$ and
\begin{equation}\label{eq. H is h.e. in G}
    \{H_{\lambda}\}_{\lambda\in\Lambda}\hookrightarrow_h G.
\end{equation}
\end{proposition}

The rest of this subsection is devoted to the proof of the above proposition. We first find a ``sufficiently large'' $n\in\mathbb{N}^+$. As $F_{\lambda}=\Asterisk^{k_{\lambda}}_{i=1}(\langle f_{\lambda,i}\rangle\ast\langle g_{\lambda,i}\rangle)$, we have $\{\langle f_{\lambda,i}\rangle,\langle g_{\lambda,i}\rangle\}^{k_{\lambda}}_{i=1}\hookrightarrow_h F_{\lambda}$ by Example \ref{eg. product hyperbolically embedded}. It follows from Proposition \ref{prop. iterative hyperbolically embedded} that there exists a set $X\subset G$ such that
\begin{equation}\label{eq. he embedded free factors}
    \{\langle f_{\lambda,i}\rangle,\langle g_{\lambda,i}\rangle\}_{\lambda\in\Lambda,i\in\{1,...,k_{\lambda}\}}\hookrightarrow_h (G,X).
\end{equation}
By \cite[Corollary 4.27]{dahmani2017hyperbolically}, we may assume that $a_{\lambda,i} \in X$ for all $\lambda\in \Lambda$ and $i=1,2,\cdots,k_{\lambda}$, as $\Lambda$ and $k_{\lambda}$ are finite.

For $\lambda\in \Lambda$ and $i\in\{1,...,k_{\lambda}\}$, let
$$\widehat{d}_{\lambda,i,f}: \langle f_{\lambda,i} \rangle\times \langle f_{\lambda,i} \rangle\rightarrow[0,+\infty],~~~~\widehat{d}_{\lambda,i,g}: \langle g_{\lambda,i} \rangle\times \langle g_{\lambda,i} \rangle\rightarrow[0,+\infty]$$ 
be the relative metrics corresponding to \eqref{eq. he embedded free factors}. The metrics $\widehat{d}_{\lambda,i,f}$ and $\widehat{d}_{\lambda,i,g}$ are locally finite. As $\card(\Lambda),k_{\lambda}<\infty$, for sufficiently large $n$, we will have 
$$\widehat{d}_{\lambda,i,f}(1,f^n_{\lambda,i}),\widehat{d}_{\lambda,i,g}(1,g^n_{\lambda,i})>50D$$
for all $\lambda$ and $i$, where $D>0$ is given by Lemma \ref{lem. total length of i.c. in g.p.}. We fix one such $n$ and let $H_{\lambda}\leqslant G$ be as in Proposition \ref{prop. showing h.e. general}. For simplicity, denote the set $\{f^n_{\lambda,i}a_{\lambda,i}g^n_{\lambda,i}\}^{k_{\lambda}}_{i=1}$ by $U_{\lambda}$.

Notice that $F_{\lambda}\cap F_{\mu}=\{1\}$ whenever $\lambda\neq\mu$ \cite[Proposition 4.33]{dahmani2017hyperbolically}. The following are easy consequences of this fact.
\begin{align}
      \langle f_{\lambda,i}\rangle \cap \langle g_{\mu,j} \rangle &= \{1\}&&\text{for all } \lambda,\mu\in\Lambda,&\label{eq. subgroups disjoint 1}\\
      \langle f_{\lambda,i}\rangle \cap \langle f_{\mu,j}\rangle &= \langle g_{\lambda,i}\rangle \cap \langle g_{\mu,j}\rangle = \{1\} &&\text{for all }\lambda,\mu\in\Lambda\text{ with }\lambda\neq \mu,&\label{eq. subgroups disjoint 2}\\
      \langle f_{\lambda,i} \rangle \cap \langle f_{\lambda,j} \rangle &=
      \langle g_{\lambda,i} \rangle \cap \langle g_{\lambda,j} \rangle =\{1\}&&\text{for all }\lambda\in\Lambda\text{ and } i,j\in\{1,...,k_{\lambda}\} \text{ with } i\neq j.&\label{eq. subgroups disjoint 3}
\end{align}

In particular, $\{\langle f_{\lambda,i}\rangle,\langle g_{\lambda,i}\rangle\}_{\lambda\in\Lambda,i\in\{1,...,k_{\lambda}\}}$ is a distinct family of hyperbolically embedded subgroups of $G$. 

For simplicity, let 
$$\mathcal{K}_{\lambda}=\left(\bigsqcup^{k_{\lambda}}_{i=1}\langle f_{\lambda,i} \rangle\right)\sqcup \left(\bigsqcup^{k_{\lambda}}_{i=1}\langle g_{\lambda,i} \rangle\right),~~~~\mathcal{K}=\bigsqcup_{\lambda\in\Lambda}\mathcal{K}_{\lambda}.$$

\begin{remark}\label{rem. think of w as a word over}
We will apply Lemma \ref{lem. consecutive components} for the group $G$, the hyperbolically embedded family of subgroups $\{\langle f_{\lambda,i}\rangle,\langle g_{\lambda,i}\rangle\}_{\lambda\in\Lambda,i\in\{1,...,k_{\lambda}\}}\hookrightarrow_h (G,X)$, and reduced words $w$ over $U_{\lambda}$ for every $\lambda\in\Lambda$ and every freely reduced word $w$ over $U_{\lambda}$. We can think of $w$ as a word over the alphabet $X\sqcup \mathcal{K}$, i.e., regard every $f^n_{\lambda,i}$ (resp. $g^n_{\lambda,i}$) as a letter from $\langle f_{\lambda,i}\rangle$ (resp. $\langle g_{\lambda,i}\rangle$) and regard every $a_{\lambda,i}$ as a letter from $X$. In this sense, $w$ satisfies the conditions (W1), (W2), and (W3) of Lemma \ref{lem. consecutive components}.
\end{remark}

\begin{lemma}\label{lem. H free groups}
For all $\lambda\in \Lambda$, $H_{\lambda}$ is free with basis $U_{\lambda}$.
\end{lemma}

\begin{proof}
We need to show that for every $\lambda\in\Lambda$, every nonempty freely reduced word over $U_{\lambda}$ does not represent $1$ in $G$. Suppose that 
$$w\equiv\prod^{\ell}_{k=1} (f^n_{\lambda,i_k}a_{\lambda,i_k}g^n_{\lambda,i_k})^{\epsilon_k}$$
is a freely reduced word over $U_{\lambda}$ for some $\lambda\in\Lambda$ such that $w$ represents $1$ in $G$, where $\epsilon_k=\pm 1$ for $k=1,...,\ell$. As Remark \ref{rem. think of w as a word over}, we think of $w$ as a word over $X\sqcup \mathcal{K}$. Then $w$ labels a $3\ell$-gon $p$ in $\Gamma(G,X\sqcup\mathcal{K})$ with geodesic sides. Notice that $p$ has $2\ell$ components. By Lemma \ref{lem. consecutive components} (b), each of these components is isolated. Proposition \ref{lem. total length of i.c. in g.p.} then implies
$$2n\cdot 50D\leqslant 3nD,$$
which is absurd. Therefore, such a word $w$ does not exist.
\end{proof}

Consider the action $G\curvearrowright\Gamma(G,X\sqcup \mathcal{K})$. Note that $\Gamma(G,X\sqcup \mathcal{K})$ is hyperbolic, by \eqref{eq. he embedded free factors}. Let $d_{X\sqcup \mathcal{K}}$ be the combinatorial metric of $\Gamma(G,X\sqcup \mathcal{K})$. We verify that, with respect to this action, the family $\{H_{\lambda}\}_{\lambda\in\Lambda}$ satisfies the conditions (C$_1$), (C$_2$), and (C$_3$) of Lemma \ref{lem. criterion for he}, which then implies $\{H_{\lambda}\}_{\lambda\in\Lambda}\hookrightarrow_h G$. The verification is divided into the following Lemmas \ref{lem. H action proper}, \ref{lem. H orbit quasi convex}, and \ref{lem. geometrically separated}.

\begin{lemma}\label{lem. H action proper}
For every $\lambda\in\Lambda$, the action $H_{\lambda}\curvearrowright\Gamma(G,X\sqcup \mathcal{K})$ is proper.
\end{lemma}

\begin{proof}
Fix $\lambda\in\Lambda$. It suffices to prove that for every $R>0$, there are only finitely many $h\in H_{\lambda}$ such that $d_{X\sqcup \mathcal{K}}(1,h)\leqslant R$. Let $h\in H_{\lambda}$ such that $d_{X\sqcup \mathcal{K}}(1,h)\leqslant R$ and let $w$ be a freely reduced word over $U_{\lambda}$ representing $h$ in $G$. As in Remark \ref{rem. think of w as a word over}, think of $w$ as a word over $X\sqcup\mathcal{K}$. By Lemma \ref{lem. consecutive components} (a), $w$ labels a $(4,1)$-quasi-geodesic in $\Gamma(G,X\sqcup \mathcal{K})$. Thus, 
$$\|w\|_{U_{\lambda}}=\dfrac{\|w\|_{X\sqcup\mathcal{K}}}{3}\leqslant \dfrac{4R+1}{3}.$$
There are only finitely many words $w$ satisfying the above inequality. It follows that the number of $h\in H_{\lambda}$ such that $d_{X\sqcup \mathcal{K}}(1,h)\leqslant R$ is finite.
\end{proof}

For every $\lambda\in\Lambda$, we identify $H_{\lambda}$ with the subset of $\Gamma(G,X\sqcup \mathcal{K})$ labeled by elements of $H_{\lambda}$. Equivalently, we identify $H_{\lambda}$ with the $H_{\lambda}$-orbit of the identity vertex of $\Gamma(G,X\sqcup \mathcal{K})$. 

\begin{lemma}\label{lem. H orbit quasi convex}
For every $\lambda\in\Lambda$, the orbit $H_{\lambda}$ is quasi-convex in $\Gamma(G,X\sqcup \mathcal{K})$.
\end{lemma}

\begin{proof}
Fix $\lambda\in\Lambda$. Let $h\in H_{\lambda}$ and let $\gamma$ be a geodesic in $\Gamma(G,X\sqcup \mathcal{K})$ from the vertex $1$ to the vertex $h$. As $\Gamma(G,X\sqcup \mathcal{K})$ is a Gromov hyperbolic space, there exists $R>0$ such that if $\alpha$ and $\beta$ are $(4,1)$-quasi-geodesics with the same endpoint, then $d_{Hau}(\alpha,\beta)\leqslant R$, where $d_{Hau}$ is the Hausdorff metric with respect to $d_{X\sqcup \mathcal{K}}$.

Let $w$ be a freely reduced word over $U_{\lambda}$ representing $h$ in $G$. As in Remark \ref{rem. think of w as a word over}, think of $w$ as a word over $X\sqcup\mathcal{K}$. By Lemma \ref{lem. consecutive components} (a), $w$ labels a $(4,1)$-quasi-geodesic $\alpha$ in $\Gamma(G,X\sqcup \mathcal{K})$. Note that $\alpha$ lies in the $2$-neighborhood of the orbit $H_{\lambda}$, and $\gamma$ lies in the $R$-neighborhood of $\alpha$. Thus, $\gamma$ lies in the $(R+2)$-neighborhood of $H_{\lambda}$.
\end{proof}

For $\lambda,\mu\in\Lambda$, the orbits $H_{\lambda}$ and $H_{\mu}$ are subsets of $\Gamma(G,X\sqcup \mathcal{K})$. Thus, it makes sense to talk about the diameter of $H_{\mu}\cap (gH_{\lambda})^{+\epsilon}$ in $\Gamma(X\sqcup \mathcal{K})$, which is denoted by $\diam\left(H_{\mu}\cap \left(gH_{\lambda}\right)^{+\epsilon}\right)$.

\begin{lemma}\label{lem. geometrically separated}
For every $\epsilon>0$, there exists $R>0$ such that the following holds. Suppose that for some $g\in G$ and $\lambda,\mu\in\Lambda$, we have
$$\diam\left(H_{\mu}\cap \left(gH_{\lambda}\right)^{+\epsilon}\right)\geqslant R.$$
Then $\lambda=\mu$ and $g\in H_{\lambda}$.
\end{lemma}

\begin{proof}
Fix $\epsilon>0$ and let $R$ be the constant given by Lemma \ref{lem. consecutive components} (c). Suppose that $\diam\left(H_{\mu}\cap \left(gH_{\lambda}\right)^{+\epsilon}\right)\geqslant R$ for some $g\in G$ and $\lambda,\mu\in\Lambda$. Then there exist vertices $v_1,v_2\in H_{\mu}$ and $v_3,v_4\in gH_{\lambda}$ such that $$d(v_1,v_2)\geqslant R,~~~~ d(v_1,v_3),d(v_2,v_4)\leqslant \epsilon.$$
Let $p$ (resp. $q$) be a path from $v_1$ (resp. $v_3$) to $v_2$ (resp. $v_4$) such that $\lab(p)$ (resp. $\lab(q)$) is a freely reduced word over $U_{\mu}$ (resp. $U_{\lambda}$). Then $\ell(p)\geqslant R$ and $p,q$ are oriented $\epsilon$-close. By Lemma \ref{lem. consecutive components} (b), there exist five consecutive components of $p$ which are connected to five consecutive components of $q$. In particular, there exist two pairs of adjacent components of $p$ which are connected to four consecutive components of $q$. Some of the possible configurations of these two pairs of adjacent components are shown by Figure \ref{fig. cases}, where each horizontal line represents one possible configuration, the red and blue segments represent the two pairs of adjacent components of $p$, and the corresponding labels are written on top of the subpaths.


\begin{figure}
    \centering
    {\small{
    \begin{tikzpicture}
\draw[thick] (0,0) -- (2,0);
\draw[red, thick] (2,0) -- (4,0);
\draw[blue, thick] (4,0) -- (6,0);
\draw[yellow, thick] (6,0) -- (8,0);
\draw[red, thick] (8,0) -- (10,0);
\draw[blue, thick] (10,0) -- (12,0);
\draw[thick] (12,0) -- (14,0);

\draw (11,0) circle (0pt) node[anchor=south] {$g^{-n}_{\mu,r}$};
\draw (9,0) circle (0pt) node[anchor=south] {$f^{-n}_{\mu,j}$};
\draw (7,0) circle (0pt) node[anchor=south] {$a^{-1}_{\mu,j}$};
\draw (5,0) circle (0pt) node[anchor=south] {$g^{-n}_{\mu,j}$};
\draw (3,0) circle (0pt) node[anchor=south] {$g^{n}_{\mu,i}$};

\draw[thick] (0,1) -- (2,1);
\draw[red, thick] (2,1) -- (4,1);
\draw[blue, thick] (4,1) -- (6,1);
\draw[yellow, thick] (6,1) -- (8,1);
\draw[red, thick] (8,1) -- (10,1);
\draw[blue, thick] (10,1) -- (12,1);
\draw[thick] (12,1) -- (14,1);

\draw (11,1) circle (0pt) node[anchor=south] {$f^{n}_{\mu,r}$};
\draw (9,1) circle (0pt) node[anchor=south] {$f^{-n}_{\mu,j}$};
\draw (7,1) circle (0pt) node[anchor=south] {$a^{-1}_{\mu,j}$};
\draw (5,1) circle (0pt) node[anchor=south] {$g^{-n}_{\mu,j}$};
\draw (3,1) circle (0pt) node[anchor=south] {$g^{n}_{\mu,i}$};

\draw[thick] (0,2) -- (2,2);
\draw[red, thick] (2,2) -- (4,2);
\draw[blue, thick] (4,2) -- (6,2);
\draw[yellow, thick] (6,2) -- (8,2);
\draw[red, thick] (8,2) -- (10,2);
\draw[blue, thick] (10,2) -- (12,2);
\draw[thick] (12,2) -- (14,2);

\draw (11,2) circle (0pt) node[anchor=south] {$f^{n}_{\mu,r}$};
\draw (9,2) circle (0pt) node[anchor=south] {$g^{n}_{\mu,j}$};
\draw (7,2) circle (0pt) node[anchor=south] {$a_{\mu,j}$};
\draw (5,2) circle (0pt) node[anchor=south] {$f^{n}_{\mu,j}$};
\draw (3,2) circle (0pt) node[anchor=south] {$g^{n}_{\mu,i}$};

\draw[thick] (0,3) -- (2,3);
\draw[red, thick] (2,3) -- (4,3);
\draw[blue, thick] (4,3) -- (6,3);
\draw[yellow, thick] (6,3) -- (8,3);
\draw[red, thick] (8,3) -- (10,3);
\draw[blue, thick] (10,3) -- (12,3);
\draw[thick] (12,3) -- (14,3);

\draw (11,3) circle (0pt) node[anchor=south] {$f^{n}_{\mu,r}$};
\draw (9,3) circle (0pt) node[anchor=south] {$f^{-n}_{\mu,j}$};
\draw (7,3) circle (0pt) node[anchor=south] {$a^{-1}_{\mu,j}$};
\draw (5,3) circle (0pt) node[anchor=south] {$g^{-n}_{\mu,j}$};
\draw (3,3) circle (0pt) node[anchor=south] {$f^{-n}_{\mu,i}$};

\draw[thick] (0,4) -- (2,4);
\draw[red, thick] (2,4) -- (4,4);
\draw[blue, thick] (4,4) -- (6,4);
\draw[yellow, thick] (6,4) -- (8,4);
\draw[red, thick] (8,4) -- (10,4);
\draw[blue, thick] (10,4) -- (12,4);
\draw[thick] (12,4) -- (14,4);

\draw (11,4) circle (0pt) node[anchor=south] {$g^{-n}_{\mu,r}$};
\draw (9,4) circle (0pt) node[anchor=south] {$g^{n}_{\mu,j}$};
\draw (7,4) circle (0pt) node[anchor=south] {$a_{\mu,j}$};
\draw (5,4) circle (0pt) node[anchor=south] {$f^{n}_{\mu,j}$};
\draw (3,4) circle (0pt) node[anchor=south] {$f^{-n}_{\mu,i}$};
\end{tikzpicture}
}}
    \caption{Some of the possible configurations of the two pairs of adjacent components of $p$}
    \label{fig. cases}
\end{figure}
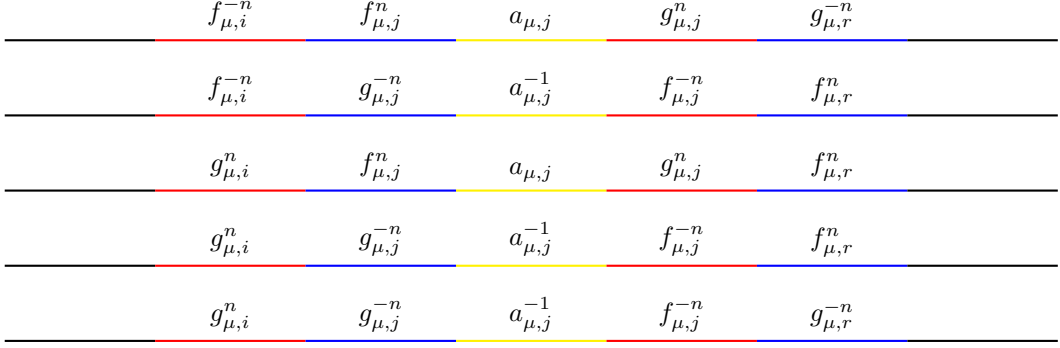

Below, we assume, without loss of generality, that these two pairs of adjacent components are of the form $$f^{-n}_{\mu,i}g^{-n}_{\mu,j},f^{-n}_{\mu,j}g^{-n}_{\mu,r}.$$ 
Other possible configurations can be analyzed similarly. We distinguish two cases.

\textit{Case 1.} The first pair of adjacent components of $p$ are respectively connected to a pair of adjacent components of $q$.


\begin{figure}[!ht]
    \centering
    {\small{
    \begin{tikzpicture}
\draw[thick] (0,0) -- (2,0);
\draw[red, thick] (2,0) -- (4,0);
\draw[blue, thick] (4,0) -- (6,0);
\draw[yellow, thick] (6,0) -- (8,0);
\draw[red, thick] (8,0) -- (10,0);
\draw[blue, thick] (10,0) -- (12,0);
\draw[thick] (12,0) -- (14,0);

\draw (11,2) circle (0pt) node[anchor=south] {$g^{-n}_{\mu,r}$};
\draw (9,2) circle (0pt) node[anchor=south] {$f^{-n}_{\mu,j}$};
\draw (7,2) circle (0pt) node[anchor=south] {$a^{-1}_{\mu,j}$};
\draw (5,2) circle (0pt) node[anchor=south] {$g^{-n}_{\mu,j}$};
\draw (3,2) circle (0pt) node[anchor=south] {$f^{-n}_{\mu,i}$};
\draw (1,2) circle (0pt) node[anchor=south] {$u$};
\draw (1,0) circle (0pt) node[anchor=north] {$v$};

\draw (5,0) circle (0pt) node[anchor=north] {$g^{-n}_{\mu,j}$};
\draw (3,0) circle (0pt) node[anchor=north] {$f^{-n}_{\mu,i}$};

\draw (3.5,1) circle (0pt) node[anchor=east] {$1$};
\draw (4.5,1) circle (0pt) node[anchor=west] {$1$};

\draw[thick] (0,2) -- (2,2);
\draw[red, thick] (2,2) -- (4,2);
\draw[blue, thick] (4,2) -- (6,2);
\draw[yellow, thick] (6,2) -- (8,2);
\draw[red, thick] (8,2) -- (10,2);
\draw[blue, thick] (10,2) -- (12,2);
\draw[thick] (12,2) -- (14,2);

\draw[red, thick, dash pattern={on 7pt off 2pt}] (4,2) .. controls (3.5,1.8) and (3.5,0.2) .. (4,0);
\draw[blue, thick,dash pattern={on 7pt off 2pt}] (4,2) .. controls (4.5,1.8) and (4.5,0.2) .. (4,0);

\filldraw[black] (2,0) circle (1pt);
\filldraw[black] (4,0) circle (1pt);
\filldraw[black] (6,0) circle (1pt);
\filldraw[black] (8,0) circle (1pt);
\filldraw[black] (10,0) circle (1pt);
\filldraw[black] (12,0) circle (1pt);

\filldraw[black] (2,2) circle (1pt);
\filldraw[black] (4,2) circle (1pt);
\filldraw[black] (6,2) circle (1pt);
\filldraw[black] (8,2) circle (1pt);
\filldraw[black] (10,2) circle (1pt);
\filldraw[black] (12,2) circle (1pt);
\end{tikzpicture}}}
    \caption{Case 1}
    \label{fig. adjacent also in q}
\end{figure}

Case 1 is displayed by Figure \ref{fig. adjacent also in q}, where the red (resp. blue) dashed line represents a path with label in $\langle f^n_{\mu,i}\rangle$ (resp. $\langle g^n_{\mu,j}\rangle$) connecting the corresponding red (resp. blue) components. Equations \eqref{eq. subgroups disjoint 1}, \eqref{eq. subgroups disjoint 2}, and \eqref{eq. subgroups disjoint 3} imply that $\lambda=\mu$ and the red (resp. blue) component of $q$ is labeled by $f^{-n}_{\mu,i}$ (resp. $g^{-n}_{\mu,j}$). As the red and blue dashed lines form a loop, another consequence of \eqref{eq. subgroups disjoint 1} is that both of these dashed lines are labeled by $1$. 

Let $p_1$ (resp. $q_1$) be the subpath of $p$ (resp. $q$) labeled by $uf^{-n}_{\mu,i}$ (resp. $vf^{-n}_{\mu,i}$). Then $p^+_1=q^+_1$. By the structure of $U_{\mu}$, $\lab(p_1)$ and $\lab(q_1)$ are words over $U_{\mu}$ and thus represent elements in $H_{\mu}$. Therefore, $p^+_1\in H_{\mu}\cap gH_{\mu}$. It follows that $g\in H_{\mu}$.

\textit{Case 2.} The first pair of adjacent components of $p$ are respectively connected to two consecutive, but not adjacent, components of $q$.


\begin{figure}[!ht]
    \centering
    {\small{
    \begin{tikzpicture}
\draw[thick] (0,0) -- (14/8,0);
\draw[red, thick] (14/8,0) -- (28/8,0);
\draw[yellow, thick] (28/8,0) -- (42/8,0);
\draw[blue, thick] (42/8,0) -- (56/8,0);
\draw[red, thick] (56/8,0) -- (70/8,0);
\draw[yellow, thick] (70/8,0) -- (84/8,0);
\draw[blue, thick] (84/8,0) -- (98/8,0);
\draw[thick] (98/8,0) -- (14,0);

\draw[thick] (0,2) -- (2,2);
\draw[red, thick] (2,2) -- (4,2);
\draw[blue, thick] (4,2) -- (6,2);
\draw[yellow, thick] (6,2) -- (8,2);
\draw[red, thick] (8,2) -- (10,2);
\draw[blue, thick] (10,2) -- (12,2);
\draw[thick] (12,2) -- (14,2);

\draw[red, thick, dash pattern={on 7pt off 2pt}] (4,2) -- (28/8,0);
\draw[blue, thick, dash pattern={on 7pt off 2pt}] (4,2) -- (42/8,0);
\draw[red, thick, dash pattern={on 7pt off 2pt}] (10,2) -- (70/8,0);
\draw[blue, thick, dash pattern={on 7pt off 2pt}] (10,2) -- (84/8,0);

\draw (11,2) circle (0pt) node[anchor=south] {$g^{-n}_{\mu,r}$};
\draw (9,2) circle (0pt) node[anchor=south] {$f^{-n}_{\mu,j}$};
\draw (7,2) circle (0pt) node[anchor=south] {$a^{-1}_{\mu,j}$};
\draw (5,2) circle (0pt) node[anchor=south] {$g^{-n}_{\mu,j}$};
\draw (3,2) circle (0pt) node[anchor=south] {$f^{-n}_{\mu,i}$};
\draw (1,2) circle (0pt) node[anchor=south] {$u$};

\draw (182/16,0) circle (0pt) node[anchor=north] {$g^n_{\mu,i}$};
\draw (154/16,0) circle (0pt) node[anchor=north] {$a_{\mu,i}$};
\draw (126/16,0) circle (0pt) node[anchor=north] {$f^n_{\mu,i}$};
\draw (98/16,0) circle (0pt) node[anchor=north] {$g^n_{\mu,i}$};
\draw (70/16,0) circle (0pt) node[anchor=north] {$a_{\mu,i}$};
\draw (42/16,0) circle (0pt) node[anchor=north] {$f^n_{\mu,i}$};
\draw (14/16,0) circle (0pt) node[anchor=north] {$v$};

\filldraw[black] (2,2) circle (1pt);
\filldraw[black] (4,2) circle (1pt);
\filldraw[black] (6,2) circle (1pt);
\filldraw[black] (8,2) circle (1pt);
\filldraw[black] (10,2) circle (1pt);
\filldraw[black] (12,2) circle (1pt);

\filldraw[black] (14/8,0) circle (1pt);
\filldraw[black] (28/8,0) circle (1pt);
\filldraw[black] (42/8,0) circle (1pt);
\filldraw[black] (56/8,0) circle (1pt);
\filldraw[black] (70/8,0) circle (1pt);
\filldraw[black] (84/8,0) circle (1pt);
\filldraw[black] (98/8,0) circle (1pt);
\end{tikzpicture}}}
    \caption{Case 2}
    \label{fig. not adjacent in q}
\end{figure}

Case 2 is displayed by Figure \ref{fig. not adjacent in q}. Once again, Equations \eqref{eq. subgroups disjoint 1}, \eqref{eq. subgroups disjoint 2}, and \eqref{eq. subgroups disjoint 3} imply $\lambda=\mu$. The structures of $U_{\mu}$ imply $i=j=r$. The red (resp. blue) dashed line on the left is labeled by an element in $\langle f^n_{\mu,i} \rangle$ (resp. $\langle g^n_{\mu,i} \rangle$). As these dashed lines and the yellow segment labeled by $a_{\mu,i}$ form a loop, assumption (b) of Proposition \ref{prop. showing h.e. general} implies that both of these dashed lines are labeled by $1$. Similarly, the red and blue dashed lines on the right are both labeled by $1$.

Therefore, the word $g^{2n}_{\mu,i}f^{2n}_{\mu,i}a_{\mu,i}$ labels a loop in $\Gamma(G,X\sqcup\mathcal{K})$ and thus represents $1$ in $G$, which is in contradiction with assumption (b) of Proposition \ref{prop. showing h.e. general}. Hence, Case 2 is in fact impossible.
\end{proof}

\begin{proof}[Proof of Proposition \ref{prop. showing h.e. general}]
The first assertion follows from Lemma \ref{lem. H free groups} and formula \eqref{eq. H is h.e. in G} follows Lemmas \ref{lem. criterion for he}, \ref{lem. H action proper}, \ref{lem. H orbit quasi convex}, and \ref{lem. geometrically separated}.
\end{proof}

\subsection{Common quotients of acylindrically hyperbolic groups}\label{sec. common quotient}

In this subsection, we prove Theorem \ref{thm. common quotient of acylindrically hyperbolic groups}. Given finitely generated acylindrically hyperbolic groups $G_1$ and $G_2$, we construct a common quotient $G$ of $G'_1=G_1/K(G_1)$ and $G'_2=G_2/K(G_2)$ satisfying the conclusions of that theorem, where $K(G_1)$ (resp. $K(G_2)$) is the maximal finite normal subgroup of $G_1$ (resp. $G_2$).

The idea is to consider $\widetilde{G}=G'_1\ast G'_2$ and pick a finite generating set $A$ (resp. $B$) of $G'_1$ (resp. $G'_2$). The quotient $G$ is constructed by adding particular relations (which will be done by Dehn filling) to $\widetilde{G}$ which identify elements of $A$ (resp. $B$) with certain elements of $G'_2$ (resp. $G'_1$). 

There exists a finite generating set $A=\{a_1,...,a_k\}$ (resp. $B=\{b_1,...,b_k\}$) of $G'_1$ (resp. $G'_2$) for some $k\in\mathbb{N}^+$. To simplify the argument, let 

\begin{equation}\label{eq. a,b}
a_{k+1}=a_{k+2}=b_{k+1}=b_{k+2}=1.
\end{equation}

To perform Dehn filling on $\widetilde{G}$, the first step is to find hyperbolically embedded subgroups. By \cite[Lemma 5.10]{hull2013small}, $G'_1$ and $G'_2$ are acylindrically hyperbolic with $K(G'_1)=K(G'_2)=\{1\}$. Thus, Theorem \ref{thm. hyperbolically embedded virtually free subgroup} implies that there exist free groups 
$$F_1\hookrightarrow_h G'_1,~~~~F_2\hookrightarrow_h G'_2,$$
each of which has rank $2k+4$. By Example \ref{eg. product hyperbolically embedded}, we have $\{G'_1,G'_2\}\hookrightarrow_h \widetilde{G}$. Thus, Proposition \ref{prop. iterative hyperbolically embedded} implies
$$\{F_1,F_2\}\hookrightarrow_h \widetilde{G}.$$

In fact, $F_1,F_2$ are not quite the subgroups that we want, and we will apply Proposition \ref{prop. showing h.e. general} to construct other hyperbolically embedded subgroups from $A,B,F_1,F_2$. Note that for $i=1,2,\cdots,k+2$,
$$a_i\not\in F_2\smallsetminus\{1\},~~ b_i\not\in F_1\smallsetminus\{1\}.$$
Let $\{f_{1,i},g_{1,i}\}^{k+2}_{i=1}$ (resp. $\{f_{2,i},g_{2,i}\}^{k+2}_{i=1}$) be a basis of the free group $F_1$ (resp. $F_2$). Then Proposition \ref{prop. showing h.e. general} implies the following.

\begin{lemma}\label{lem. definition and properties of H1 H2}
For sufficiently large $\ell\in\mathbb{N}^+$, $\{f^{\ell}_{1,i}b_ig^{\ell}_{1,i}\}^{k+2}_{i=1}$ (resp.$\{f^{\ell}_{2,i}a_ig^{\ell}_{2,i}\}^{k+2}_{i=1}$) freely generates a subgroup $H_1\leqslant \widetilde{G}$ (resp. $H_2\leqslant \widetilde{G}$) and
$$\{H_1,H_2\}\hookrightarrow_h\widetilde{G}.$$
\end{lemma}

\begin{proof}[Proof of Theorem \ref{thm. common quotient of acylindrically hyperbolic groups}]
As $G'_1$ and $G'_2$ are acylindrically hyperbolic, we have $|G'_1|=|G'_2|=\infty$ and thus $\cd(G'_1),\cd(G'_2)\geqslant 1$. Suppose $\cd(G'_1)=\cd(G'_2)=1$. Then $G'_1$ and $G'_2$ are free by the Stallings-Swan theorem \cite[corollary to Theorem 1]{swan1969groups}. Without loss of generality, we may assume that the rank of $G'_1$ is greater than or equal to the rank of $G'_2$. It follows that $G'_2$ is a quotient of $G'_1$. Let $G=G'_2$. Statements (i), (ii), and (iii) follow trivially. Statement (iv) also holds because if $G_2$ is of type $FP_n$ for some $n\in\{2,3,...,\infty\}$, then $G'_2$ is a finite rank free group and thus of type $FP_{\infty}$.

Thus, let us assume $\max\{\cd(G'_1),\cd(G'_2)\}\geqslant 2$. Fix a sufficiently large $\ell\in\mathbb{N}^+$ and let $H_1$ and $H_2$ be the subgroups given by Lemma \ref{lem. definition and properties of H1 H2}. By Remark \ref{rem. sufficiently deep for he}, Lemma \ref{lem. definition and properties of H1 H2}, and Theorems \ref{thm. group theoretic Dehn filling}, \ref{thm. CL-property}, there exist finite sets $\mathcal{F}_1\subset H_1\smallsetminus\{1\}$ and $\mathcal{F}_2\subset H_2\smallsetminus\{1\}$ such that if $N_1\lhd H_1,N_2\lhd H_2$ and $N_1\cap\mathcal{F}_1=N_2\cap\mathcal{F}_2=\emptyset$, then the following hold.
\begin{enumerate}
    \item[(a)] $\{H_1/N_1,H_2/N_2\}\hookrightarrow_h \widetilde G/\ll N_1\cup N_2 \rr$.
    \item[(b)] $(\widetilde G,\{H_1,H_2\},\{N_1,N_2\})$ is a Cohen-Lyndon triple and thus Theorems \ref{direct sum of cohomology} and \ref{property FP} and Corollary \ref{cd of Dehn fillings} can be applied to it.
\end{enumerate}

Let $\{u_i\}^k_{i=1}$ (resp. $\{v_i\}^k_{i=1}$) be freely reduced words over the alphabet $\{f^{\ell}_{1,k+1}b_{k+1}g^{\ell}_{1,k+1},f^{\ell}_{1,k+2}b_{k+2}g^{\ell}_{1,k+2}\}$ (resp. $\{f^{\ell}_{2,k+1}a_{k+1}g^{\ell}_{2,k+1},f^{\ell}_{2,k+2}a_{k+2}g^{\ell}_{2,k+2}\}$) satisfying the $C'(1/6)$ small cancellation condition. Note that $u_i\in G'_1$ and $v_i\in G'_2$ for $i=1,...,k$ by \eqref{eq. a,b}. Let $N_1$ (resp. $N_2$) be the normal subgroup of $H_1$ (resp. $H_2$) generated by $\{f^{\ell}_{1,i}b_ig^{\ell}_{1,i}u_i\}^k_{i=1}$ (resp. $\{f^{\ell}_{2,i}a_ig^{\ell}_{2,i}v_i\}^k_{i=1}$).

By Lemma \ref{lem. definition and properties of H1 H2}, $H_1$ and $H_2$ are freely generated by $\{f^{\ell}_{1,i}b_ig^{\ell}_{1,i}\}^{k+2}_{i=1}$ and $\{f^{\ell}_{2,i}a_ig^{\ell}_{2,i}\}^{k+2}_{i=1}$, respectively. Thus, $H_1/N_1$ and $H_2/N_2$ can be presented as
\begin{align}
    H_1/N_1&=\langle f^{\ell}_{1,i}b_ig^{\ell}_{1,i}~~ (i=1,...,k+2)\mid f^{\ell}_{1,j}b_jg^{\ell}_{1,j}u_j ~~(j=1,...,k) \rangle\nonumber\\
    &=\langle f^{\ell}_{1,k+1}b_{k+1}g^{\ell}_{1,k+1},f^{\ell}_{1,k+2}b_{k+2}g^{\ell}_{1,k+2} \rangle,\label{eq. presentation of H1/N1}\\
    H_2/N_2&=\langle f^{\ell}_{2,i}a_ig^{\ell}_{2,i}~~ (i=1,...,k+2)\mid f^{\ell}_{2,j}a_jg^{\ell}_{2,j}v_j ~~(j=1,...,k) \rangle\nonumber\\
    &=\langle f^{\ell}_{2,k+1}a_{k+1}g^{\ell}_{2,k+1},f^{\ell}_{2,k+2}a_{k+2}g^{\ell}_{2,k+2} \rangle,\label{eq. presentation of H2/N2}
\end{align}
where the last equality of \eqref{eq. presentation of H1/N1} (resp. \eqref{eq. presentation of H2/N2}) follows from eliminating $f^{\ell}_{1,i}b_ig^{\ell}_{1,i}$ (resp. $f^{\ell}_{2,i}a_ig^{\ell}_{2,i}$) for $i=1,...,k$ using Tietze transformations (see \cite[Chapter II]{lyndon2015combinatorial}).

Thus, $H_1/N_1$ and $H_2/N_2$ are free groups of rank $2$. In particular,
\begin{equation}\label{eq. N1 N2 infinite}
    \card(H_1/N_1)=\infty.
\end{equation}

By the Greendlinger's lemma for free groups \cite[Chapter V Theorem 4.5]{lyndon2015combinatorial}, if $\|u_i\|,\|v_i\|,1\leqslant i\leqslant k,$ are sufficiently large, then 
$$N_1\cap \mathcal{F}_1=N_2\cap \mathcal{F}_2=\emptyset.$$

Let 
$$G=\widetilde{G}/\ll N_1\cup N_2\rr.$$
As $a_{k+1}=a_{k+2}=b_{k+1}=b_{k+2}=1$, $G$ is a common quotient of $G'_1$ and $G'_2$. In particular, $G$ is a common quotient of $G_1$ and $G_2$.

If $H_1/N_1=G$, then $G$ is a non-cyclic free group and thus is acylindrically hyperbolic. If $H_1/N_1$ is a proper subgroup of $G$, then equation \eqref{eq. N1 N2 infinite}, item (a), and Theorem \ref{thm. group theoretic Dehn filling} imply that $G$ is acylindrically hyperbolic. Statement (i) is proved.

Consider statement (ii). For every $n\geqslant 3$ and every $G$-module $A$, we have
\begin{align*}
    &  H^n(G;A)\\
    \cong &  H^n(\widetilde{G};A)\oplus  H^n(H_1/N_1;A)\oplus  H^n(H_2/N_2;A)&&\text{by Corollary \ref{direct sum of cohomology} and that }H_1,H_2 \text{ are free}&\\
    \cong &  H^n(\widetilde{G};A)&&\text{as }H_1/N_1\text{ and }H_2/N_2\text{ are free groups}&\\
    \cong &  H^n(G'_1;A)\oplus  H^n(G'_2;A)&&\text{as }\widetilde{G}=G'_1\ast G'_2,&
\end{align*}
which is (ii).

Consider statement (iii). If either $K(G_1)$ or $K(G_2)$ is not the trivial group $\{1\}$, then (iii) is trivial. So let us assume that $K(G_1)=K(G_2)=\{1\}$. Then Corollary \ref{cd of Dehn fillings} implies
\begin{align*}
  \cd(G)&\leqslant \max\{\cd(\widetilde{G}),\cd(H_1/N_1),\cd(H_2/N_2),\cd(H_1)+1,\cd(H_2)+1\}\\
  &=\max\{\cd(\widetilde{G}),1,2\}&&\hspace{-4.2cm}\text{as }H_1,H_2,H_1/N_1,\text{ and }H_2/N_2\text{ are free groups}\\
  &=\max\{\cd(G'_1),\cd(G'_2)\}&&\hspace{-4.2cm}\text{as }\cd(G'_1),\cd(G'_2)\geqslant 2\\
  &= \max\{\cd(G_1),\cd(G_2)\}.
\end{align*}

Finally, suppose that $G_1$ and $G_2$ are of type $FP_n$ for some $n\in\{2,3,...,\infty\}$. Then so are $G'_1$ and $G'_2$ \cite[Corollary 9]{alonso1994finiteness}. As $H_1/N_1$ and $H_2/N_2$ are free groups of finite rank, they are of type $FP_{\infty}$. Therefore, Theorem \ref{property FP} implies that $G$ is of type $FP_n$ and thus statement (iv) holds.
\end{proof}

\begin{remark}\label{rm. hd of common quotient}
Analogous to the above proof, in the setting of Theorem \ref{thm. common quotient of acylindrically hyperbolic groups}, we have $\hd(G)\leqslant\max\{\hd(G_1),\hd(G_2),2\}$.
\end{remark}

\section{Some applications} \label{sec. applications}

In this section, we give some applications of Theorems \ref{cd SQ-universality} and \ref{thm. common quotient of acylindrically hyperbolic groups}.

\subsection{Infinite dimension torsion-free $FP_{\infty}$ groups} 

If a group $G$ has $\cd(G)<\infty$, then $G$ is necessarily torsion-free. Of course, the converse is not true, e.g., $\cd(\mathbb{Z}^{\infty})=\infty$. Observe however that $\mathbb{Z}^{\infty}$ is not of type $FP_{\infty}$. In fact, Bieri asked if every torsion-free group
of type $FP_{\infty}$ is also of type $FP$ or if there is an $FP_{\infty}$
group with a non-finitely generated free abelian subgroup \cite[Problem F11]{wall1979homological}. In \cite{brown1984infinite}, Brown and Geoghegan settled both questions, by showing that Thompson's group $F$ is of type $FP_{\infty}$. 

For every torsion-free acylindrically hyperbolic group $G$, Theorem \ref{cd SQ-universality} embeds $F$ into an acylindrically hyperbolic quotient $\bg$ of $G$, which is a torsion-free $FP_{\infty}$ group of infinite cohomological dimension. We thus have the following.

\begin{corollary}\label{cor. true F}
Every torsion-free acylindrically hyperbolic group $G$ of type $FP_{\infty}$ has a torsion-free acylindrically hyperbolic quotient $\bg$ of type $FP_{\infty}$ which contains the Thompson group $F$. In particular, $\cd(\bg)=\infty$. 
\end{corollary}

\subsection{Quotients of hyperbolic groups} 

SQ-universality of non-elementary  hyperbolic groups was proved by Olshanskii \cite{Olsh} and independently by Delzant \cite{Del}. In Theorem \ref{cd SQ-universality}, if $G$ and $C$ are word-hyperbolic, a hyperbolic quotient $\bg$ of $G$ can be constructed so that the conclusions hold. The next corollary is thus a strengthening of the aforementioned result.

\begin{corollary}\label{cor. hyperbolic}
Let $G$ be a non-elementary hyperbolic group and $C$ any hyperbolic group. Then there is a hyperbolic quotient $\bg$ of $G/K(G)$ (in particular, $\bg$ is a quotient of $G$), where $K(G)$ is the maximal finite normal subgroup of $G$, such that $C$ embeds into $\bg$ and the following hold.
\begin{enumerate}
\item[(i)] For all $n\geqslant 3$ and every $\bg$-module $A$, we have
\[ H^n(\bg;A)\cong  H^n(G/K(G);A)\oplus  H^n(C;A),\]
where the action of $G/K(G)$ (resp. $C$) on $A$ is induced by the quotient map $G/K(G)\rightarrow\bg$ (resp. the embedding $C\hookrightarrow \bg$).\\
\item[(ii)] $\cd(\bg)\leqslant \max\{\cd(G),\cd(C)\}$.\\
\end{enumerate}
\end{corollary}

\begin{proof}
By passing to $G/K(G)$, we may assume that $K(G)=\{1\}$. By \cite[Corollary 4.21]{osin2020topological}, there is a free subgroup $F_2<G$ of rank $2$ such that $G$ is hyperbolic relative to $F_2$. By \cite{jitsukawa2002malnormal}, four random elements of $F_2$ freely generate a free subgroup $F_4\leqslant F_2$ that is malnormal in $F_2$. By Marshall Hall's theorem, there exists a finite-index subgroup $H$ of $F_2$ such that $F_4$ is a free factor of $H$. In particular, $F_4$ is quasi-convex in $F_2$, and thus \cite[Theorem 7.11]{bowditch2012relatively} implies that $F_2$ is hyperbolic relative to $F_4$. Thus, $G$ is hyperbolic relative to $F_4$ \cite{osin2006relatively}. By Lemma \ref{free group situation}, $C$ embeds into a quotient $R$ of $F_4$ such that $R$ is hyperbolic relative to $C$ and the quotient map $F_4\twoheadrightarrow R$ is injective on any given finite set $\mathcal{F}\subset F_4\smallsetminus\{1\}$. Since $C$ is hyperbolic, so is $R$ \cite[Corollary 2.41]{osin2006relatively}. Let $\bg$ be the quotient of $G$ by the Dehn filling $F_4\twoheadrightarrow R$. By avoiding a suitable finite set $\mathcal{F}$, we may assume that $\bg$ is hyperbolic relative to $R$ \cite[Theorem 1.1]{osin2007peripheral}. Therefore, $\bg$ is hyperbolic \cite[Corollary 2.41]{osin2006relatively}. Items (i) and (ii) are immediate consequences of Theorem \ref{cd SQ-universality}.
\end{proof}

Recently in \cite[Corollary 2]{italiano2021hyperbolic}, Italiano-Martelli-Migliorini settled a long-standing open problem by constructing the first example of hyperbolic group which contains a subgroup of type $F$ that is not hyperbolic. 
Their group is the fundamental group of a $5$-dimensional hyperbolic pseudo-manifold that fibers over the circle. The fundamental group of the fiber is of type $F$ but not hyperbolic.  

\begin{corollary}\label{cor. non-hyperbolic} 
Let $n\geq 5$ be an integer.
Every non-elementary hyperbolic group $G$ with $\cd(G)\leq n$ has a hyperbolic  quotient $\bg$ with $\cd(\bg)=n$ such that $\bg$ contains the Italiano-Martelli-Migliorini group. In particular, there is a type $F$ non-hyperbolic subgroup $H<\bg$.
\end{corollary}

\begin{proof}
Let $C_1$ be the group constructed by Italiano-Martelli-Migliorini in \cite[Corollary 2]{italiano2021hyperbolic}. Then $\cd(C_1)=5$.  Let $C_2$ be a hyperbolic group with $\cd(C_2)=n$, and let $C=C_1\ast C_2$. Since $\cd(G)<\infty$, we have $K(G)=\{1\}$. Corollary \ref{cor. hyperbolic} thus yields a hyperbolic quotient $\bg$ such that $C$ embeds into $\bg$ and $\cd(\bg)\leqslant \max\{\cd(G),\cd(C)\}=\max\{\cd(G),\cd(C_1),\cd(C_2)\} = n$. Since $C$ embeds into $\bg$, we also have $\cd(\bg)\geqslant \cd(C)\geqslant \cd(C_2) = n$, and thus $\cd(\bg)=n$. 
\end{proof}

\subsection{Property (T) quotients}
A group $G$ has \textit{Kazhdan's property (T)} if every affine isometric action of $G$ on a Hilbert space has a global fixed point. The next result strengthens \cite[Corllary 1.7]{hull2013small}.
\begin{corollary}\label{cor:T}
Every acylindrically hyperbolic group $G$ of type $FP_n$ for some $n\in\{2,3,...,\infty\}$ (resp.~$FP$) has an acylindrically hyperbolic quotient $\bg$  of type $FP_n$ (resp.~$FP$) with Kazhdan's Property (T) such that $\cd(\bg)\leq\max\{\cd(G), 2\}$.
\end{corollary}

\begin{proof} 
Let $\Gamma$ be an type $FP$ acylindrically hyperbolic property (T) group with $\cd(\Gamma)=2$. For example, one can let $\Gamma$ be a random group in the Gromov density model with density $1/3<d<1/2$. By \cite[Section 9.B]{gromov1993asymptotic} (see also \cite[Theorem 1]{ollivier2004sharp}), with overwhelming probability $\Gamma$ is hyperbolic and has an aspherical presentation, and thus is of type $FP$ and satisfies $\cd(\Gamma)\leqslant 2$. By \cite[Theorem 4]{zuk2003property} (see also \cite[Section I.3.g]{ollivier2005invitation}), with overwhelming probability $\Gamma$ has property (T). The result now follows from Theorem \ref{thm. common quotient of acylindrically hyperbolic groups} applied to $G$ and $\Gamma$. 
\end{proof}

The above result can for example be applied to mapping class groups of surfaces of finite type, outer automorphism groups of free groups of finite rank and (non-virtually polycyclic) fundamental groups of compact orientable $3$-manifolds which all exhibit nice cohomological finiteness conditions. 

\subsection{Acylindrically hyperbolic quotients distinguishable by cohomology}

Let $G$ be any acylindrically hyperbolic group. For each $k\geqslant 3$, Theorem \ref{cd SQ-universality} implies that one can embed $\mathbb{Z}^k$ into a quotient $G_k$ of $G$ such that $H^n(G_k;\mathbb{Z})\cong H^n(G;\mathbb{Z})\oplus  H^n(\mathbb{Z}^k;\mathbb{Z})$ for all $n\geqslant 3$. Suppose $\cd(G)<\infty$. Then $H^{\ast}(G_k;\mathbb{Z})\neq H^{\ast}(G_{\ell};\mathbb{Z})$ for $k,\ell>\cd(G)$. Suppose that $G$ is of type $FP_{\infty}$ instead. Then $H^k({G;\mathbb{Z}})$ is finitely generated, and thus $H^k(G_k;\mathbb{Z})\not\cong H^k(G_{\ell};\mathbb{Z})$ for $k\neq \ell$. This establishes the following.

\begin{corollary}
Let $G$ be any acylindrically hyperbolic group. Then there exists an infinite family $\{G_k\}^{\infty}_{k=3}$ of acylindrically hyperbolic quotients of $G$ such that $H^n(G_k;\mathbb{Z})\cong H^n(G;\mathbb{Z})\oplus H^n(\mathbb{Z}^k;\mathbb{Z})$ for all $n,k\geqslant 3$. In particular, if $\cd(G)<\infty$ or $G$ is of type $FP_{\infty}$, then each pair of elements of $\{G_k\}_{k=l}^{\infty}$ have non-isomorphic integral cohomology for ${l=\cd(G)+1}$ and $l=3$, respectively.
\end{corollary}

If instead of embedding $\mathbb{Z}^k$, we embed the family of Houghton's groups or Abel's groups \cite{brow1987finiteness}, then we will obtain quotients that are separated by homological finiteness properties.

\begin{corollary}\label{cor. fp_n but not fp_n+1}
Let $G$ be any acylindrically hyperbolic group of type $FP_{\infty}$. Then $G$ has a family of acylindrically hyperbolic quotients $\{G_k\}^{\infty}_{k=2}$ such that for each $k$, $G_k$ has Kazhdan's Property (T), is of type $FP_{k-1}$ but not of type $FP_{k}$.
\end{corollary}

We remark that being of type $FP_n$ is a quasi-isometric invariant \cite[Corollary 9]{alonso1994finiteness}. Therefore, the quotients $\{G_k\}_{k\geqslant 1}$ in Corollary \ref{cor. fp_n but not fp_n+1} are pairwise non-quasi-isometric.

\begin{proof} By Corollary \ref{cor:T}, we can pass to a quotient $\bg$ of $G$ that is acylindrically hyperbolic and of type $FP_{\infty}$ and has property (T).
There exists a family of groups $\{H_k\}^{\infty}_{k=2}$ such that for all $k$, $H_k$ is of type $FP_{k-1}$ but not of type $FP_{k}$. For example, one can let $\{H_k\}^{\infty}_{k=2}$ be the family of Houghton's groups or Abel's groups \cite[Theorems 5.1 and 6.1]{brow1987finiteness}. Theorem \ref{cd SQ-universality} then embeds each $H_k$ to a quotient $G_k$ of $\bg$ with the desired properties.
\end{proof}

\subsection{First Betti number and a question of Osin}

Motivated by the Virtual Positive First Betti Number Conjecture which has now been settled by Agol \cite{agol2013virtual}, Osin posed the following question.

\begin{question}[{Osin, \cite[Problem 2.3]{osin2008questions}}]\label{Q. Osin}
Let $G$ be a Poincar\'{e} duality group of dimension $3$ such that $G$ is hyperbolic relative to its subgroup $H$. If $G$ has positive virtual first Betti number, is it true that for most $N\lhd H$, $G/\ll N \rr$ also has positive virtual first Betti number?
\end{question}

Question \ref{Q. Osin} is only about virtual Betti numbers. For actual Betti numbers there are counterexamples. For example, let $S$ be an orientable surface of genus two. Take a suitable pseudo-Anosov homeomorphism $\phi$ on $S$ such that the mapping torus $T_{\phi}$ has $b_1(T_{\phi})=1$. Then $\pi_1(T_{\phi})$ is hyperbolic and is a semi-direct product $\pi_1(T_{\phi})=\pi_1(S)\rtimes\langle t \rangle$. $\pi_1(T_{\phi})$ is hyperbolic relative to $\langle t \rangle$ \cite{bowditch2012relatively}, and for every $n\geqslant 1$, $b_1(\pi_1(T_{\phi})/\ll t^n \rr)=0$ since the image of $t$ generates $H_1(\pi_1(T_{\phi});\mathbb{Q})$. Using this example, we prove the following.

\begin{corollary}
Let $G$ be any acylindrically hyperbolic group. Then $G$ has an acylindrically hyperbolic quotient $\bg$ such that $b_1(\bg)=0,\cd(\bg)\leqslant\max\{\cd(G),3\}$, and $H^n(\bg;A)\cong H^n(G;A)$ for all $n\geqslant 4$ and any $\bg$-module $A$.
\end{corollary}

\begin{proof}
The idea is to construct an acylindrically hyperbolic group $G'$ such that $b_1(G')=0$ and $\cd(G')\leqslant 3$. Once such a $G'$ has been constructed, Theorem \ref{thm. common quotient of acylindrically hyperbolic groups} will produce a common quotient $\bg$ of $G$ and $G'$ with the desired properties.

So it remains to construct $G'$. Let $T_{\phi},t$ be as above. Take eight elements $a_1,a_2,\cdots,a_8$ that generate $\pi_1(S)$ as a group. Then $a_1,a_2,\cdots,a_8,t$ generate $\pi_1(T_{\phi})$ as a group. Let $X$ be the set consisting of $a_1,a_2,\cdots,a_8,t$ and their inverses. Randomly generate two words $w_1,w_2$ over $X$ of length $n$ and identify $w_1,w_2$ with the elements of $\pi_1(T_{\phi})$ they represent.

By \cite[Theorem 1]{maher2019random}, as $n\rightarrow \infty$, with probability $1$ we have that $\langle w_1,w_2 \rangle\leqslant \pi_1(T_{\phi})$ is free of rank $2$ and $\langle w_1,w_2 \rangle\hookrightarrow_h \pi_1(T_{\phi})$. Moreover, as $n\rightarrow \infty$, with positive probability we have that both $w_1$ and $w_2$ have non-zero $t$-exponents. Thus, there exist $w_1,w_2\in \pi_1(T_{\phi})$ such that 
\begin{enumerate}
    \item[(a)] $\langle w_1,w_2 \rangle\leqslant \pi_1(T_{\phi})$ is free of rank $2$,
    \item[(b)] $\langle w_1,w_2 \rangle\hookrightarrow_h \pi_1(T_{\phi})$, and
    \item[(c)] $w_1$ and $w_2$ have non-zero $t$-exponents.
\end{enumerate}

Let $w$ be a word over $w_1,w_2$ such that $w$ is not a proper power, satisfies the $C'(1/6)$ small cancellation condition, and $w$ has positive $t$-exponent. Let $N$ be the normal subgroup of $\langle w_1,w_2 \rangle$ generated by $w$. By \cite[Theorem 11.1]{lyndon1950cohomology}, $\langle w_1,w_2 \rangle/N$ has $\cd(\langle w_1,w_2 \rangle/N)\leqslant 2$. Let $\ll N \rr$ be the normal closure of $N$ in $\pi_1(T_{\phi})$. By making $w$ sufficiently long, we can guarantee that $N$ avoids any given finite subset of $\langle w_1,w_2 \rangle\smallsetminus\{1\}$, and therefore guarantee that $\pi_1(T_{\phi})/\ll N \rr$ is acylindrically hyperbolic and $\cd(\pi_1(T_{\phi})/\ll N \rr)\leqslant 3$ by Theorems \ref{thm. group theoretic Dehn filling} and \ref{thm. main}. Since $w$ has positive $t$-exponent, we have $b_1(\pi_1(T_{\phi})/\ll N \rr)=0$. So it suffices to let $G'=\pi_1(T_{\phi})/\ll N \rr$.
\end{proof}

\bibliographystyle{alpha}
\bibliography{bin_refs}

\begin{thebibliography}{AGM16}

\bibitem[ACG18]{antolin2017farrell}
{Y}. Antol{\'\i}n, {R}. {C}oulon, and {G}. {G}andini.
\newblock Farrell-{J}ones via {D}ehn fillings.
\newblock {\em J. Topol. Anal.}, 10(4):873--895, 2018.

\bibitem[AGM16]{agol2016alternate}
I.~Agol, D.~Groves, and J.~Manning.
\newblock An alternate proof of {W}ise's malnormal special quotient theorem.
\newblock {\em Forum Math. Pi}, 4:e1, 54, 2016.

\bibitem[Ago13]{agol2013virtual}
I.~Agol.
\newblock The virtual {H}aken conjecture.
\newblock {\em Doc. Math.}, 18:1045--1087, 2013.
\newblock With an appendix by I. Agol, D. Groves, and J. Manning.

\bibitem[Alo94]{alonso1994finiteness}
J.~Alonso.
\newblock Finiteness conditions on groups and quasi-isometries.
\newblock {\em J. Pure Appl. Algebra}, 95(2):121--129, 1994.

\bibitem[Are22]{Arenas}
M.~Arenas.
\newblock A cubical rips construction.
\newblock {\em arXiv:2202.01048}, 2022.

\bibitem[BE78]{bieri1978relative}
R.~Bieri and B.~Eckmann.
\newblock Relative homology and {P}oincar\'{e} duality for group pairs.
\newblock {\em J. Pure Appl. Algebra}, 13(3):277--319, 1978.

\bibitem[Besa]{bridson?problems}
M.~Bestvina.
\newblock Problems concerning hyperbolic and {CAT}(0) groups.

\bibitem[Besb]{bestvina?questions}
M.~Bestvina.
\newblock Questions in geometric group theory.

\bibitem[BF10]{bestvina2009hyperbolic}
M.~Bestvina and M.~Feighn.
\newblock A hyperbolic {$Out(F_n)$}-complex.
\newblock {\em Groups Geom. Dyn.}, 4(1):31--58, 2010.

\bibitem[BG84]{brown1984infinite}
K.~Brown and R.~Geoghegan.
\newblock An infinite-dimensional torsion-free {${FP}_{\infty }$} group.
\newblock {\em Invent. Math.}, 77(2):367--381, 1984.

\bibitem[BH96]{Bleiler_Hodgson}
Steven~A. Bleiler and Craig~D. Hodgson.
\newblock Spherical space forms and {D}ehn filling.
\newblock {\em Topology}, 35(3):809--833, 1996.

\bibitem[Bie75]{bieri1975mayer}
{R}. Bieri.
\newblock Mayer-{V}ietoris sequences for {HNN}-groups and homological duality.
\newblock {\em Math. Z.}, 143(2):123--130, 1975.

\bibitem[Bie81]{bieri1981homological}
R.~Bieri.
\newblock {\em Homological dimension of discrete groups}.
\newblock Queen Mary College Mathematical Notes. Queen Mary College, Department
  of Pure Mathematics, London, second edition, 1981.

\bibitem[Bow08]{bowditch2008tight}
B.~Bowditch.
\newblock Tight geodesics in the curve complex.
\newblock {\em Invent. Math.}, 171(2):281--300, 2008.

\bibitem[Bow12]{bowditch2012relatively}
B.~Bowditch.
\newblock Relatively hyperbolic groups.
\newblock {\em Internat. J. Algebra Comput.}, 22(3):1250016, 66, 2012.

\bibitem[Bra99]{brady1999branched}
N.~Brady.
\newblock Branched coverings of cubical complexes and subgroups of hyperbolic
  groups.
\newblock {\em J. London Math. Soc. (2)}, 60(2):461--480, 1999.

\bibitem[Bre93]{Bredon}
Glen~E. Bredon.
\newblock {\em Topology and geometry}, volume 139 of {\em Graduate Texts in
  Mathematics}.
\newblock Springer-Verlag, New York, 1993.

\bibitem[Bro75]{brown1975homological}
{K}. Brown.
\newblock Homological criteria for finiteness.
\newblock {\em Comment. Math. Helv.}, 50:129--135, 1975.

\bibitem[Bro87]{brow1987finiteness}
K.~Brown.
\newblock Finiteness properties of groups.
\newblock In {\em Proceedings of the {N}orthwestern conference on cohomology of
  groups ({E}vanston, {I}ll., 1985)}, volume~44, pages 45--75, 1987.

\bibitem[Bro94]{brown1982cohomology}
{K}. Brown.
\newblock {\em Cohomology of groups}, volume~87 of {\em Graduate Texts in
  Mathematics}.
\newblock Springer-Verlag, New York, 1994.
\newblock Corrected reprint of the 1982 original.

\bibitem[CE56]{Cartan_Eilenberg}
Henri Cartan and Samuel Eilenberg.
\newblock {\em Homological algebra}.
\newblock Princeton University Press, Princeton, NJ, 1956.

\bibitem[CL63]{cohen1963free}
{D}. Cohen and {R}. {L}yndon.
\newblock Free bases for normal subgroups of free groups.
\newblock {\em Trans. Amer. Math. Soc.}, 108:526--537, 1963.

\bibitem[CL13]{Cantat2013normal}
S.~Cantat and S.~Lamy.
\newblock Normal subgroups in the {C}remona group.
\newblock {\em Acta Math.}, 210(1):31--94, 2013.
\newblock With an appendix by Yves de Cornulier.

\bibitem[Deg23]{deg2016}
D.~Degrijse.
\newblock Amenable groups of finite cohomological dimension and the zero
  divisor conjecture.
\newblock {\em Math. Ann.}, 386(1-2):1151--1162, 2023.

\bibitem[Dek96]{Dekimpe}
Karel Dekimpe.
\newblock {\em Almost-{B}ieberbach groups: affine and polynomial structures},
  volume 1639 of {\em Lecture Notes in Mathematics}.
\newblock Springer-Verlag, Berlin, 1996.

\bibitem[Del96]{Del}
Thomas Delzant.
\newblock Sous-groupes distingu\'{e}s et quotients des groupes hyperboliques.
\newblock {\em Duke Math. J.}, 83(3):661--682, 1996.

\bibitem[DG18]{dahmani2018recognizing}
F.~Dahmani and V.~Guirardel.
\newblock Recognizing a relatively hyperbolic group by its {D}ehn fillings.
\newblock {\em Duke Math. J.}, 167(12):2189--2241, 2018.

\bibitem[DGO17]{dahmani2017hyperbolically}
{F}. Dahmani, {V}. {G}uirardel, and {D}. {O}sin.
\newblock Hyperbolically embedded subgroups and rotating families in groups
  acting on hyperbolic spaces.
\newblock {\em Mem. Amer. Math. Soc.}, 245(1156):v+152, 2017.

\bibitem[Far98]{farb1998relatively}
B.~Farb.
\newblock Relatively hyperbolic groups.
\newblock {\em Geom. Funct. Anal.}, 8(5):810--840, 1998.

\bibitem[FM10]{fujmann2010}
K.~Fujiwara and J.~F. Manning.
\newblock {${\rm CAT(0)}$} and {${\rm CAT}(-1)$} fillings of hyperbolic
  manifolds.
\newblock {\em J. Differential Geom.}, 85(2):229--269, 2010.

\bibitem[FM11]{fujmann2011}
K.~Fujiwara and J.~F. Manning.
\newblock Simplicial volume and fillings of hyperbolic manifolds.
\newblock {\em Algebr. Geom. Topol.}, 11(4):2237--2264, 2011.

\bibitem[Fra18]{franc2018}
F.~Franceschini.
\newblock A characterization of relatively hyperbolic groups via bounded
  cohomology.
\newblock {\em Groups Geom. Dyn.}, 12(3):919--960, 2018.

\bibitem[Fri17]{frigerio2017}
R.~Frigerio.
\newblock {\em Bounded cohomology of discrete groups}, volume 227 of {\em
  Mathematical Surveys and Monographs}.
\newblock American Mathematical Society, Providence, RI, 2017.

\bibitem[GM08]{groves2008dehn}
{D}. Groves and {J}. {M}anning.
\newblock Dehn filling in relatively hyperbolic groups.
\newblock {\em Israel J. Math.}, 168:317--429, 2008.

\bibitem[GM18]{groves2018dehn}
D.~Groves and J.~Manning.
\newblock Dehn fillings and elementary splittings.
\newblock {\em Trans. Amer. Math. Soc.}, 370(5):3017--3051, 2018.

\bibitem[GMS19]{groves2016boundaries}
D.~Groves, J.~F. Manning, and A.~Sisto.
\newblock Boundaries of {D}ehn fillings.
\newblock {\em Geom. Topol.}, 23(6):2929--3002, 2019.

\bibitem[Gro57]{grothendieck1957sur}
A.~Grothendieck.
\newblock Sur quelques points d'alg{\`e}bre homologique.
\newblock {\em Tohoku Math. J. (2)}, 9(2):119--221, 1957.

\bibitem[Gro78]{gromov1978}
M.~Gromov.
\newblock Almost flat manifolds.
\newblock {\em J. Differential Geometry}, 13(2):231--241, 1978.

\bibitem[Gro82]{gromov1982}
M.~Gromov.
\newblock Volume and bounded cohomology.
\newblock {\em Inst. Hautes \'{E}tudes Sci. Publ. Math.}, (56):5--99 (1983),
  1982.

\bibitem[Gro93]{gromov1993asymptotic}
M.~Gromov.
\newblock Asymptotic invariants of infinite groups.
\newblock In {\em Geometric group theory, {V}ol. 2 ({S}ussex, 1991)}, volume
  182 of {\em London Math. Soc. Lecture Note Ser.}, pages 1--295. Cambridge
  Univ. Press, Cambridge, 1993.

\bibitem[GS18]{Gruber2018infinitely}
D.~Gruber and A.~Sisto.
\newblock Infinitely presented graphical small cancellation groups are
  acylindrically hyperbolic.
\newblock {\em Ann. Inst. Fourier (Grenoble)}, 68(6):2501--2552, 2018.

\bibitem[HS53]{hochschild1953cohomology}
G.~Hochschild and J.~Serre.
\newblock Cohomology of group extensions.
\newblock {\em Trans. Amer. Math. Soc.}, 74:110--134, 1953.

\bibitem[Hul16]{hull2013small}
{M}. Hull.
\newblock Small cancellation in acylindrically hyperbolic groups.
\newblock {\em Groups Geom. Dyn.}, 10(4):1077--1119, 2016.

\bibitem[IMM23]{italiano2021hyperbolic}
G.~Italiano, B.~Martelli, and M.~Migliorini.
\newblock Hyperbolic 5-manifolds that fiber over {$S^1$}.
\newblock {\em Invent. Math.}, 231(1):1--38, 2023.

\bibitem[Iva85]{ivanov85}
N.~V. Ivanov.
\newblock Foundations of the theory of bounded cohomology.
\newblock volume 143, pages 69--109, 177--178. 1985.
\newblock Studies in topology, V.

\bibitem[Jit02]{jitsukawa2002malnormal}
T.~Jitsukawa.
\newblock Malnormal subgroups of free groups.
\newblock In {\em Computational and statistical group theory ({L}as {V}egas,
  {NV}/{H}oboken, {NJ}, 2001)}, volume 298 of {\em Contemp. Math.}, pages
  83--95. Amer. Math. Soc., Providence, RI, 2002.

\bibitem[JNW21]{jankiewicz2021virtually}
K.~Jankiewicz, S.~Norin, and D.~Wise.
\newblock Virtually fibering right-angled {C}oxeter groups.
\newblock {\em J. Inst. Math. Jussieu}, 20(3):957--987, 2021.

\bibitem[JW72]{johnwall72}
F.~E.~A. Johnson and C.~T.~C. Wall.
\newblock On groups satisfying {P}oincar\'{e} duality.
\newblock {\em Ann. of Math. (2)}, 96:592--598, 1972.

\bibitem[L{\"{o}}h11]{loh2011}
C.~L{\"{o}}h.
\newblock Simplicial volume.
\newblock {\em Bulletin of the Manifold Atlas}, 2011.
\newblock http://www.boma.mpim-bonn.mpg.de/data/29print.pdf.

\bibitem[LP19]{lamy2019acylindrical}
S.~Lamy and P.~Przytycki.
\newblock Acylindrical hyperbolicity of the three-dimensional tame automorphism
  group.
\newblock {\em Ann. Sci. \'{E}c. Norm. Sup\'{e}r. (4)}, 52(2):367--392, 2019.

\bibitem[LS01]{lyndon2015combinatorial}
{R}. Lyndon and {P}. {S}chupp.
\newblock {\em Combinatorial group theory}.
\newblock Classics in Mathematics. Springer-Verlag, Berlin, 2001.
\newblock Reprint of the 1977 edition.

\bibitem[Lyn48]{lyndon1948cohomology}
R.~Lyndon.
\newblock The cohomology theory of group extensions.
\newblock {\em Duke Math. J.}, 15:271--292, 1948.

\bibitem[Lyn50]{lyndon1950cohomology}
{R}. Lyndon.
\newblock Cohomology theory of groups with a single defining relation.
\newblock {\em Ann. of Math. (2)}, 52:650--665, 1950.

\bibitem[MM99]{masur1999geometry}
H.~Masur and Y.~Minsky.
\newblock Geometry of the complex of curves. {I}. {H}yperbolicity.
\newblock {\em Invent. Math.}, 138(1):103--149, 1999.

\bibitem[MS19]{maher2019random}
J.~Maher and A.~Sisto.
\newblock Random subgroups of acylindrically hyperbolic groups and hyperbolic
  embeddings.
\newblock {\em Int. Math. Res. Not. IMRN}, (13):3941--3980, 2019.

\bibitem[MY07]{minyam2007}
I.~Mineyev and A.~Yaman.
\newblock Relative hyperbolicity and bounded cohomology.
\newblock 2007.
\newblock preprint, http://www.math.uiuc.edu/~mineyev/math/art/rel-hyp.pdf.

\bibitem[Oll04]{ollivier2004sharp}
Y.~Ollivier.
\newblock Sharp phase transition theorems for hyperbolicity of random groups.
\newblock {\em Geom. Funct. Anal.}, 14(3):595--679, 2004.

\bibitem[Oll05]{ollivier2005invitation}
Y.~Ollivier.
\newblock {\em A {J}anuary 2005 invitation to random groups}, volume~10 of {\em
  Ensaios Matem\'{a}ticos [Mathematical Surveys]}.
\newblock Sociedade Brasileira de Matem\'{a}tica, Rio de Janeiro, 2005.

\bibitem[Ols95]{Olsh}
A.~Yu. Olshanskii.
\newblock {${\rm SQ}$}-universality of hyperbolic groups.
\newblock {\em Mat. Sb.}, 186(8):119--132, 1995.

\bibitem[Osi06]{osin2006relatively}
{D}. Osin.
\newblock Relatively hyperbolic groups: intrinsic geometry, algebraic
  properties, and algorithmic problems.
\newblock {\em Mem. Amer. Math. Soc.}, 179(843):vi+100, 2006.

\bibitem[Osi07]{osin2007peripheral}
{D}. Osin.
\newblock Peripheral fillings of relatively hyperbolic groups.
\newblock {\em Invent. Math.}, 167(2):295--326, 2007.

\bibitem[Osi08]{osin2008questions}
D.~Osin.
\newblock Questions on relatively hyperbolic groups and related classes.
\newblock 2008.

\bibitem[Osi16]{osin2016acylindrically}
{D}. Osin.
\newblock Acylindrically hyperbolic groups.
\newblock {\em Trans. Amer. Math. Soc.}, 368(2):851--888, 2016.

\bibitem[Osi18]{osin2017groups}
{D}. Osin.
\newblock Groups acting acylindrically on hyperbolic spaces.
\newblock {\em Proc. Int. Cong. of Math.– 2018, Rio deJaneiro}, 1:915--936,
  2018.

\bibitem[Osi21]{osin2020topological}
D.~Osin.
\newblock A topological zero-one law and elementary equivalence of finitely
  generated groups.
\newblock {\em Ann. Pure Appl. Logic}, 172(3):102915, 2021.

\bibitem[Ruh82]{ruh1982}
E.~A. Ruh.
\newblock Almost flat manifolds.
\newblock {\em J. Differential Geometry}, 17(1):1--14, 1982.

\bibitem[Sun19]{sun2017dynamical}
B.~Sun.
\newblock A dynamical characterization of acylindrically hyperbolic groups.
\newblock {\em Algebr. Geom. Topol.}, 19(4):1711--1745, 2019.

\bibitem[Sun20]{sun2018cohomologyi}
B.~Sun.
\newblock Cohomology of group theoretic {D}ehn fillings {I}: {C}ohen-{L}yndon
  type theorems.
\newblock {\em J. Algebra}, 542:277--307, 2020.

\bibitem[Swa69]{swan1969groups}
{R}. Swan.
\newblock Groups of cohomological dimension one.
\newblock {\em J. Algebra}, 12:585--610, 1969.

\bibitem[Thu82]{thurston1983three}
{W}. Thurston.
\newblock Three-dimensional manifolds, {K}leinian groups and hyperbolic
  geometry.
\newblock {\em Bull. Amer. Math. Soc. (N.S.)}, 6(3):357--381, 1982.

\bibitem[Wal79]{wall1979homological}
C.~T.~C. Wall, editor.
\newblock {\em Homological group theory}, volume~36 of {\em London Mathematical
  Society Lecture Note Series}. Cambridge University Press, Cambridge-New York,
  1979.

\bibitem[Wan21]{wang2018spectral}
O.~H. Wang.
\newblock A spectral sequence for {D}ehn fillings.
\newblock {\em Algebr. Geom. Topol.}, 21(5):2257--2272, 2021.

\bibitem[Wei94]{Weibel}
C.~A. Weibel.
\newblock {\em An introduction to homological algebra}, volume~38 of {\em
  Cambridge Studies in Advanced Mathematics}.
\newblock Cambridge University Press, Cambridge, 1994.

\bibitem[\.{Z}03]{zuk2003property}
A.~\.{Z}uk.
\newblock Property ({T}) and {K}azhdan constants for discrete groups.
\newblock {\em Geom. Funct. Anal.}, 13(3):643--670, 2003.

\end{thebibliography}

\end{document}